\documentclass[11pt]{amsart}
\usepackage[margin=3cm]{geometry}               
\makeatletter
\@namedef{subjclassname@2020}{%
  \textup{2020} Mathematics Subject Classification}
\makeatother
\usepackage[usenames,dvipsnames,svgnames,table]{xcolor}    
\usepackage{graphicx}
\usepackage{tikz}
\usepackage{tikz-cd}
\usepackage{mathrsfs}
\usetikzlibrary{decorations.pathreplacing}
\usepackage{lineno, hyperref}
\usepackage{amssymb,amsfonts, amsmath, amsthm}
\usepackage{epstopdf}
\usepackage[shortlabels]{enumitem}
\usepackage{ulem}
\usepackage{cancel}
\usepackage{thmtools} 
\usepackage{mathtools} 
\usepackage{mathrsfs}
\usepackage{yhmath}
\usepackage{amsmath}
\usepackage[english]{babel}
\usepackage[toc,page]{appendix}
\usepackage{float}
\restylefloat{table}
 \usepackage{hyperref}
  \hypersetup{
colorlinks=true,
linkcolor=cyan,
citecolor=cyan
}
\usepackage{mathrsfs}
\usepackage{amsbsy}

\swapnumbers
\newtheorem*{theorem*}{Theorem}
\newtheorem*{question*}{Question}
\newtheorem*{corollary*}{Corollary}
\newtheorem*{theoremA}{Theorem A}
\newtheorem*{theoremB}{Theorem B}
\newtheorem*{theoremC}{Theorem C}

\newtheorem*{main-theorem}{Main Theorem}

\newtheorem{theorem}[subsubsection]{Theorem}
\newtheorem{lemma}[theorem]{Lemma}
\newtheorem{fact}[theorem]{Fact}
\newtheorem{claim}[theorem]{Claim}
\newtheorem{proposition}[theorem]{Proposition}
\newtheorem{corollary}[theorem]{Corollary}

\theoremstyle{definition}
\newtheorem{definition}[theorem]{Definition}

\newtheorem{convention}[theorem]{Convention}
\newtheorem{assumption}[theorem]{Section assumption}
\newtheorem{notation}[theorem]{Notation}

\newtheorem{remark}[theorem]{Remark}
\newtheorem{lem-def}[theorem]{Lemma-definition}

\newtheorem{examples}[theorem]{Examples}

\newtheorem{question}[theorem]{Question}

\newcommand{\R}{\mathbb R}
\newcommand{\Q}{\mathbb Q}
\newcommand{\N}{\mathbb N}
\newcommand{\Z}{\mathbb Z}

\newcommand{\cF}{\mathscr F}
\newcommand{\cN}{\mathcal N}
\newcommand{\cM}{\mathcal M}
\newcommand{\cU}{\mathcal U}

\newcommand{\OO}{\mathcal{O}}
\newcommand{\cL}{\mathcal{L}}

\newcommand{\cA}{\mathcal{A}}
\newcommand{\cB}{\mathcal{B}}
\newcommand{\cK}{\mathcal{K}}
\newcommand{\res}{\mathrm{res}}
\newcommand{\Div}{\mathrm{div}}
\newcommand{\Def}{\mathrm{def}}

\newcommand{\eq}{\mathrm{eq}}
\newcommand{\dcl}{\mathrm{dcl}}

\newcommand{\acl}{\mathrm{acl}}
\newcommand{\RCF}{\mathrm{RCF}}
\newcommand{\ACF}{\mathrm{ACF}}
\newcommand{\ACVF}{\mathrm{ACVF}}
\newcommand{\DOAG}{\mathrm{DOAG}}
\newcommand{\RCVF}{\mathrm{RCVF}}
\newcommand{\PCF}{p\mathrm{CF}}

\newcommand{\Th}{\mathrm{Th}}

\newcommand{\bdd}{\mathrm{bdd}}
\newcommand{\conv}{\mathrm{conv}}

\newcommand{\D}[1]{{S^\mathrm{def}_{#1}}}
\newcommand{\tp}{\mathrm{tp}}

\newcommand{\id}{\mathrm{id}}

\newcommand{\SE}{\mathcal{S}\mathcal{E}}

\newcommand{\rv}{\mathrm{rv}}
\newcommand{\RV}{\mathbf{RV}}
\newcommand{\RVV}{\mathrm{RV}}
\newcommand{\BP}{\mathrm{bp}}
\newcommand{\triv}{\mathrm{triv}}
\newcommand{\Aut}{\mathrm{Aut}}
\newcommand{\vd}{\mathrm{vd}}
\newcommand{\ld}{\mathrm{ld}}
\newcommand{\alg}{\mathrm{alg}}
\newcommand{\qftp}{\mathrm{qftp}}

\newcommand{\Ring}{\mathrm{ring}}
\newcommand{\og}{\mathrm{og}}

\newcommand{\ac}{\mathrm{ac}}
\newcommand{\Pas}{\mathrm{Pas}}
\newcommand{\Res}{\mathrm{Res}}
\newcommand{\VF}{\mathbf{VF}}
\newcommand{\RF}{\mathbf{k}}
\newcommand{\VG}{\mathbf{\Gamma}}
\newcommand{\VGi}{\mathbf{\Gamma_\infty}}
\newcommand{\cLr}{\cL_{\mathrm{ring}}}
\newcommand{\cLog}{\cL_{\mathrm{og}}}
\newcommand{\cLovf}{\cL_{\mathrm{ovf}}}
\newcommand{\cLpcf}{\cL_{\mathrm{pcf}}}
\newcommand{\cLpres}{\cL_{\mathrm{Pres}}}
\newcommand{\cLmac}{\cL_{\mathrm{Mac}}}
\newcommand{\cLgk}{\cL_{\Gamma k}}

\newcommand{\cLrvg}{\cL_{\mathrm{RV}, \Gamma}}

\newcommand{\bA}{\mathbf{A}}
\newcommand{\bB}{\mathbf{B}}
\newcommand{\bC}{\mathbf{C}}
\DeclareMathAlphabet{\mathscrbf}{OMS}{mdugm}{b}{n}
\newcommand{\bbA}{\pmb{\mathscrbf{A}}}

\DeclareMathAlphabet{\mathscrbf}{OMS}{mdugm}{b}{n}

\newcommand{\sepa}{\text{\tiny$\mathcal{S}\mathcal{E}$}}

\newcommand{\forkindep}[1][]{%
  \mathrel{
    \mathop{
      \vcenter{
        \hbox{\oalign{\noalign{\kern-.3ex}\hfil$\vert$\hfil\cr
              \noalign{\kern-.7ex}
              $\smile$\cr\noalign{\kern-.3ex}}}
      }
    }\displaylimits_{#1}
  }
}

\def\Ind#1#2{#1\setbox0=\hbox{$#1x$}\kern\wd0\hbox to 0pt{\hss$#1\mid$\hss}
\lower.9\ht0\hbox to 0pt{\hss$#1\smile$\hss}\kern\wd0}
\def\ind{\mathop{\mathpalette\Ind{}}}
\def\notind#1#2{#1\setbox0=\hbox{$#1x$}\kern\wd0\hbox to 0pt{\mathchardef
\nn=12854\hss$#1\nn$\kern1.4\wd0\hss}\hbox to
0pt{\hss$#1\mid$\hss}\lower.9\ht0 \hbox to
0pt{\hss$#1\smile$\hss}\kern\wd0}

\makeatletter
\def\widebreve{\mathpalette\wide@breve}
\def\wide@breve#1#2{\sbox\z@{$#1#2$}%
     \mathop{\vbox{\m@th\ialign{##\crcr
\kern0.08em\brevefill#1{0.8\wd\z@}\crcr\noalign{\nointerlineskip}%
                    $\hss#1#2\hss$\crcr}}}\limits}
\def\brevefill#1#2{$\m@th\sbox\tw@{$#1($}%
  \hss\resizebox{#2}{\wd\tw@}{\rotatebox[origin=c]{90}{\upshape(}}\hss$}
\makeatletter

\title{Beautiful pairs}

\author[Pablo Cubides Kovascics]{Pablo Cubides Kovacsics}
\address{Departamento de Matemáticas, Universidad de los Andes, 
Carrera 1 \# 18A - 12, 
111711, Bogot\'a, Colombia}
\email{p.cubideskovacsics@uniandes.edu.co}
\thanks{}

\author[Martin Hils]{Martin Hils}
\address{Institut f\"{u}r Mathematische Logik und Grundlagenforschung, Universit\"{a}t M\"{u}nster, Einsteinstr. 62, D-48149 M\"{u}nster, Germany}
\email{hils@uni-muenster.de}
\thanks{}

\author[Jinhe Ye]{Jinhe Ye}
\address{School of Mathematics, Nanjing University, Nanjing, Jiangsu, China, 210093}
\email{jinhe.ye@nju.edu.cn}
\thanks{PCK was funded by the German Research Foundation (DFG) via the individual research grant No.~426488848 \textit{Archimedische und nicht-archimedische Stratifizierungen höherer Ordnung}.\\
\indent MH was partially supported by the German Research Foundation (DFG) via HI 2004/1-1 (part of the French-German ANR-DFG project GeoMod) and under Germany's Excellence Strategy EXC 2044-390685587, \textit{Mathematics M\"unster: Dynamics-Geometry-Structure}.\\ 
\indent JY was partially supported by GeoMod AAPG2019 (ANR-DFG), \textit{Geometric and
Combinatorial Configurations in Model Theory}. \\
\indent This research was supported through the programme \textit{Research in Pairs} by the Mathematisches Forschungsinstitut Oberwolfach in March 2021. We thank them for their hospitality.
}

 \dedicatory{Dedicated to Bruno Poizat on the occasion
     of his $75$th birthday}



\begin{document}
\begin{abstract} We introduce an abstract framework to study certain classes of stably embedded pairs of models of a complete $\cL$-theory $T$, called \textit{beautiful pairs}, which comprises Poizat's belles paires of stable structures and van den Dries-Lewenberg's tame pairs of o-minimal structures. Using an amalgamation construction, we relate several properties of beautiful pairs with properties analogous to properties in Fra\"{i}ss\'e classes. 

After characterizing beautiful pairs of various theories of ordered abelian groups and valued fields, including the theories of algebraically closed, $p$-adically closed and real closed valued fields, we show an Ax-Kochen-Ershov type result for beautiful pairs of henselian valued fields. 
As an application, we derive strict pro-definability of particular classes of definable types. When $T$ is one of the theories of valued fields mentioned above, the corresponding classes of types are related to classical geometric spaces and our main result specializes to their strict pro-definability. Most notably, we exhibit the strict pro-definability of a natural space of types associated to Huber's analytification. In this way, we also recover a result of Hrushovski-Loeser on the strict pro-definability of stably dominated types in algebraically closed valued fields, which corresponds to Berkovich's analytification.
\end{abstract}
\maketitle

\setcounter{secnumdepth}{4}

\setcounter{tocdepth}{1}
{
  \hypersetup{linkcolor=black}
  \tableofcontents
}


\newcounter{eqn}

\normalem

\section{Introduction}
In their seminal work \cite{HL}, E. Hrushovski and F. Loeser presented a novel perspective on non-archimedean semi-algebraic geometry, with far-reaching implications for the topology of semi-algebraic subsets of Berkovich analytifications of algebraic varieties. Their approach introduces certain spaces of definable types as a model-theoretic counterpart of the analytification of an algebraic variety, which they call its \emph{stable completion}. These spaces possess the remarkable property of being a \emph{strict pro-definable set}, i.e., a projective limit of definable sets with surjective transition maps. This structural property endows the spaces with a definable structure that is in strong resemblance with familiar definable sets and enables Hrushovski and Loeser to leverage several model-theoretic tools, leading to the aforementioned applications.

Is there also a model-theoretic analogue of Huber's adic spaces \cite{Adic}? More specifically, is there a strict pro-definable structure on such spaces? In \cite{cubi-ye}, the first and third authors provided a partial positive answer to this question by showing that the natural candidate for such a space has the structure of a pro-definable set. However, the crucial issue of \emph{strict} pro-definability, i.e., the surjectivity of transition maps, proved to be significantly more intricate and remained unresolved. Thanks to the formalism presented in this paper, we can now provide a complete positive answer not only for Huber's analytification of an algebraic variety but also for other notable geometric spaces such as the real analytification of semi-algebraic sets as defined in \cite{jell_etal}. In other words, our work lays the foundations for a model-theoretic approach to non-archimedean geometry in various contexts.

Besides the geometric motivation, the formalism introduced here also holds significant model-theoretic importance. Rather than studying spaces of definable types directly, we instead investigate stably embedded pairs of models of a given theory $T$, following the approach outlined in \cite{cubi-ye}. The study of pairs of models of a complete theory is a well-established topic in model theory, dating back to early results of A. Robinson \cite{Robinson} on pairs of algebraically closed fields. B. Poizat \cite{poizat} later introduced the notion of \emph{belles paires} of models of a complete stable theory and extended Robinson's results to this context. Y. Baisalov and Poizat \cite{BaPo98} subsequently built on the work of D. Marker and C. Steinhorn \cite{MS} and A. Pillay \cite{Pillay_def} to develop similar ideas for stably embedded pairs of o-minimal structures, i.e., elementary pairs $(N, M)$ in which the trace in $M$ of every $N$-definable set is again $M$-definable. The formalism presented here unifies the stable and o-minimal settings into a more general context of stably embedded pairs of models of a complete theory $T$. In keeping with Poizat's terminology, we refer to the pairs under consideration as \emph{beautiful pairs}.

The following subsection provides a brief summary of the paper's main results.

\subsection{Summary of main results and new concepts}

The main theorem of the present article is the following (see later Theorem \ref{cor:acvf-strictpro}, Theorem \ref{cor:rcvf-strictpro} and Theorem \ref{thm:pcf-strictpro}). 

\begin{main-theorem} The model-theoretic spaces of types which are counterparts of the following spaces are strict pro-definable:
\begin{enumerate}[label=(\arabic*)] 
    \item the analytification of an algebraic variety in the sense of Huber \cite{Adic}, 
    \item the real analytification of a (real) semi-algebraic set in the sense of \cite{jell_etal},
    \item the Zariski-Riemann space of an algebraic variety~\cite{Zariski-Riemann},  
    \item the $p$-adic spectrum of a $p$-adic algebraic variety in the sense of \cite{ERob-padic}. 
\end{enumerate}
\end{main-theorem}

The precise meaning of the model-theoretic counterpart is in complete analogy to~\cite[Chapter 14]{HL}. To illustrate, in (1) above, when $K$ is an algebraically closed maximally complete  non-archimedean field with value group $\R$ and $V$ is an algebraic variety over $K$, then the space of bounded types concentrating on $V$ and definable over $K$, denoted $\widetilde{V}(K)$, corresponds precisely to Huber's adic space of $V$. 

The proof of the Main Theorem uses properties of beautiful pairs of certain valued fields. We establish an Ax-Kochen-Ershov principle for such pairs in particular. However, the general machinery of beautiful pairs turns out to work in a completely abstract framework, encompassing not only valued fields but also other interesting structures and their respective theories. As a result, this abstract formalism yields numerous significant findings about the corresponding first-order structures. We will now describe our technical setup for beautiful pairs in complete generality.

Given a complete first-order $\cL$-theory $T$, we distinguish collections of well-behaved classes of stably embedded pairs of $\cL$-structures (\sepa-pairs for short, see Definition \ref{def:se-pairs}), called \emph{natural classes} (Definition \ref{def:natural-class-K}) and, given such a natural class $\cK$ and an infinite cardinal $\lambda$, we define the notion of \emph{$\lambda$-$\cK$-beautiful pairs} as those pairs which are $\lambda$-generic (Definition \ref{def:bpair}). We say a pair is \emph{$\cK$-beautiful} if it is $\lambda$-$\cK$-beautiful for some $\lambda>|T|$.   To have an intuition, the reader may think of natural classes and $\lambda$-$\cK$-beautiful pairs in analogy to Fra\"iss\'e classes and Fra\"iss\'e limits, respectively. Poizat's treatment of beautiful pairs of models of a stable theory transfers surprisingly smoothly to this more general context. However, some new phenomena arise from the fact that in an unstable theory not all types are definable. 
In a way, our framework generalizes the stable case in a direction which is orthogonal to the generalization to so-called ``lovely pairs'', introduced by I. Ben-Yaacov, A.  Pillay and E. Vassiliev in \cite{BYPV}. 

The following summarizes some of the main results about beautiful pairs (see later Theorem \ref{thm:charct-beaut-pairs}, Corollary \ref{cor:elem-equiv} and Theorem \ref{thm:qe}) in this general context. Recall that for an infinite cardinal $\lambda$, an $\cL_{\infty,\lambda}$-sentence is a sentence using symbols in $\cL$ with arbitrary conjunctions and disjunctions and fewer than $\lambda$-many quantifiers. Moreover, $\cM\equiv_{\infty,\lambda}\cN$ means $\cM$ and $\cN$ satisfy the same $\cL_{\infty,\lambda}$-sentences. Equivalently, $\cM\equiv_{\infty,\lambda}\cN$ means that the system of partial elementary maps between substructures of $\cM$ and $\cN$ of size $<\lambda$ is non-empty and has the back-and-forth property.

\medskip

\begin{theoremA} Let $\cK$ be a natural class. Then the following holds.
\begin{enumerate}[label=(\arabic*)]
    \item $\cK$ has the amalgamation property\footnote{Recall that $\cK$ has the amalgamation property if for any diagram $B\leftarrow A\rightarrow C$ in $\cK$, there are $B\rightarrow D\leftarrow C$ in $\cK$ making the diagram commute.} if and only if $\lambda$-$\cK$-beautiful pairs exist for all $\lambda\geqslant |T|^+$. Moreover, in this case, $\cK$ has the extension property (see Definition~\ref{def:ext_prop}) if and only if all beautiful pairs are elementary pairs of models of $T$.
    \item $\cM\equiv_{\infty,\lambda} \cN$ for any two $\lambda$-$\cK$-beautiful pairs $\cM$ and $\cN$.   Therefore, all $\cK$-beautiful pairs are elementarily equivalent. 
    \item If there is a $\lambda$-$\cK$-beautiful pair which is $\lambda$-saturated for  $\lambda\geqslant |T|^+$ (we will say in this case that $\cK$ has beauty transfer), then the common theory of $\cK$-beautiful pairs $T_\BP(\cK)$ admits quantifier elimination (in the language of beautiful pairs, see Definition \ref{def:Lbp}), and the predicate $P$ is stably embedded and pure (i.e., there is no new induced structure). 
    \end{enumerate}
\end{theoremA}

 Dually, one may represent natural classes in terms of global definable types (see Definition \ref{def:natuclass-F} and Lemma~\ref{lem:nat-to-nat}), and the following result is the desired link with strict pro-definability of spaces of definable types. 

\begin{theoremB}[Later Theorem \ref{thm:strictness}] Let $\cF$ be a natural subclass of definable types and let $\cK_\cF$ be its associated natural class. Suppose that $\cK_\cF$-beautiful pairs exist and $\cK_\cF$ has beauty transfer. Then $\cF$ is strict pro-definable.  \end{theoremB}

It is from Theorem B that the Main Theorem is derived. We also recover (Theorem \ref{cor:acvf-strictpro}) the strict pro-definability result of Hrushovski and Loeser \cite[Theorem 3.1.1]{HL} mentioned above. 

Along the way, we characterize all completions of stably embedded pairs of models of $T$ when $T$ is one of the following theories: the theory of divisible ordered abelian groups (Theorem \ref{thm:DOAG-complete}, Presburger arithmetic (Theorem~\ref{thm:pres-complete}), and any completion of algebraically (resp. real, resp. $p$-adically) closed valued fields (Corollary~\ref{cor:all-completions-acvf}, Corollary~\ref{rem:all-completions-rcvf}, and Theorem~\ref{thm:pcf-Beauty}). Moreover, we study beautiful pairs of certain henselian valued fields (called \emph{benign}) by considering their  $\RV$-theory, or even the theories of their residue field and value group. One of our main results consists in showing the following ``Ax-Kochen-Ershov principle for beauty'' (see also Theorem \ref{thm:RV-pairs}). For a benign valued field $(K,v)$ (see Definition \ref{def:benign}) with value group $\Gamma$ and residue field $k$, we consider definitional $\cL_\RVV$-expansions with an additional set of sorts $\bbA$ for $k^*/(k^*)^n$ for all $n\geqslant 0$ and $\VG$ for the value group (with the natural quotient maps). 
\begin{theoremC}[Later Theorem \ref{thm:red-to-k-gamma}]
Let $(K,v)$ be a benign valued field and let $T$ be its theory. Given natural classes $\cK_{\mathscr{A}}$ and $\cK_\Gamma$ of \sepa-pairs in the theories of the residue field $k$ (in the sorts $\bbA$) and of the value group $\Gamma$ (in the sort $\VG$), respectively, assume both $\cK_{\mathscr{A}}$-beautiful pairs and $\cK_\Gamma$-beautiful pairs exist, are elementary and satisfy beauty transfer. Let $\cK$ be the class of $\cL_P$-structures of valued fields $\cM\in \cK_\Def$ induced by $\cK_{\mathscr{A}}$ and $\cK_\Gamma$. Then, $\cK$-beautiful pairs exist. Moreover, $\cK$ has beauty transfer and $T_\BP(\cK)$ is axiomatized by the following conditions on an $\cL_P$-structure $\cM=(M,P(\cM))$:
\begin{itemize}
    \item $\VF(M)/\VF(P(\cM))$ is vs-defectless\footnote{See page~\pageref{VS-def} for the definition of vs-defectless.};
    \item $P(\cM)\preccurlyeq M\models T$;
    \item $(\bbA(M),\bbA(P(\cM)))\models T_\BP(\cK_{\mathscr{A}})$ and $(\VG(M),\VG(P(\cM)))\models T_\BP(\cK_\Gamma)$.
\end{itemize}
\end{theoremC}

Part of the proof consists in showing a similar reduction result for short exact sequences in the framework recently investigated by M. Aschenbrenner, A. Chernikov, A. Gehret and M. Ziegler \cite{AsChGeZi20} (see Theorem \ref{thm:exact}).

\subsection*{Structure of the paper}

The paper is organized as follows. In Section \ref{sec:BP-pairs-main} we introduce the abstract framework of beautiful pairs and prove some of their main properties. Moreover, we exhibit the relation with spaces of definable types and strict pro-definability. Section~\ref{sec:stable-omin} is devoted to showing how our abstract framework is related to previous work, both for stable and o-minimal theories. In most of the remaining sections we study particular theories and their corresponding beautiful pairs. Section~\ref{sec:bp-oag} focuses on ordered abelian groups. Beautiful pairs of short exact sequences are studied in Section~\ref{sec:SES}. In Section~\ref{sec:domination} we gather various domination results in valued fields. These results are later used in Section~\ref{sec:BP-val-fields} and Section~\ref{sec:Ax-Kochen-RV} to study beautiful pairs of various theories of henselian valued fields and prove the Main Theorem.  
\subsection*{Acknowledgement}We would like to thank Anand Pillay for directing us to the article by Y. Baisalov and B. Poizat \cite{BaPo98}. We would also like to thank Zhengqing He, Zixuan Zhu, Tingxiang Zou and in particular the anonymous referees for helpful comments on earlier versions of the paper.

\section{General theory of beautiful pairs}\label{sec:BP-pairs-main}

\subsection{Preliminaries and notation}\label{sec:notation}

Let $\cL$ be a possibly multi-sorted language, $T$ be a complete $\cL$-theory and let $\cU$ denote a universal domain (monster model) of~$T$. The sorts in~$\cL$ are called the \emph{real sorts}, while \emph{imaginary sorts} are~sorts in $\cL^\eq$. Given a subset $A\subseteq \cU$ we let $\langle A\rangle_\cL$ denote the $\cL$-substructure of $\cU$ generated by $A$ and often omit the subscript $\cL$ when it is clear from the context.  

Recall that given a subset $A\subseteq M\models T$, a type $p\in S_x(M)$ is \emph{$A$-definable} (or definable over $A$) if for every $\cL$-formula $\varphi(x,y)$ there is an $\cL(A)$-formula $d_p\varphi(y)$ such that for every $c\in M^y$
\[
\varphi(x,c)\in p(x) \Leftrightarrow \ M\models d_p\varphi(c). 
\]
The map $\varphi(x,y)\mapsto d_p\varphi(y)$ is called a \emph{scheme of definition for $p$}, and the formula $d_p\varphi(y)$ is called a \emph{$\varphi$-definition for $p$}. We say $p\in S_x(M)$ is \emph{definable} if it is $M$-definable. Given any set $B$ containing $M$, we use $p|B$ to denote the type $\{\varphi(x,b)\mid b\in B^{y}\text{ such that } \cU\models d_p\varphi(b)\}$. We let $S_x^\Def(A)$ denote the subset of $S_x(\cU)$ of $A$-definable types, and for an $A$-definable set $X$, we let $S_X^\Def(A)$ be the set of global $A$-definable types concentrating on $X$. We refer the reader to \cite[Section 1]{pillay83} for proofs and basic facts on definable types. 

\begin{definition}\label{def:uddt}
We say that~$T$ has \emph{uniform definability of definable types (over models)}, or $\mathrm{UDDT}$ in short, if for every $\cL$-formula $\varphi(x, y)$ there is an $\cL$-formula $\psi(y,z)$ such that for every model $M$ of $T$ and every definable type $p\in S_x(M)$ there is a $c\in M^z$ such that $\psi(y,c)$ is a $\varphi$-definition for $p$.  
\end{definition}

\begin{remark} An alternative definition of uniform definability of definable types would be to require a uniform scheme as in the previous definition for definable $\varphi$-types instead of definable types in $S_x(M)$. We will not consider this variant in the present article.   
\end{remark}

Unless otherwise stated, we assume in what follows that $T$ is complete, and $T$ and $T^\eq$ have quantifier elimination.

\subsection{Stably embedded pairs}\label{sec:pairal} 

We let $\cL_P$ denote the language of pairs of $\cL$-structures.

\begin{definition}\label{def:se-pairs} An $\cL_P$-structure $\cA=(A,P(\cA))$ is called a \emph{stably embedded pair}, in short, \emph{\sepa-pair}, if $P(\cA)\models T$, $P(\cA)\subseteq A\models T_{\forall}$ and $\tp_{\cL}(A/P(\cA))$ is definable. We use $\cK_\Def$ to denote the class of \sepa-pairs.
\end{definition}

\begin{remark}\label{rmk:UDDT}
By a straightforward compactness argument, a theory $T$ has $\mathrm{UDDT}$ if and only if $\cK_\Def$ is an elementary class.
\end{remark}

\begin{definition}
    Assuming that the class of all stably embedded elementary pairs of models of $T$ as $\cL_P$-structures (i.e., \sepa-pairs $\cM=(M,P(\cM))$ with $P(\cM)\preccurlyeq M\models T$) is $\cL_P$-elementary, we use $T_{\SE}$ to denote their common theory.
\end{definition}

\begin{notation}\label{notation:triv} \
\begin{itemize}
    \item Given a model $A\models T$, we associate to $A$ the \sepa-pair $(A,A)$ which we denote by $A_\triv$, the \emph{trivial \sepa-pair associated to $A$}.  
  \item   Let $\cA$  and $\cB$ be \sepa-pairs with $\cA\subseteq_{\cL_P}\cB$, i.e., with $\cA$ an $\cL_P$-substructure of $\cB$. We say $\cA$ is a \emph{bp-substructure} of $\cB$, and we write $\cA\subseteq_{\BP}\cB$, if $\tp_{\cL}(A/P(\cB))$ is $P(\cA)$-definable. We denote this definable extension of $\tp_{\cL}(A/P(\cA))$ to $P(\cB)$ by $\tp_{\cL}(A/P(\cA))|P(\cB)$. By definition, a \emph{bp-embedding} is an isomorphism onto a bp-substructure.
\end{itemize}
\end{notation}

The following lemma follows directly from the definition of \sepa-pairs.   

\begin{lem-def}[Base extension]\label{lem:base-change} Let $\cA$ be an \sepa-pair and $B$ be an $\cL$-elementary extension of $P(\cA)$. Then there is a unique amalgam $\cA_B$ of \sepa-pairs with bp-embeddings
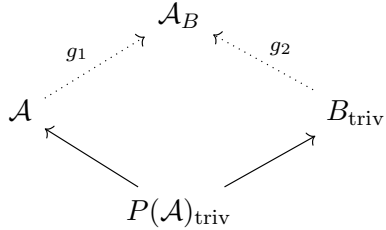
\[
\begin{tikzcd}
             & \cA_B                       &              \\
\cA \arrow[ru,dotted,"g_1"] &                     & B_\triv \arrow[lu,dotted,"g_2"'] \\
             & P(\cA)_\triv \arrow[lu] \arrow[ru] &             
\end{tikzcd}
\]
such that $\mathcal{A}_B=(\langle g_1(A)\cup g_2(B)\rangle_{\cL},B)$. We call the structure $\cA_B$ \emph{the base extension of $\cA$ to $B$}. \qed 
\end{lem-def}

\begin{definition}\label{def:natural-class-K}
Let $\cK$ be a subclass of $\cK_{\mathrm{def}}$. We say $\cK$ is a \emph{natural} class if 
\begin{enumerate}
    \item $\cK$ is closed under isomorphism;
    \item $A_\triv\in \cK$ for any small $A\preccurlyeq \cU$; 
    \item $\cK$ is closed under $\dcl_{\cL}$, i.e., if $(A,P(\cA))\in\cK$, then $(\dcl_{\cL}(A),P(\cA))\in\cK$;
    \item $\cA\in \cK$ if and only if for any $\cA_0\subseteq \cA$ that contains $P(\cA)$ and is finitely generated over $P(\cA)$ as an $\cL$-structure, one has that $\cA_0 \in \cK$;
    \item $\cK$ is closed under base extension by any small $B\preccurlyeq\cU$;
    \item $\cK$ is closed under bp-substructures.
\end{enumerate}
For $\lambda\geqslant|T|^+$ we denote by $\cK_{<\lambda}$ the subclass of $\cK$ of elements of cardinality strictly less than $\lambda$. Abusing notation, we will use $\cK$ to denote both the class of structures or the category whose objects are the structures from $\cK$ and whose morphisms are the $\BP$-embeddings.
\end{definition}

Note that $\cK_\Def$ is a natural class. In addition, by conditions (iv) and (v) above, natural classes are closed under unions of chains of bp-embeddings.

\subsection{Natural classes and definable types}\label{sec:def-type}

In practice, we think of a natural class as induced by a given class of definable types. Let $\cF$ denote a class of global definable types. For $\cM\preccurlyeq \cU$ and a tuple of variables $x$ (possibly infinite), we let $\cF_x(M)$ denote the subset of $\cF$ of $M$-definable types in variables $x$. For $X$ an $M$-definable set, we let $\cF_X$ be the subclass of $\cF$ of types concentrating on $X$, and more generally, for $C\subseteq \cU^{eq}$ containing $M$, we let $\cF_X(C)$ be the subset of $\cF$ of $C$-definable types concentrating on $X$.

\begin{definition}\label{def:natuclass-F} A non-empty class $\cF$ of global definable types is called \emph{natural} if $\cF$ satisfies the following properties for every $\cM\preccurlyeq \cU$:
\begin{itemize}[leftmargin=*]
    \item (Finitary) For any tuple $x$ of variables, $p\in \cF_x(M)$ if and only if $p|_{x'}\in \cF_{x'}(M)$ for any $x'\subseteq x$ and $x'$ finite.
    \item (Invariance) Given $\sigma\in \Aut(\cU)$, $\sigma(\cF(M))=\cF(\sigma(M))$. 
    \item (Push forward) $\cF(M)$ is closed under push forwards by $M$-definable functions.
\end{itemize}
\end{definition}

\begin{definition} \
\begin{itemize}[leftmargin=*]
    \item Let $\cF$ be a natural class of global definable types. We define the class of \sepa-pairs $\cK_\cF$ associated to $\cF$ as follows: for an \sepa-pair $\cA$,
$\cA\in \cK_\cF$ if and only if $\tp_{\cL}(A/P(\cA))\in\cF(P(\cA))$. 
\item Let $\cK$ be a natural class. Let $\cF_\cK$ be the following class of global definable types. A global type $p(x)$ is in $\cF$ if there is $\cA\in \cK$ and $a\in A^x$ such that $p=\tp(a/P(\cA))|\cU$. 
\end{itemize} 
\end{definition}

The following lemma follows easily from the definitions. Its proof is left to the reader.  

\begin{lemma}\label{lem:nat-to-nat} If $\cF$ is a natural class of global definable types, then $\cK_\cF$ is a natural class. Conversely, if $\cK$ is a natural class, then $\cF_\cK$ is a natural class of global definable types. In addition, the functions $\cF\mapsto \cK_\cF$ and $\cK \mapsto \cF_\cK$ are inverses to each other.  \qed
\end{lemma}

\begin{examples}\label{eg:class} The following are major examples that will be considered in the paper.
\begin{enumerate}[leftmargin=*,label=(\arabic*)]
    \item The class $\cF=\D{}(\cU)$ is natural and the corresponding natural class is the class $\cK_\Def$ of all \sepa-pairs.  
    \item The subclass of definable types which are orthogonal to a given sort (in the sense of \cite[Section 2.5]{HL}), or to a given invariant type, is natural.
    \item The subclass of generically stable types/stably dominated types is natural. 
\end{enumerate}
\end{examples}

\subsection{Beautiful pairs}\label{sec:b-pairs}

\begin{definition}[Beautiful pairs]\label{def:bpair} Let $\cK$ be a natural class. For $\lambda\geqslant |T|^+$, an element $\cM\in\cK$ is called a \emph{$\lambda$-$\cK$-beautiful pair} if the following condition is satisfied:
\begin{itemize}\label{BP-condition}
    \item[(BP)] whenever there are $\BP$-embeddings $f\colon \cA \to \cM$ and $g\colon \cA \to \cB$ with $\cA,\cB\in \cK_{<\lambda}$, there is a $\BP$-embedding $h\colon \cB\to \cM$ such that $f=h\circ g$. 
\end{itemize}
We call an $\cL_P$-structure $\cM$ a \emph{$\cK$-beautiful} pair if it is a $\lambda$-$\cK$-beautiful pair for some $\lambda\geqslant|T|^+$, and we will omit the $\cK$ when $\cK=\cK_{\mathrm{def}}$.
\end{definition}

\begin{remark}\label{rem:small_saturated} 
   It follows from Definition \ref{def:natural-class-K}(ii) that if $\cM$ is a $\lambda$-$\cK$-beautiful pair of $T$, then $P(\cM)$ is a $\lambda$-saturated model of $T$. However, note that it might happen that $M$ is \emph{not} a model of $T$. Although, in most cases of interest $\cK$-beautiful pairs will satisfy that $P(\cM)\preccurlyeq M\models T$ i.e. $(M,P(\cM))$ is an elementary pair. As we will see later (see Corollary \ref{cor:elem-equiv}), either all $\cK$-beautiful pairs are elementary pairs or no $\cK$-beautiful pair is an elementary pair.
\end{remark}

The following is an equivalent way of defining beautiful pairs. The properties are slightly easier to verify.

\begin{lemma}\label{lem:b-pair-finite} A structure $\cM\in \cK$ is $\lambda$-$\cK$-beautiful if and only if the following holds: 
\begin{itemize}
    \item[(i)] $P(\cM)$ is a $\lambda$-saturated model of $T$; 
    \item[(ii)]  whenever there are $\BP$-embeddings $f\colon \cA\to \cM$ and $g\colon \cA\to \cB$ with $\cA\in\cK_{<\lambda}$, $g(P(\cA))=P(\cB)$ and $\cB=\langle g(\cA)\cup\{b\}\rangle$ for some $b\in \cB$, there is a $\BP$-embedding $h\colon \cB\to \cM$ such that $f=h\circ g$. 

\end{itemize}
\end{lemma} 
\begin{proof} Suppose $\cM$ is a $\lambda$-$\cK$-beautiful pair. Definition \ref{def:bpair} clearly implies (ii), and condition (i) follows by Remark \ref{rem:small_saturated}. 

For the converse, let $f\colon \cA\to \cM$ and $g\colon \cA \to \cB$ be $\BP$-embeddings with $\cB\in\cK_{<\lambda}$. Without loss of generality, suppose $f$ is just inclusion.  

\emph{Step 1:} We may suppose $g(P(\cA))=P(\cB)$. Indeed, since $P(\cM)$ is $\lambda$-saturated by (i), there is an $\cL$-embedding $s\colon P(\cB)\to P(\cM)$ such that $(s\circ g)_{|P(\cA)} = \id_{P(\cA)}$. Let $\cA_{P(\cB)}$ be the base extension of $\cA$ to $P(\cB)$. By Lemma-definition~\ref{lem:base-change}, we may replace the base in our original amalgamation problem by $\cA_{P(\cB)}$. Solving for $\cA_{P(\cB)}$, $\cB$ and $\cM$ gives also a solution for $\cA$, $\cB$ and $\cM$. 

\medskip

\emph{Step 2:} Suppose by Step 1, that $g(P(\cA))=P(\cB)$. Let $(b_\alpha)_{\alpha<\kappa}$ be an enumeration of $B\setminus P(\cB)$. Inductively define \sepa-pairs $\cB_i\in\cK$
\begin{enumerate}[label=(\arabic*)]
    \item $\cB_0 \coloneqq\langle\cA\cup\{b_0\}\rangle$, 
    \item $\cB_{\alpha+1} \coloneqq \langle\cB_\alpha\cup\{b_{\alpha+1}\}\rangle$,
    \item $\cB_{\alpha} \coloneqq \bigcup_{\beta<\alpha} \cB_\beta$ for $\alpha$ limit.
\end{enumerate}

Note that for each $\alpha <\kappa$, since $\cK$ satisfies axiom (iv) in Definition~\ref{def:natural-class-K}, $\cB_\alpha\in \cK$. By (ii), there is a $\BP$-embedding $h_0\colon \cB_0\to \cM$ such that $h_0(P(\cB_0))=P(\cA)$. By induction and (ii), for each $\alpha<\kappa$ there is a $\BP$-embedding $h_\alpha\colon \cB_\alpha\to \cM$ such that $h_\alpha(P(\cB_\alpha))=P(\cA)$ and moreover, for $\alpha<\beta<\kappa$, $h_\beta$ extends $h_\alpha$. 
\end{proof}

In analogy to the usual Fra\"iss\'e theory, assuming $\cK$ has the joint embedding property (JEP)\footnote{ Recall that $\cK$ has the joint embedding property if for any $A,B\in \cK$, there is $C\in \cK$ and $\BP$-embeddings in $\cK$ such that $A\rightarrow C\leftarrow B$.} and the amalgamation property (AP), one may construct $\cK$-beautiful pairs. Note that our assumptions on $\cK$ entail that AP implies JEP. Indeed, it follows from the fact that substructures of $\cU$ satisfy JEP that any JEP problem in $\cK$ can be converted to an AP problem in $\cK$. Actually, the tensor product of definable types shows that $\cK_{\mathrm{def}}$ always satisfies JEP. However, in order to ensure that $\cK$-beautiful pairs are elementary pairs we need to impose on $\cK$ the following additional property:

\begin{definition}\label{def:ext_prop} We say that $\cK$ has the \emph{extension property} (EP) if the following holds. Given a single variable $x$, a structure $\cA$ from $\cK$ and a consistent (with $T$) $\cL(A)$-formula $\varphi(x)$, there is $\cB\in \cK$, a bp-embedding $f\colon \cA\to \cB$ and $b\in B$ satisfying $\varphi(x)$. 
\end{definition}

\begin{theorem}\label{thm:charct-beaut-pairs} The following are equivalent:
\begin{enumerate}[label=(\arabic*)]
    \item $\cK$ has the amalgamation property;
    \item $\lambda$-$\cK$-beautiful pairs exist for all $\lambda\geqslant |T|^+$.
\end{enumerate}
Moreover, assuming the above equivalent conditions hold, $\cK$-beautiful pairs are elementary pairs if and only if $\cK$ has the extension property.
\end{theorem}
\begin{proof} 
The equivalence between (1) and (2) uses standard arguments typically employed in Fra\"{i}ss\'e theory. So suppose (1) and (2) hold. 

For the left-to-right implication, let $\cA\in \cK$ and $\varphi(x)$ be a consistent $\cL(A)$-formula. Take $\lambda$ sufficiently big so that $\cA$ embeds into a $\lambda$-$\cK$-beautiful pair $\cM$. Note that $\cM$ is an elementary pair by assumption. Taking $\cM$ as $\cB$ shows EP, since $\varphi(M)\neq\emptyset$. The converse follows from the Tarski-Vaught test using EP over $\cA$ with respect to $\varphi(x)$. 
\end{proof}

We will later give examples of theories and natural classes for which EP fails (see later Section \ref{sec:tree}).

\begin{corollary}\label{cor:elem-equiv} Let $\lambda\geqslant |T|^+$, $\cM, \cN\in \cK$ and suppose that $\cM$ is $\lambda$-$\cK$-beautiful. Then, the following are equivalent: 
\begin{enumerate}[label=(\arabic*)]
    \item $\cN$ is $\lambda$-$\cK$-beautiful; 
    \item the set of partial isomorphisms between bp-substructures of $\cM$ and $\cN$ of size smaller than $\lambda$ (that are in $\cK$ by definition) has the back-and-forth property;
    \item $\cM\equiv_{\infty,\lambda} \cN$.
\end{enumerate} 
In particular, $\cM\equiv_{\infty,\lambda} \cN$ for any two $\lambda$-$\cK$-beautiful pairs $\cM$ and $\cN$, and therefore, all $\cK$-beautiful pairs are elementarily 
equivalent. 
\end{corollary}

\begin{proof}
The equivalence (1)$\Leftrightarrow$(2) and the implication (2)$\Rightarrow$(3) are clear. As for (3)$\Rightarrow$(2), it suffices to observe that any map in a back-and-forth system is partial elementary and that if $f$ is a partial elementary map between two structures from $\cK$, then the image of a $\BP$-substructure under $f$ is again a $\BP$-substructure.
\end{proof}

\begin{definition}\label{def:tbp}
When $\cK$-beautiful pairs exist, we will use $T_\BP(\cK)$ to denote the common $\cL_P$-theory of $\cK$-beautiful pairs. When $\cK=\cK_{\mathrm{def}}$, we will simply write $T_\BP$ instead of $T_\BP(\cK)$.
\end{definition}

Classical examples of $T_\BP$ will be gathered in Section \ref{sec:stable-omin}.

The following standard fact mirrors the situation of Corollary \ref{cor:elem-equiv}.   

\begin{fact}\label{rem:infinitary} Let $T'$ be any complete theory, $\lambda\geqslant |T'|^+$, and $\cM,\cN$ be two models of $T'$ such that $\cM$ is  $\lambda$-saturated. Then $\cN$ is $\lambda$-saturated if and only if $\cM\equiv_{\infty,\lambda} \cN$. \qed 
\end{fact}

\begin{lemma}\label{lem:beauty-transfer-cardinal} Suppose $\cK$-beautiful pairs exist. The following are equivalent: 
\begin{enumerate}[label=(\arabic*)]
    \item there is a cardinal $\lambda\geqslant |T|^+$ and $\cM\models T_\BP(\cK)$ which is $\lambda$-saturated and $\lambda$-$\cK$-beautiful; 
    \item  for every cardinal $\lambda\geqslant |T|^+$ and every model $\cM\models T_\BP(\cK)$, $\cM$ is $\lambda$-saturated if and only if $\cM$ is $\lambda$-$\cK$-beautiful. 
\end{enumerate}
\end{lemma}

\begin{proof} The implication (2)$\Rightarrow$(1) is trivial. For (1)$\Rightarrow$(2), let $\cM_0$ be a $\lambda$-saturated $\lambda$-$\cK$-beautiful pair and $\lambda'\geqslant|T|^+$ be a cardinal number. If $\lambda'\leqslant \lambda$, (2) follows directly from Fact~\ref{rem:infinitary} and Corollary~\ref{cor:elem-equiv}. So suppose $\lambda<\lambda'$. Assume $\cM$ is a $\lambda'$-saturated model of $T_\BP(\cK)$. Then it is also $\lambda$-saturated, and therefore, by Corollary \ref{cor:elem-equiv} and Fact~\ref{rem:infinitary}, $\lambda$-$\cK$-beautiful. By Lemma \ref{lem:b-pair-finite}, it suffices to show that for $\cA\subseteq_\BP \cM$ with $\cA\in \cK_{<\lambda'}$ and a $\BP$-embedding $\cA\to \cB$ with $\cB=\langle \cA\cup b\rangle$ such that $P(\cA)=P(\cB)$, there is a $\BP$-embedding $\cB \to \cM$ which is the identity on $\cA$. Let $(b_i)_{i\in I}$ be an enumeration of the set of finite tuples of $\cB$. For each $i\in I$, let $x_i$ be a tuple of variables associated to $b_i$ and set $x=(x_i)_{i\in I}$. Consider the partial $\cL_P$-type
\begin{equation}\label{eq:type}\tag{$\dagger$}
\Sigma(x)\coloneqq\tp_{\cL}(B/A) \cup \{(\forall y)(P(y) \to (\varphi(x_i,y) \leftrightarrow d_{p_i}\varphi(y)) : \varphi(x_i,y) \text{ in $\cL$}\}.    
\end{equation}
Here, $p_i$ denotes $\tp_{\cL}(b_i/P(\cB))$, and $d_{p_i}\varphi(y)$ the $\varphi$-definition of $p_i$, which is an $\cL(P(\cA))$-formula by assumption.

By $\lambda'$-saturation, if $\Sigma$ is consistent, then it is realized in $\cM$. Moreover, any realisation of $\Sigma$ in $\cM$ gives us the desired embedding. Indeed, note that for any singleton $b_i\in \cB$, if $\theta$ denotes the formula $x_i=y$, then 
\[
P(b_i) \ \Leftrightarrow \ \models (\exists y)d_{p_i}\theta(y),  
\]
which shows that $P$ is preserved by any such embedding. In addition, if $e_i\in \cM$ is the realisation of the variables $x_i$, the right-hand side of \eqref{eq:type} ensures that the resulting embedding is a $\BP$-embedding.   

Thus it suffices to show $\Sigma(x)$ is consistent. For every $i\in I$ and every finite subset $A_0$ of $\cA$, there is $\cA_1$ such that $A_0\subseteq \cA_1\subseteq_\BP \cA$, $\cA_1\in \cK_{<\lambda}$, $P(\cB_i)=P(\cA_1)$ where $\cB_i=\langle A_1\cup b_i\rangle$ and $\cA_1\subseteq_\BP \cB_i$. Since $\cM$ is $\lambda$-$\cK$-beautiful,  $\cB_i$ $\BP$-embeds into $\cM$ over $\cA_1$, which shows that $\Sigma(x)$ is finitely satisfiable in $\cM$. 

For the converse, assume $\cM$ is a $\lambda'$-$\cK$-beautiful pair. Let $\cN$ be a $\lambda'$-saturated elementary extension of $\cM$. By the the above implication, $\cN$ is $\lambda'$-$\cK$-beautiful. Then $\cM$ is $\lambda'$-saturated by Corollary \ref{cor:elem-equiv} and Fact~\ref{rem:infinitary}.    
\end{proof}

\begin{definition}\label{def:beauty_transfer} Suppose $\cK$-beautiful pairs exist. We say that \emph{$\cK$ has beauty transfer} if one of the equivalent conditions in Lemma \ref{lem:beauty-transfer-cardinal} holds. 
\end{definition}

\begin{definition}\label{def:UDDT-F} Let $\cF$ be a natural class of global definable types. We say $T$ has $\mathrm{UDDT}$($\cF$) if the types from $\cF$ are uniformly definable.
\end{definition}

\begin{proposition}\label{prop:UDDT(F)}
    Assume $\cK$-beautiful pairs exist and \emph{$\cK$} has beauty transfer. Then $\cK$ is an $\cL_P$-elementary class and $T$ has $\mathrm{UDDT}$($\cF_{\cK}$).
\end{proposition}

\begin{proof}
The class $\cK$ is closed under bp-substructures. As any elementary substructure of an \sepa-pair is a bp-substructure, it follows that $\cK$ is closed under elementary substructures, and so every model of $T_{\BP}(\cK)$ is in $\cK$, by beauty transfer. To show $\cK$ is $\cL_p$-elementary, it remains to show it is closed under ultraproducts.

We will first prove $\mathrm{UDDT}$($\cF_{\cK}$). Suppose for contradiction that this does not hold. We may then find a formula $\phi(x,y)$ such that for each $\chi_i(y,z)$ there is $\cM_i\models T_{\BP}(\cK)$ and a tuple $a_i$ from $\cM_i^x$ such that the type 
of $a_i$ over the predicate has no $\phi$-definition of the form $\chi_i(y,c)$. It is then clear that in any non-principal ultraproduct $\cM=\prod_{U}\cM_i$ the $\phi$-type of $a=(a_i)_{i\in I}/ U$ is not definable over $P(\cM)$, contradicting $\cM\models T_{\BP}(\cK)$.

Let us now show that $\cK$ is closed under ultraproducts. Let $(\cA_i)_{i\in I}$ be a family of structures from $\cK$. Let $\cA_i\subseteq_{\BP}\cM_i\models T_{\BP}(\cK)$, for $i\in I$. It follows from $\mathrm{UDDT}$($\cF_{\cK}$) in a straight forward manner that $\cA:=\prod_U \cA_i$ is a $\BP$-substructure of $\cM:=\prod_U\cM_i\models T_{\BP}(\cK)$, so in particular $\cA$ is in $\cK$. 
\end{proof}

For the remaining part of this section, we assume that $\cK$-beautiful pairs exist and that $\cK$ has beauty transfer. In particular, by Proposition \ref{prop:UDDT(F)}, for each $\cL$-formula $\varphi(x,y)$ there is an $\cL^{\eq}$-formula $d(\varphi)(y,z_{\varphi})$ such that every definable type $p\in\cF_\cK$ admits a $\varphi$-definition of the form $d(\varphi)(y,c)$ for some (unique) imaginary element $c$.

\begin{definition}\label{def:Lbp}
We define the \emph{language of beautiful pairs} $\cL_{\BP}$ (depending on $T$) as the following extension of $\cL_P$. First, we add to $\cL_{P}$ all $\cL^\eq$-sorts restricted to the predicate $P$. In addition, we add function symbols $c_\varphi$ for each $\cL$-formula $\varphi(x,y)$. 
\end{definition}

\begin{convention}\label{conv:axiom_BP} From now on, unless otherwise stated, every $\cL_\BP$-structure will satisfy the following axiom scheme, stating that the function $c_\varphi$ selects $\varphi$-canonical parameters:
\begin{equation*}\label{axiom_BP}
(\forall x)\big[c_\varphi(x)\in P \wedge(\forall y\in P)(\varphi(x,y) \leftrightarrow d(\varphi)(y,c_\varphi(x)))\big] 
\end{equation*}
where $\varphi(x,y)$ is an $\cL$-formula.
\end{convention}
\begin{remark}\label{rem:Lbp}
    Note that any \sepa-pair $\cA$ extends to an $\cL_\BP$-structure uniquely, by respecting the convention and interpreting the new imaginary sorts from $\cL^{\eq}$ by $P(\cA)^{\eq}$. It is easy to see that this is an expansion by definitions of the $\cL_P$-structure $\cA$. Moreover, a $\BP$-embedding  $f:\cA\to \cB$ is the same as an $\cL_\BP$-embedding between the induced $\cL_\BP$-structures.

In addition, it is worth noting that to equip every $\cA\in\cK$ with an $\cL_{\BP}$-structure, it is sufficient that $T$ has $\mathrm{UDDT}$($\cF_\cK$). 
\end{remark}

\begin{theorem}\label{thm:qe} Suppose $\cK$-beautiful pairs exists and $\cK$ has beauty transfer. Then the following holds: 
\begin{enumerate}[label=(\arabic*)]
    \item $T_\BP(\cK)$ is complete and admits quantifier elimination in $\cL_\BP$. 
    \item The predicate $P$ is stably embedded in $T_\BP(\cK)$ and pure as an $\cL$-structure (i.e., there is no new induced structure beyond $\cL$).  
\end{enumerate}
\end{theorem}

\begin{proof} Part (1) follows from the fact that $|T|^+$-saturated models are $|T|^+$-$\cK$-beautiful by Lemma \ref{lem:beauty-transfer-cardinal}, and the system of partial isomorphisms between their $\cL_\BP$-substructures of size smaller than $|T|^+$ is a back-and-forth system. Indeed, given an $\cL_{\BP}$-substructure $\cA=(A,P(\cA))$ of a $|T|^+$-saturated model of $T_{\BP}(\cK)$, we may canonically replace $\cA$ by $(\langle AB\rangle, B^{\eq})$, where $B\models T$ and $P(\cA)\subseteq B^{\eq}$ and then use the back-and-forth system for \sepa-pairs that are bp-embedded (Corollary \ref{cor:elem-equiv}). Part (2) follows from (1). 
\end{proof}
\begin{corollary}\label{cor:NIP-transfer} 
Suppose $\cK$-beautiful pairs exist, $\cK$ has beauty transfer and the extension property. Then if $T$ is NIP so is $T_\BP(\cK)$. 
\end{corollary}

\begin{proof} By Theorem \ref{thm:qe}, the $\cL_P$-theory $T_\BP(\cK)$ is bounded (i.e., every $\cL_P$-formula is equivalent to an $\cL_P$-formula where all quantifiers are over the predicate).  By the extension property, the big model is a model of $T$. By~\cite[Corollary 2.5]{cherinikov-simon}, $T_\BP(\cK)$ is NIP.   
\end{proof}

We suspect that by following carefully the proofs of \cite[Theorem 2.4]{cherinikov-simon} and \cite[Corollary 2.5]{cherinikov-simon}, one should be able to remove the assumption that $\cK$ has the extension property in Corollary \ref{cor:NIP-transfer}. But we leave it as a question.

\begin{question}\label{rem:NTP2-transfer}  Is Corollary~\ref{cor:NIP-transfer} still valid without assuming that $\cK$ has the extension property?
\end{question}

\subsection{Strict pro-definability of definable types and beauty transfer}
In this subsection, we relate beauty transfer to strict pro-definability of definable types. For a natural class of definable types $\cF$, we assume that $\cK_\cF$-beautiful pairs exist and $\cK_\cF$ has beauty transfer.
\begin{definition}\label{def:prodef}
Given a natural class $\cF$ of global definable types, for a finite tuple of variables $x$ and a sufficiently saturated model $M$ of $T$, consider the map $\tau_\cF$ sending a type $p\in \cF_x(M)$ to the infinite tuple of its canonical parameters. More precisely,  
\begin{equation*}
\tau_\cF\colon\cF_x(M)\to \prod_{\varphi} M^{z_\varphi} \hspace{1cm} p\mapsto (c(p,\varphi))_{\varphi} 
\end{equation*}
where $\varphi=\varphi(x,y)$ runs over all $\cL$-formulas  and where $z_\varphi$ corresponds to the imaginary sort variable in $d(\varphi)(y,z_\varphi)$ for $\varphi(x,y)$. We say $\cF$ is \emph{pro-definable} if the image of $\tau_\cF$ is $\ast$-definable (in the sense of Shelah), i.e., type-definable in a small number of variables. Assuming $\cF$ is pro-definable, we say $\cF$ is \emph{strict pro-definable} if the image of the projection of $\cF_x(M)$ onto any finite set of coordinates is a definable set. Since one may encode finitely many formulas in one, this is equivalent to  $\pi_\psi(\tau_\cF(\cF_x(M)))$ being definable for every formula $\psi$, where $\pi_\psi\colon \prod_{\varphi}M^{z_\varphi} \to M^{z_\psi}$ is the canonical projection. 
\end{definition}

Note that if $\cF$ is (strict) pro-definable, then so is $\cF_X$ for any $\cL(M)$-definable set $X$.

\begin{remark}\label{rem:prodef1}
By a result of M. Kamensky \cite{kamensky}, the definition of pro-definability (resp. strict pro-definability) given above agrees with the standard one. We remit the reader to \cite[Section 4]{cubi-ye} for details of the standard definition.
\end{remark}

\begin{fact}[{\cite[Proposition 4.1]{cubi-ye}}]
Assuming  $\mathrm{UDDT}(\cF)$, $\cF$ is pro-definable.   \qed
\end{fact}

The following theorem improves the above fact.

\begin{theorem}\label{thm:strictness} Let $\cF$ be a natural class of definable types. Suppose $\cK_\cF$-beautiful pairs exist and $\cK_\cF$ has beauty transfer. Then $\cF$ is strict pro-definable.   
\end{theorem}

\begin{proof} 
Let $\cN$ be a $\lambda$-$\cK_\cF$-beautiful pair with $\lambda>|T|^+$ so that $P(\cN)$ is $|T|^+$-saturated. Let $\psi(x,y)$ be an $\cL$-formula. Consider the $\cL_{\BP}$-definable subset $c_\psi(\cN)$ of $P(\cN)$. Since $\cN$ is $\lambda$-$\cK_\cF$-beautiful, every type from $\cF_x(P(\cN)$ is realized in $\cN$. Thus
\[
\pi_\psi(\tau_\cF(\cF_x(P(\cN)))=c_\psi(\cN).  
\]
By Part (2) of Theorem \ref{thm:qe}, $c_\psi(\cN)$ is $\cL^\eq$-definable in $P(\cN)$. 
\end{proof}

We do not know if the converse is true in general (see later Question \ref{question:nfcp}). 

\begin{definition}\label{def:surj_transfer} Let $\cF$ be a natural class of definable types. For a definable function $f\colon X\to Y$, we let $f^\cF\colon \cF_X(\cU)\to \cF_Y(\cU)$ be the restriction of the pushforward $f_*$ to $\cF_X(\cU)$. We say $\cF$ has \emph{surjectivity transfer} if whenever $f$ is surjective, so is $f^\cF$. When $\cF=\D{}$, we simply say that $T$ has surjectivity transfer. 
\end{definition}

\begin{remark}
If $T$ has $\mathrm{UDDT}$ and $T^{\eq}$ has surjectivity transfer, then $T^\eq$ has $\mathrm{UDDT}$. 
\end{remark}

\begin{lemma}\label{lem:EP-to-surj} Let $\cF$ be a natural class of definable types. Then $\cK_\cF$ has the extension property if and only if $\cF$ has surjectivity transfer.  
\end{lemma}

\begin{proof} $(\Rightarrow)$ Let $f\colon X\to Y$ be a definable surjection and let $p(x)\in \cF_Y(C)$, where $C$ is a small elementary substructure of $\cU$ over which the data are definable. Let $a$ be a realisation of $p$ and $\cA$ be the $\cL_P$-structure given by $A=\langle Ca\rangle_{\cL}$ and $P(\cA)=C$. By construction, $\cA\in \cK_\cF$. Let $\varphi(x)$ be the formula $f(x)=a$. By the extension property, there is $\cB\in\cK_\cF$ such that $\cA\subseteq_{\BP} \cB$ together with $b\in \cB$ such that $\varphi(b)$ holds. Letting $q\in \cF_X(P(\cB))$ be the corresponding global extension of $\tp_{\cL}(b/P(\cB))$, we have that $f^\cF(q)=p$.

$(\Leftarrow)$ Conversely, let $\cA\in\cK_{\cF}$ and $\varphi(x)$ be a consistent $\cL(A)$-formula. Choose an $\cL$-formula $\psi(x,\overline{y})$ and $\overline{a}$ from $A$ such that $\varphi(x)=\psi(x,\overline{a})$. Set $\chi(\overline{y}):=\exists x\psi(x,\overline{y})$, let $X:=\psi(\cU)$, $Y:=\chi(\cU)$ and $f:X\rightarrow Y$ the projection map, which is surjective by construction. Let 
$p(x):=\tp_{\cL}(\overline{a}/P(\cA))|\cU$ which lies in $\cF_Y(\cU)$. By surjectivity transfer, there is $q\in \cF_X(\cU)$ such that $p=f^{\cF}(q)$. Let $C$ be a small model containing $P(\cA)$ such that $q$ is defined over $C$. We choose $b$ such that $(b,\overline{a})\models q|C$. Then $\cB:=(\langle ACb\rangle,C)$ is a bp-extension of $\cA$ in $\cK_{\cF}$ realizing $\varphi(x)$.
\end{proof}

For more discussion about the two notions above, see also~\cite[Lemma 4.2.6, Remark 4.2.8]{HL}.

Recall that $T$ has density of definable types if $\D{X}(\acl^\eq(\ulcorner X\urcorner))\neq\emptyset$ for every non-empty definable subset $X$ in the real sorts. Note that if $T$ has density of definable types then $T^{\eq}$ also has it (for an interpretable set $Y\subseteq M^n/E$, take the pushforward of a definable type in $\D{X}(\acl^\eq(\ulcorner X\urcorner))$ where $X$ is the union of all the $E$ classes contained in $Y$).  

\begin{lemma}\label{lem:density} Suppose $T$ has density of definable types. Then 
\begin{enumerate}[label=(\arabic*)]
    \item the class $\cK_\Def$ has the extension property; 
    \item if $\cM=(M,P(\cM))$ is a beautiful pair for $T$, then $\cM^*=(M^{\eq}, P(M^{\eq}))$ is a beautiful pair for $T^\eq$.
\end{enumerate} 
\end{lemma}

\begin{proof}
For (1), let $\cA$ be an \sepa-pair and $\varphi(x)$ be a consistent $\cL(A)$-formula with $x$ a real variable. By density of definable types, there is $p\in\D{X}(\acl^\eq(\ulcorner X\urcorner))$ where $X$ is the definable set associated to $\varphi$ (as a definable set in the monster model of $T$). In particular, $p$ is $\acl^\eq(A)$-definable. Letting $b$ be a realisation of $p|\acl^\eq(A)$, we have that  
\[
\tp_{\cL^\eq}(b, \acl^\eq(A)/P(\cA))
\]
is definable. So in particular, $\tp_{\cL}(Ab/P(\cA))$ is definable. Finally, we conclude by setting $\cB=(\langle A b\rangle,P(\cA))$. 

\medskip

For (2), let $\cA$ be an $\cL_P^\eq$-substructure of $\cM^*$ and $\cA\rightarrow \cB$ be a $\BP$-embedding such that $P(\cA)=P(\cB)$. Without loss of generality, suppose $\acl^\eq(A)=A$. Take a small model $\cM_0\preccurlyeq_{\cL_P} \cM$ such that $\cA\subseteq\cM_0$. Since density of definable types transfers to $T^\eq$, there is a set of real elements $E$ such that $E^\eq$ contains both $A$ and $P(\cA)$ with  $\tp_{\cL^\eq}(E/A)$ admitting a global $A$-definable extension (it exists since $A$ is $\acl^\eq$-closed). Let $E'\models \tp_{\cL^\eq}(E/A)|M_0$ and $\cA'$ be $(\langle E' \cup M_0 \rangle_{\cL^\eq}, P(\cM_0))$. Note that $\cA'$ is an \sepa-pair for $T^\eq$ since $\cM_0$ is an \sepa-pair containing $A$ and $\tp(E'\cup M_0/P(\cM_0))$ is definable (by transitivity of definable types). We may $\BP$-embed $\cA'_{|\cL}$ in $\cM$ over $(\cM_0)_{|\cL}$ by beauty of $\cM$ and thus $\BP$-embed $\cA'$ in $\cM^*$ over $\cM_0$ since they are generated by real elements.   

We may assume by (1) (applied to $T^\eq$), that the real part of $B$ generates all of $\cB$. Take $A''$ realizing $\tp(A'/A)|B$ and set $\cB'=(\langle A''\cup B\rangle, P(\cB))$. By construction, $\cB'$ is an \sepa-pair extending $\cA'$. Since $\cA'$, $\cB'$ are generated by real elements, 
a solution of the problem for $\cA_{|\cL}'$, $\cB_{|\cL}'$ and $\cM$, induces the desired solution for $\cA$, $\cB$ and $\cM^*$. 
\end{proof}

\begin{remark}\label{rem:density} Given a natural class of definable types $\cF$, Lemma \ref{lem:density}(1) could also be shown for the class $\cK_\cF$ assuming $T=T^\eq$, $T$ has density of $\cF$-types, $\cF$ is closed under $\acl^\eq$ and closed in towers. 
\end{remark}

\subsection*{An example where surjectivity transfer fails: the leveled binary tree}\label{sec:tree}

Consider the binary tree $2^{<\omega}$ and let $\ell\colon 2^{<\omega}\to \omega$ be the function sending an element $t\in 2^{<\omega}$ to its distance to the root $r\in 2^{<\omega}$, where $\ell(r)=0$. We will study $(2^{<\omega},\omega)$ as a two-sorted structure in the following language.  
\begin{definition}\label{def:tree-language} We define the two sorted language of binary leveled meet trees $\cL_\mathrm{tree}$ as follows. The \emph{tree sort} $\mathbf{T}$ contains a binary relation $<$, a binary function $\wedge$ and a unary function $\mathrm{pred}_T$. The \emph{value sort} $\mathbf{V}$ contains a binary relation $<$ and a unary function $\mathrm{pred}_V$. We also add a function symbol $\ell\colon \mathbf{T} \to \mathbf{V}$. 
\end{definition}

We interpret $M=(2^{<\omega},\omega)$ as an $\cL_\mathrm{tree}$-structure in the standard way, where $\wedge$ corresponds to the meet of two elements, $\mathrm{pred}_T$ (resp. $\mathrm{pred}_V$) to the predecessor function, extended by $\mathrm{pred}_T(r)=r$  (resp. $\mathrm{pred}_V(0)=0$), and $\ell$ as the above distance function. Let $T_{\mathrm{tree}}$ be the $\cL_{\mathrm{tree}}$-theory of $M=(2^{<\omega},\omega)$. The following is folklore.  

\begin{fact}\label{fact:tree-qe} The theory $T_{\mathrm{tree}}$ has quantifier elimination.\qed
\end{fact}

\begin{corollary}\label{cor:no-surj-transfer}
The theory $T_{\mathrm{tree}}$ does not have surjectivity transfer. In particular, $\cK_\Def$ does not have the extension property.
\end{corollary}

\begin{proof}
It follows from Fact \ref{fact:tree-qe} that the value sort is purely stably embedded. Moreover, note that every global definable one-type in the tree sort must be a realized type. In particular, the definable type at $+\infty$ in the value sort is not in the image under $\ell_*$ of any definable type, which shows that the $\cL_{\mathrm{tree}}$-theory of $M=(2^{<\omega},\omega)$ does not have surjectivity transfer. By Lemma \ref{lem:EP-to-surj}, this theory also constitutes an example of a theory for which $\cK_\Def$ does not have the extension property. 
\end{proof}

It is easy to see that beautiful pairs exist and $\cK_\Def$ has beauty transfer. Indeed, models of $(T_{\mathrm{tree}})_\BP$ are of the form $\cM=(M, P(\cM))$ where the $\cL_{\mathrm{tree}}$-reduct of
the structure $P(\cM) = (\mathbf{T}(P(\cM)), \mathbf{V}(P(\cM)))$ is a model of $T_{\mathrm{tree}}$, and for $M = (\mathbf{T}(M), \mathbf{V}(M))$ it holds that $\mathbf{T}(M)=\mathbf{T}(P(\cM))$ and $\mathbf{V}(M)$ is a proper elementary end-extension of $\mathbf{V}(P(\cM))$.

\section{Relation with previous work and variants}\label{sec:stable-omin} 

In this section we show how our general context of beautiful pairs is related to two previous constructions, first in the context of stable theories and second, in the context of o-minimal theories. In the second context we also consider variants for o-minimal expansions of divisible ordered abelian groups without poles.

\subsection{Beautiful pairs of stable theories} Suppose $T$ is stable. Note that $T$ has $\mathrm{UDDT}$. Recall that an elementary pair $(M,P(M))$ of models of $T$ is a \emph{belle paire} in the sense of Poizat \cite{poizat} if the following two conditions hold: 
\begin{enumerate}[(i)]
    \item $P(M)$ is $|T|^+$-saturated;
    \item for every finite tuple $a$ of $M$, every $\cL$-type over $P(M)\cup\{a\}$ is realized in $M$.   
\end{enumerate}

Recall that a theory $T$ has $\mathrm{nfcp}$ if no formula $\varphi(x,y)$ satisfies the following condition: for every integer $n$ there exists a subset $A$ of a model of $T$ and an inconsistent set $H\subseteq \{\varphi(x,a), \neg\varphi(x,a) : a\in A^y\}$ such that every subset of $H$ of size less than $n$ is consistent. Note, $T$ has $\mathrm{nfcp}$ if and only if $T$ is stable and $T^{\eq}$ eliminates $\exists^\infty$. The following is a reformulation of results of Poizat from \cite{poizat}: 

\begin{theorem}[Poizat]\label{thm:poizat} Let $T$ be stable. Then belles paires exist and any two belles paires are elementary equivalent. Moreover, the following are equivalent:
\begin{enumerate}[label=(\arabic*)]
    \item $T$ has $\mathrm{nfcp}$;
    \item there is a $|T|^+$-saturated belle paire;   
    \item $\D{X}$ is strict pro-definable for every definable set $X$.
\end{enumerate}
\end{theorem}

Given $\lambda\geqslant|T|^+$, call an elementary pair $(M,P(M))$ a $\lambda$-\emph{belle paire} if the following two conditions hold 
\begin{enumerate}
    \item[(i$\lambda$)] $P(M)$ is $\lambda$-saturated;
    \item[(ii$\lambda$)] for every subset $A\subseteq M$ of cardinality $<\lambda$, every $\cL$-type over $P(M)\cup A$ is realized in $M$.   
\end{enumerate}

The following proposition is left as an exercise. 

\begin{proposition}\label{prop:poizat_equiv} Let $T$ be stable and $\lambda\geqslant |T|^+$. An $\cL_P$-structure $\cM$ is a $\lambda$-beautiful pair if and only if it is a $\lambda$-belle paire. In particular, $T_\BP$ is the common theory of belles paires. \qed
\end{proposition}

As a corollary, by Lemma \ref{lem:beauty-transfer-cardinal}, $\cK_\Def$ has beauty transfer if and only if any of the conditions (1)-(3) in Theorem \ref{thm:poizat} holds.

\subsubsection{Algebraically closed fields}\label{sec:ACF}

We finish this section with a brief description of beautiful pairs of algebraically closed fields. Although the results are classical, we present them here to provide an illustration of our framework. 

Let $\cL$ denote the language of rings $\cL_\Ring$ and let $\ACF$ denote a completion of the $\cL$-theory of algebraically closed fields. Let $\ACF^2$ be the $\cL_P$-theory of pairs of models of $\ACF$ and set $T_0=\ACF^2\cup \{(\forall x)(P(x))\}$ and $T_1=\ACF^2\cup \{(\exists x)(\neg P(x))\}$. 

\begin{proposition}\label{prop:ACF} The theories $T_0$ and $T_1$ axiomatize respectively the theories of beautiful pairs $\ACF(\cK_\triv)$ and $\ACF(\cK_\Def)$. In particular, $T_0$ and $T_1$ are the only completions of pairs of models of $\ACF$ and the corresponding natural classes have beauty transfer.  
\end{proposition}

\begin{proof} A classical result of A. Robinson \cite{Rob77} shows that $T_0$ and $T_1$ are complete (and hence the only completions of $\ACF^2$). Since $\ACF$ has nfcp, by Theorem \ref{thm:poizat}, both $T_0$ and $T_1$ axiomatize respectively the theories of beautiful pairs $\ACF(\cK_\triv)$ and $\ACF(\cK_\Def)$, and $\cK_\triv,\cK_\Def$ have beauty transfer. Note that in this precise example, quantifier elimination in $\cL_{\BP}$ was already obtained by F. Delon in \cite{Del12}, who provided an explicit interpretation of the canonical parameter functions $c_\varphi$.  
\end{proof}

\subsection{Beautiful pairs of o-minimal theories}\label{sec:omin}

Let $T$ be a complete (dense) o-minimal theory. By Marker-Steinhorn's theorem \cite{MS}, an elementary pair of models $M\preccurlyeq N$ of $T$ is stably embedded if and only if $M$ is Dedekind complete in $N$. In particular, the class of stably embedded elementary pairs of models of $T$ is an elementary class. Recall that we use $T_{\SE}$ to denote the corresponding theory. Moreover, by \cite[Theorem 6.3]{cubi-ye}, $T$ has uniform definability of types. Recall that the non-realized definable types in one variable over a model $M$ of $T$ correspond to the type at infinity (resp. minus infinity) $p_{\infty}$ (resp. $p_{-\infty}$) and to the types of elements infinitesimally close to an element of $a\in M$ from the right (resp. left) $p_{a^+}$ (resp. $p_{a^-}$). 

In \cite{Pillay_def}, A. Pillay provided an axiomatization of an $\cL_P$-theory $T^*$ of stably embedded pairs of models of $T$ which, if consistent, is complete \cite[Theorem 2.3]{Pillay_def}. It was later shown by B. Baisalov and B. Poizat that $T^*$ is always consistent \cite[Proposition, p.574]{BaPo98}. They also provided the following simpler axiomatization (see the proof of \cite[Proposition, p.574]{BaPo98}), where a pair $(M,P(M))$ is a model of $T^*$ if  the following conditions are satisfied:  
\begin{itemize}
    \item $(M,P(M))\models T_{\SE}$
    \item Given $\cL$-definable partial functions $f(x, y), g(x, y)$ (possibly constant $\pm\infty$) and a tuple $a\in M^x$, if for all $b_1,b_2\in P(M)^y$ 
    \[
M\models -\infty\leqslant f(a, b_1) < g(a, b_2)\leqslant +\infty
\]
whenever both sides are defined, then there is $e \in M$ such that for all $b_1,b_2\in P(M)^y$
\begin{equation}\label{eq:axiomA}\tag{$\ast$}
M\models f(a,b_1)<e<g(a,b_1),
\end{equation}
whenever both sides are defined. 
\end{itemize}

As the following result shows, $T^*$ axiomatizes the theory of beautiful pairs of $T$. 

\begin{theorem}\label{thm:omin-beauty} The theory $T_\BP$ exists and $\cK_\Def$ has beauty transfer. Moreover, it is axiomatized by $T^*$.    
\end{theorem}

\begin{proof} Since $T^*$ is consistent by \cite[Proposition, p.574]{BaPo98}, Lemma \ref{lem:beauty-transfer-cardinal} ensures it suffices to show that every $|T|^+$-saturated model of $T^*$ is a beautiful pair. Let $\cM=(M,P(M))$ be such a model. We use Lemma \ref{lem:b-pair-finite}. Condition $(i)$ is clear. For condition (ii), let $f\colon \cA\to \cM$ and $g\colon \cA\to \cB$ be $\BP$-embeddings with $\cB\in \cK_{\Def,<|T|^+}$, $g(P(\cA))=P(\cB)$ and $B=\langle A\cup\{b\}\rangle$. Note that $P(\cA)$ is a model of $T$. Furthermore, we may assume that $b\notin \dcl_{\cL}(\cA\cup P(\cM))$. We are exactly in case (IIb) of the proof of \cite[Theorem 2.3]{Pillay_def}. Proceeding word for word, we can find a  $\BP$-embedding $g\colon \cB\to \cM$ over $\cA$.
\end{proof}

\begin{corollary}\label{prop:omin_bpairs} The class $\cK_\Def$ has the amalgamation property and the extension property. 
\end{corollary}

\begin{proof} Immediate from Theorems~\ref{thm:charct-beaut-pairs} and~\ref{thm:omin-beauty}.  
\end{proof}

\subsubsection{Bounded and convex pairs}\label{sec:bounded-convex}

\begin{definition}\label{def:bdd-convec} Let $M\preccurlyeq N$ be an elementary extension of models of $T$. 

\begin{itemize}[leftmargin=*]
    \item The pair $(N,M)$ is \emph{bounded} (we also say that \emph{$N$ is bounded by $M$}) if for every $b\in N$, there are $c_1, c_2\in M$ such that $c_1\leqslant b\leqslant c_2$. For a small subset $A$  of $\cU$, a type $p\in \D{x}(A)$ is \emph{bounded} if for any small model $M$ containing $A$ and every realisation $a\models p|M$, there is an elementary extension $M\preccurlyeq N$ with $a\in N^x$ and such that $(N,M)$ is a bounded pair. 
    \item The pair $(N,M)$ is \emph{convex} if $M$ is a convex subset of $N$, i.e., for every $b\in N\setminus M$, either $M<b$ or $b<M$. We say a type $p\in \D{x}(A)$ is \emph{cobounded} if for any small model $M$ containing $A$ and every realisation $a\models p|M$, there is an elementary extension $M\preccurlyeq N$ with $a\in N^x$ and such that $(N,M)$ is a convex pair. Note that realized types are cobounded.  
\end{itemize}
\end{definition}

\begin{definition} Let $M$ be a model of $T$. A \emph{pole} in $M$ is a definable continuous bijection between a bounded and an unbounded interval. 
\end{definition}

Observe that, by o-minimality, no model of $T$ has a pole if and only if any type of the form $p_{a^+}$ or $p_{a^-}$ is orthogonal to $p_{+\infty}$ and to $p_{-\infty}$.\footnote{Two global definable types $p(x)$ and $q(y)$ are said to be orthogonal if $p(x)\cup q(y)$ is a complete global type.} Indeed, note that by o-minimality given any two non-realized global definable types $p$ and $q$, $p$ is non-orthogonal to $q$ if and only if there is a local $\cU$-definable homeomorphism sending a realisation of $p$ to a realisation of $q$.

We let $\cK_\bdd$ be the class of bounded \sepa-pairs. Observe it corresponds to $\cK_\cF$ for $\cF$ the collection of definable bounded types. Similarly, we let $\cK_{\conv}$ be the class of convex \sepa-pairs. It corresponds to $\cK_\cF$ for $\cF$ the collection of cobounded definable types, i.e., of those definable types orthogonal to any type of the form $p_{a^+}$ and $p_{a^-}$. Note that such classes are different from $\cK_{\triv}$ precisely when no model of $T$ has a pole. Moreover, the classes of global bounded types and global cobounded types are natural classes of definable types, hence by Lemma \ref{lem:nat-to-nat}, $\cK_\bdd$ and $\cK_{\conv}$ are natural classes of \sepa-pairs. 

\medskip

\begin{assumption}
We assume from now that no model of $T$ has a pole.
\end{assumption}

Consider the following $\cL_\BP$-theories $T_\bdd^*$ and $T_{\conv}^*$. The theory $T_\bdd^*$ consists of axioms expressing for a pair $(M, P(M))$ that   
\begin{enumerate}
    \item[$(1)_{\bdd}$] $P(M)\preccurlyeq M$, $(M,P(M)) \models T_{\SE}$, and $M$ is bounded by $P(M)$, 
    \item[$(2)_{\bdd}$] given $\cL$-definable partial functions $f(x, y), g(x, y)$ and a tuple $a\in M^x$, if there are $c_1, c_2\in P(M)$ such that for all $b_1,b_2\in P(M)^y$  
    \[
M\models c_1<f(a, b_1) < g(a, b_2)<c_2
\]
whenever both sides are defined, then there is $e \in M$ such that for all $b_1,b_2\in P(M)^y$
\begin{equation*}
M\models f(a,b_1)<e<g(a,b_2),
\end{equation*}
whenever both sides are defined.
\end{enumerate}
Let us now assume that $T$ expands $\DOAG$. The theory~$T_{\conv}^*$ consists of axioms expressing for a pair $(M, P(M))$  
\begin{enumerate}
\item[$(1)_{\infty}$] $P(M)\preccurlyeq M$, $(M,P(M)) \models T_{\SE}$, and $(M,P(M))$ is convex,
\item[$(2)_{\infty}$] given $\cL$-definable partial functions $f(x, y), g(x, y)$ (possibly constant $\pm\infty$) and a tuple $a\in M^x$, if for all $b_1,b_2\in P(M)^y$  
    \[
M\models P(M) + f(a,b_2) <  g(a, b_1) 
\]
whenever both sides are defined, then there is $e \in M$ such that for all $b_1,b_2\in P(M)^y$
\begin{equation*}
M\models f(a,b_1)<e<g(a,b_2),
\end{equation*}
whenever both sides are defined. 
\end{enumerate}

\begin{lemma}\label{lem:consistency1} For any o-minimal theory $T$, the theory $T_\bdd^*$ is consistent. If in addition $T$ expands $\DOAG$, then $T_{\conv}^*$ is consistent.
\end{lemma} 

\begin{proof} We follow the same strategy of proof as in \cite{BaPo98}. Let $M_0$ be a model of $T$ and let $N$ be an $|\cL(M_0)|^+$-saturated elementary extension of $M_0$. 

We start with $T_\bdd^*$. Let $M$ be maximal such that $M_0\preccurlyeq M\preccurlyeq N$, $M_0\preccurlyeq M$ is stably embedded and $M$ is bounded by $M_0$. Note that such a maximal $M$ exists by Zorn's lemma. Observe that for any $a\in N\setminus M$, there is $b\in \dcl(Ma)$ such that $\tp(b/M_0)$ is either not bounded or not definable. Setting $(M,P(M))=(M,M_0)$, we claim that $(M,P(M))$ is a model of $T_\bdd^*$. The axiom scheme $(1)_{\bdd}$ is clear. 

For $(2)_{\bdd}$, let $f(x, y), g(x, y)$ be $\cL$-definable partial functions and fix $a\in M^x$. Let $A$ (resp. $B$) be the image of $f(a,y)$ (resp. the image of $g(a,y)$) under all $b\in P(M)^y$ (whenever defined). Suppose there are $c_1, c_2\in P(M)$ such that $c_1<A<B<c_2$ and for a contradiction that there is no $e\in M$ such that $A<e<B$. Since the cardinalities of $A$ and $B$ are bounded by the cardinality of $P(M)$, by saturation of $N$, we find  $d\in N$ be such that $A<d<B$. Set $M':=\dcl(Md)$. Note that $M'$ is bounded by $M$ and thus by $M_0$, since $A<d<B$ and $T$ has no pole. By maximality of $M$, the pair $(M',M_0)$ is thus not an \sepa-pair, so in particular $\tp(d/M)$ is not definable. Now we finish exactly as in the proof of \cite[Proposition, p.574]{BaPo98}, finding $e\in M'\setminus M$ with $\tp(e/M_0)$ not definable and leading this to a contradiction.

The proof for $T_{\conv}^*$ is similar. Let $M$ be maximal such that $M_0\preccurlyeq M\preccurlyeq N$, $M_0\preccurlyeq M$ is stably embedded and $(M, M_0)$ is convex. Again, such a maximal $M$ exists by Zorn's lemma. It follows that for any $a\in N\setminus M$, there is $b\in \dcl(Ma)$ such that $\tp(b/M_0)$ is bounded and non-realized. Setting $(M,P(M))=(M,M_0)$, we claim that $(M,P(M))$ is a model of $T_{\conv}^*$. The axiom scheme $(1)_{\infty}$ is clear. 

For $(2)_{\infty}$, let $f(x, y), g(x, y)$ be $\cL$-definable partial functions and fix $a\in M^x$. Let $A$ (resp. $B$) be the image of $f(a,y)$ (resp. the image of $g(a,y)$) under all $b\in P(M)^y$ (whenever defined). The cases where $B=\{+\infty\}$ or $A=\{-\infty\}$ are not excluded.  Suppose that $P(M)+A<B$ and for a contradiction that there is no $e\in M$ such that $A<e<B$. Since the cardinalities of $A$ and $B$ are bounded by the cardinality of $P(M)$, by saturation of $N$, we find $d\in N$ such that $A<d<B$. Set $M':=\dcl(Md)$. By maximality of $M$ and as $T$ has no poles, $\tp(d/M)$ cannot be at infinity. We claim that it cannot be of the form $p_{c^+}$ or $p_{c^-}$ (for some  $c\in M$) either. Indeed, $A$ (resp. $B$) has no largest finite element (resp. smallest finite element) since $P(M)+A<B$ and $\tp(d/M)$ is determined by filling the cut $A<B$. Thus, $\tp(d/M)$ is not definable and we may finish the proof as in the bounded case.
\end{proof}

\begin{theorem}\label{thm:bb_infty} For any o-minimal theory $T$, the theory $T_\bdd^*$ axiomatizes $T_\BP(\cK_\bdd)$ and $\cK_\bdd$ has beauty transfer. If $T$ extends $\DOAG$, the theory $T_{\conv}^*$ axiomatizes $T_\BP(\cK_{\conv})$ and $\cK_{\conv}$ has beauty transfer.  
\end{theorem}

\begin{proof} Consistency follows from Lemma \ref{lem:consistency1}. The remaining is achieved as in the proof of Theorem \ref{thm:omin-beauty}. \end{proof}

\begin{remark}\label{rem:omin_dries} 
When $T$ extends $\RCF$, A.~H.~Lewenberg and L.~van den Dries proved in~\cite{vdd-lewen-1} that $T_\BP$ is axiomatized by $T_{\SE}\cup \{(\exists x)(\neg P(x))\}$ (they call such pairs, \emph{tame pairs}). In particular, as for $\ACF$, there are only two completions of $T_{\SE}$ and (up to passing to an expansion by definitions) each one corresponds to either $T_\BP(\cK_\triv)$ or $T_\BP$. This suggests the natural question of whether this holds more generally for expansions of $\DOAG$ (see Question \ref{question:amalg-omin}).
\end{remark}

\section{Beautiful pairs of ordered abelian groups}\label{sec:bp-oag} 

In this section we study beautiful pairs of complete theories of regular ordered abelian groups. In the cases of $\DOAG$ and Presburger arithmetic, we provide explicit axiomatizations and we characterize all completions of $T_{\SE}$.  

\subsection{Divisible ordered abelian groups}

We let $\cL$ be the language of ordered abelian groups $\cL_\og=\{+,-,<,0\}$ and $T$ be $\DOAG$. We will show that $T_{\SE}$ has exactly 4 completions: $T_\BP(\cK_\triv)$, $T_\BP(\cK_\bdd)$, $T_\BP(\cK_{\conv})$ and $T_\BP$ (see Section~\ref{sec:bounded-convex} for the definition of $\cK_\bdd$ and $\cK_{\conv}$). Consider the $\cL_P$-sentences: 

\begin{align*}
\label{infty}\tag{$E_{\infty}$} & (\exists x)(\forall y)(P(y) \to y<x);  \\
\label{0+}\tag{$E_{0^+}$}  & (\exists x)(\forall y)((P(y) \wedge 0<y) \to 0<x<y).    \\
\end{align*}

\begin{theorem}\label{thm:DOAG-complete} It holds that 
\begin{enumerate}[label=(\arabic*)]
    \item $T_\BP$ is axiomatized by $T_{\SE}\cup\{E_{0^+}, E_\infty\}$;
    \item $T_\BP(\cK_\bdd)$ is axiomatized by $T_{\SE}\cup\{E_{0^+}, \neg E_\infty\}$;
    \item $T_\BP(\cK_{\conv})$ is axiomatized by $T_{\SE}\cup\{E_{\infty}, \neg E_{0^+}\}$; 
    \item $T_\BP(\cK_\triv)$ is axiomatized by $T_{\SE}\cup\{\neg E_{\infty}, \neg E_{0^+}\}$. 
\end{enumerate} 
Moreover, the above are all the completions of $T_{\SE}$, and all the corresponding natural classes have beauty transfer.
\end{theorem}

\begin{proof} The first part of the `Moreover' statement follows from (1)-(4) in the statement of the theorem. Beauty transfer follows from Theorems \ref{thm:omin-beauty} and \ref{thm:bb_infty}. 

Let us show the result for $T_\BP$, the other cases being easily derived from the proof. Suppose $\cM$ is an $\aleph_1$-saturated model of $T_{\SE}\cup\{E_{0^+}, E_\infty\}$, we will show that it is $\aleph_1$-beautiful. By Lemma \ref{lem:b-pair-finite}, it suffices to show condition $(ii)'$, so let $f\colon \cA\to \cM$ and $g\colon \cA \to \cB$ be $\BP$-embeddings with $g(P(\cA))=P(\cB)$, $B=\langle A\cup\{b\}\rangle $ and $A$ countable. After possibly replacing $A$ with its divisible hull, we may suppose that $f$ and $g$ are the identity map and $P(\cA)\preccurlyeq_{\cL} A\preccurlyeq_{\cL} M$. We need to show that there is a $\BP$-embedding $h\colon \cB\to \cM$ such that $h_{|A}=\id$. There are two possibilities.  

\emph{Case 1:} Assume that up to interdefinability over $A$, there is some $b\models p_{0^+}|P(\cA)$ and $B$ is generated by $b$ over $A$. By saturation of $\cM$ and axiom $E_{0^+}$, there is $b'\in M$ realizing $p_{0^+}|P(\cM)$ with $\tp_\cL(b'/A)=\tp_\cL(b/A)$.

We now show that the map extending $f$ and sending $b$ to $b'$ is a $\BP$-embedding, that is,  
\begin{equation*}
    b'\cup A \models p\coloneqq \tp_{\cL}(b\cup A/P(\cA))|P(\cM).
\end{equation*}
Let $x$ denote a variable intended for $b$ and $y$ variables for elements in $A$. Let $\varphi(x,y)$ be a $P(\cM)$-formula in $p$. By quantifier elimination, we may assume that $\varphi(x,y)$ is of the form $y+c \ \square \ x$ for some  $c\in P(\cM)$ and $\square$ stands for $>$ or $<$ (note that $=$ cannot arise since $b\notin \cA$). We may assume that the element $a\in A$ corresponding to $y$ is in the convex hull of $P(\cM)$ (hence in the convex hull of $P(\cA)$), the other two cases being clear. By possibly changing $c$, we may further assume $a\models p_{0^+}|P(\cA)$, $a=0$ or $a\models p_{0^-}|P(\cA)$. Since $f$ is a $\BP$-embedding, we also have $a\models p_{0^+}|P(\cM)$, $a=0$ or $a\models p_{0^-}|P(\cM)$. In the case that $c=0$, $\varphi(b',a)$ holds since $b$ and $b'$ realize the same cut over $A$. If $c\neq 0$, the sign of $a+c-b$ only depends on the sign of $c$. Therefore, $\varphi(b',a)$ holds too.  

\emph{Case 2:} Assume that up to interdefinability over $A$, there is some $b\models p_{+\infty}|P(\cA)$ and $B$ is generated by $b$ over $A$ such that $b$ is not interdefinable with any $\Tilde{b} \models p_{0^+}|P(\cA)$ over $A$. By saturation of $\cM$, there is $b'\in M$ realizing $p_{+\infty}|P(\cM)$ and with $\tp_\cL(b'/A)=\tp_\cL(b/A)$. Note that the type described above is consistent, since $\cA$ is a $\BP$-substructure of $\cM$. As in the previous case, we want to show that 
\begin{equation*}
    p\coloneqq \tp_{\cL}(b\cup A/P(\cA))|P(\cM)=\tp_{\cL}(b'\cup A/P(\cM))
\end{equation*}
Let $\varphi(x,y)$ be a $P(\cM)$-formula in which $y$ is intended for elements in $A$. We may assume that $\varphi(x,y)$ is of the form $y+c\ \square \ x$ where $\square$ stands for $=,>$ or $<$ and $c\in P(\cM)$. Note that equality never happens and the truth of the formula does not depend on $c\in P(\cM)$. Thus we may assume the formula is of the form $y>x$ or $y<x$, which is completely determined by $\tp_\cL(b/A)=\tp_\cL(b'/A)$. 
\end{proof}

\subsection{Presburger arithmetic}

In this section we let $\cL$ be the language  of Presburger arithmetic $\cLpres\coloneqq\{+,-,<,0,1,(\equiv_n)_{n>1}\}$ and $T=\mathrm{PRES}$ be Presburger arithmetic. Recall that models of $T$ are also called $\Z$-groups.

\begin{fact}[\cite{f-generic_pres}]\label{thm:Zchar} Let $M$ be a $\Z$-group. Then, $p\in S_n(M)$ is definable if and only if for every realisation $a\models p$, $M$ is convex in $M(a)$.  \qed
\end{fact}  

It follows that the class of stably embedded pairs of models of $\mathrm{PRES}$ is $\cL_P$-elementary (see \cite{cubi-ye}). We will show that, in this case, the theory $T_{\SE}$ has two completions corresponding to $T_\BP$ and $T_\BP(\cK_\triv)$ respectively. Clearly the theory of trivial pairs is complete and $\cK_\triv$ has beauty transfer. In the non-trivial case, we have the following.

\begin{theorem}\label{thm:pres-complete} The theory $T_{\SE}\cup \{\exists x\neg P(x)\}$ is complete, axiomitizes $T_\BP$, and $\cK_\Def$ has beauty transfer. 
\end{theorem}

\begin{proof} We show that an $\aleph_1$-saturated model $\cM=(M,P(\cM))$ of $T_{\SE}\cup \{\exists x\neg P(x)\}$ is a beautiful pair. The proof follows the same strategy as the proof of Theorem \ref{thm:DOAG-complete}, and again, the result will follow using Corollary \ref{cor:elem-equiv}. 

Let $\cA$ be an $\cL_\BP$-substructure of $\cM$, and let $g\colon \cA \to \cB$ be an $\cL_\BP$-embedding with $P(\cA)=P(\cB)$, $B=\langle A\cup\{b\}\rangle$. By taking definable closure, we may assume that $P(\cA)\preccurlyeq A$. Up to interdefinability, we may assume $b$ is positive and $b$ must realize some cobounded type over $P(\cA)$ i.e., type at $+\infty$. Moreover, note that $b$ is not $\cL$-interdefinable over $A$ with any $b'$ with $b'$ bounded by elements in $P(\cA)$. 

Using saturation of $\cM$, we find an element $b'\in M$ such that
\[
b'\models p_{+\infty}|P(\cM) \cup \tp_\cL(b/A).
\]
It is now straightforward to check the following:
\[
p\coloneqq \tp_\cL(b\cup A/P(\cA))|P(\cM)=\tp_\cL(b'\cup A/P(\cM))
\]
This will imply that the map $b\mapsto b'$ induces an $\cL_\BP(\cA)$-embedding of $\cB$ into $\cM$. To illustrate one case, suppose $\varphi(x,y)$ is an $\cL(P(\cM))$-formula where $x$ is intended for $b/b'$ and $y$ is intended for elements in $A$ and $\varphi(x,y)$ is of the form $nx+y+c\equiv_m 0$ for $n, m\in \N$ and $c\in P(\cM)$. Let $l\in \mathbb{N}$ be such that $c\equiv_m l$. We see that $0\equiv_m nx+y+c \leftrightarrow 0\equiv_m nx+y+l$, and the truth of the right hand side is determined by $\tp_\cL(b/A)=\tp_\cL(b'/A)$. 
\end{proof}

\subsection{Dense regular ordered abelian groups}\label{sec:roag}
 
Recall that a (non-trivial) ordered abelian group is \emph{discrete} if it has a minimum positive element, and otherwise it is called \emph{dense}. We continue to work with $\cL=\cLpres$. Any ordered abelian group $M$ is viewed as an $\cL$-structure by interpreting $+,0,-,<, \equiv_n$ in the usual way, $1$ as the minimum positive element if $M$ is discrete and as $0$ when $M$ is dense. 

Recall that an ordered abelian group $M$ is called \emph{regular} if one of the following equivalent conditions holds:
 
 \begin{fact}[{\cite{robzak, zakon, conrad, weispfenning}}]\label{fact:regoag}
For an ordered abelian group $M$, the following are equivalent.
\begin{enumerate}[label=(\arabic*)]
\item The theory of $M$ has quantifier elimination in $\cL$.
   \item The only definable convex subgroups of $M$ are $\{0\}$ and $M$.
   \item The theory of $M$ has an archimedean model, i.e., one embeddable in $(\mathbb R,+,<)$ as an ordered abelian group.
   \item For every $n>1$, if the interval $[a,b]$ contains at least $n$ elements, then it contains an element divisible by $n$.
   \item Any quotient of $M$ by a nontrivial convex subgroup is divisible.\qed
   \end{enumerate}
 \end{fact}
The following fact characterizes all complete $\cL$-theories of regular ordered abelian groups.

\begin{fact}[{\cite{robzak, zakon}}]\label{fact:complregoag}
Every discrete regular ordered abelian group is a model of $\mathrm{PRES}$, i.e., it is a $\mathbb{Z}$-group. If $M, N$ are dense regular, then $M\equiv N$ if and only if for every prime $q$ the quotients $M/q M$ and $N/q N$ are either both infinite or have the same finite size. \qed
 \end{fact}

The following lemma characterizes definable types in regular ordered abelian groups. We leave the proof to the reader, pointing out that the proof follows a similar argument given by G. Conant and S. Vojdani~\cite{f-generic_pres} in the proof of Fact \ref{thm:Zchar}. 

Given an ordered abelian group $\Gamma$, we let $\Div(\Gamma)$ be the divisible hull of $\Gamma$ in an ambient divisible ordered abelian group.  

\begin{lemma}\label{lem:1types-roag}
Let $T$ be a complete theory of regular ordered abelian groups, and let $\Gamma\preccurlyeq \Gamma'\models T$ and $a=(a_1,\ldots, a_n)\in \Gamma'$. Then $\tp(a/\Gamma)$ is definable if and only if $\tp(b/\Gamma)$ is definable for every element  $b$ in $\langle \Gamma a\rangle$. Moreover, for every element $e\in \Gamma'$, $\tp(e/\Gamma)$ is definable if and only if $\tp_{\DOAG}(e/\Div(\Gamma))$ is definable. \qed
\end{lemma}

\begin{proposition}\label{prop:ROAG}
Let $\Gamma$ be an archimedean ordered abelian group. Then there is an archi\-me\-dean elementary extension $\Gamma\preccurlyeq\Gamma'$ such that $\Gamma'$ is stably embedded in every elementary extension.    
\end{proposition}

\begin{proof} We may suppose $\Gamma$ is dense, as for discrete $\Gamma$ we have $\Gamma\cong\Z$ and we can thus take $\Gamma'=\Gamma$. Without loss of generality, $\Gamma\leq (\R,+,<)$. Let $(a_i)_{i\in I}$ be a $\Q$-basis of $\Div(\Gamma)$. Let $(b_j)_{j\in J}$ be such that $(a_i, b_j)_{i\in I, j\in J}$ is a $\Q$-basis of $\R$. Let $\Gamma'$ be $\Gamma'=\Gamma \oplus (\bigoplus_{j\in J} \Q b_j$) equipped with the induced ordering by $\R$. The ordered abelian group $\Gamma'$ is archimedean and it is an elementary extension of $\Gamma$ by Facts \ref{fact:regoag} and \ref{fact:complregoag}. By construction, $\Div(\Gamma')=\R$, and therefore all types over $\Gamma'$ are definable by Lemma \ref{lem:1types-roag}. 
\end{proof}

\begin{proposition}\label{prop:reg-abel}
Let $T$ be the theory of a regular ordered abelian group which is neither discrete nor divisible. Then, there is a proper stably embedded pair $(\Gamma, P(\Gamma))$ of models of $T$ which is not convex. In addition, there is an elementary extension $(\Gamma', P(\Gamma'))$ of $(\Gamma, P(\Gamma))$ which is not stably embedded. In particular, the class of stably embedded pairs of models of $T$ is not elementary and $T$ does not have uniform definability of definable types.  
\end{proposition}

\begin{proof} The last part follows from the first part directly.
By Facts \ref{fact:regoag}, \ref{fact:complregoag} and Proposition~\ref{prop:ROAG}, there is a stably embedded pair $(\Gamma,P(\Gamma))$ of models of $T$ that is not convex and such that $\Gamma$ is $|P(\Gamma)|^+$-saturated (take as $P(\Gamma)$ an archimedean model of $T$ which is stably embedded in every elementary extension).  

In this proof, for every $a\in \Div(\Gamma)$, we write $p_{a^\pm}$ to denote the partial definable type in $T$ given by the corresponding definable quantifier-free $\DOAG$-type. Let $q$ be prime and $a\in P(\Gamma)$ so that $a$ is not $q$-divisible (which exists since $P(\Gamma)$ is not divisible). By saturation and Fact~\ref{fact:regoag}, there is $\varepsilon\in \Gamma$ such that $\varepsilon\models p_{0^+}|P(\Gamma)$ and $a+\varepsilon$ is divisible. Now, consider any non-principal ultrapower (with index set $\N$) $(\Gamma^*, P(\Gamma^*))$ of $(\Gamma, P(\Gamma))$. We claim that $(\Gamma^*, P(\Gamma^*))$ is no longer stably embedded. Indeed, consider the element $a'\in \Gamma^*$ represented by the sequence $((a+\varepsilon)/q^i)_{i\in \mathbb{N}}$ in $\Gamma$. Let us show that $p=\tp(a'/P(\Gamma^*))$ is not definable. By quantifier elimination in $\cL$, every definable set $X\subseteq \Gamma$ is the union of finitely many points and open intervals definable in the divisible hull of $\Gamma$ together with divisibility conditions. Since the cosets of $n\Gamma$ are dense, the convex hull of $X$ does not depend on the divisibility conditions. Hence, the only convex subsets definable in $P(\Gamma^*)$ are the ones with endpoints in the divisible hull of $P(\Gamma^*)$, working in an ambient model of $\DOAG$. If $p$ were definable, since $p$ is bounded, this would mean that there is $n\in \mathbb{N}$ and $b\in P(\Gamma^*)$ such that $p\vdash p_{(b/n)^+}|P(\Gamma^*)$ or $p\vdash p_{(b/n)^-}|P(\Gamma^*)$. Now if $b$ is represented by the sequence $(b_i)_{i\in \mathbb{N}}$, for almost every $i$ (with respect to the ultrafilter) $\tp((a+\varepsilon)/q^i)/P(\Gamma))\vdash p_{(b_i/n)^{\pm}}$. On the other hand, $\tp((a+\varepsilon)/q^i)/P(\Gamma))\vdash p_{(a/q^{i})^{\pm}}$, so $a/q^{i}=b_i/n$ for almost all $i$, contradicting the fact that $a$ is not $q$-divisible. 
\end{proof}

Let $T$ denote a complete theory of dense regular non-divisible ordered abelian groups. We denote by $\cK$ the class of \sepa-pairs $\cA=(A,P(\cA))$ such that $P(\cA)$ is convex in $A$ and for every prime $q$, the embedding $P(\cA)/qP(\cA)$ into $A/qA$ is a group isomorphism. It is not difficult to see that $\cK$ is a natural class. The proof of the following result is very similar to the one for Theorem~\ref{thm:pres-complete}. We leave the details to the reader.

\begin{theorem}\label{thm:ROAG-main} Let $T_{\conv}$ denote the theory of all proper convex elementary pairs $\cM$ of models of $T$. Then $T_{\conv}$ is complete and axiomatizes $T_\BP(\cK)$, and $\cK$ has beauty transfer. \qed
\end{theorem}

Note that when $\cM$ is a convex elementary pair of models of $T$, if follows from regularity that all cosets of $nM$ are represented in $P(\cM)$ for every $n>0$. 

Observe also that while $T$ does not have uniform definability of definable types by Proposition \ref{prop:reg-abel}, by Proposition \ref{prop:UDDT(F)}, the class $\cF$ of global cobounded definable types (i.e., $\cF_\cK$ as above) has uniform definability.

\section{Beautiful pairs of pure short exact sequences}\label{sec:SES}

We study in this section beautiful pairs of short exact sequences. Our motivation comes from valued fields, where we will apply the results of this section to the short exact sequence $1\rightarrow \RF^{\times}\rightarrow \RV^\times\rightarrow\VG\rightarrow 0$. Readers only interested in $\ACVF$, $\RCVF$, or $\PCF$ can skip this section and directly go to Section \ref{sec:domination}.    

The following elementary general fact---which is well known---will be used throughout. Its proof is is left to the reader.   

\begin{fact}\label{fact:def-intersection} The following holds for a complete theory $T$, $M\preccurlyeq\cU\models T$ and substructures $M\subseteq L,N\subseteq\cU$. Assume that $\tp(L/N)$ is $M$-definable (or it does not fork over $N$). Let $D$ be an $M$-interpretable set. Then $D(L)\cap D(N)=D(M)$. \qed
\end{fact}

Let us now recall the setting of \cite{AsChGeZi20}.  

\subsection{The setting}
In this section, we consider the complete theory of an $\{\bA\}$-$\{\bC\}$-enrichment\footnote{See \cite[Appendix A]{Rid17} for a formal definition of enrichment.} of a short exact sequence $M$ of abelian 
groups 
\[
0\longrightarrow \bA(M) \overset{\iota}{\longrightarrow} \bB(M) \overset{\nu}{\longrightarrow} \bC(M) \longrightarrow 0
\]

with $\bA(M)$ a pure subgroup of $\bB(M)$. Note that $\bA$ and $\bC$ are stably embedded, orthogonal and pure sorts in $M$. Indeed, as the sequence splits in any $\aleph_1$-saturated model (this follows from the fact that any $\aleph_1$-saturated abelian group is pure injective, see \cite[Theorems 10.7.1 and 10.7.3]{hodges}), $\Th(M)$ is a reduct of $\Th(\bA(M)\times \bC(M))$, and these same properties hold in the product structure. 

We need to consider various languages in what follows. Let $\cL_a=\{0,+,-\}$ be the language of abelian groups on sort $\bA$, and let $\cL_b$ and $\cL_c$ be similarly defined. Let $\cL_a^*$ be a relational enrichment of $\cL_a$ on the sort $\bA$, $\cL_c^*$ be a relational enrichment of $\cL_c$ on the sort $\bC$, and $T_a\coloneqq \Th_{\cL_a}(\bA(M))$, $T_a^*\coloneqq \Th_{\cL_a^*}(\bA(M))$, and let $T_c$ and $T_c^*$ be similarly defined (note that arbitrary enrichments of $\cL_a^*$ and $\cL_c^*$ may always be replaced by their Morleysation). 

Let $\cL_{abc}$ be the three sorted language $\cL_a\cup \cL_b\cup \cL_c\cup \{\iota,\nu\}$, and $\cL_{abc}^*=\cL_{abc}\cup \cL_a^*\cup \cL_c^*$. Let $T_{abc}=\Th_{\cL_{abc}}(M)$ and $T^*_{abc}=\Th_{\cL^*_{abc}}(M)$. In \cite{AsChGeZi20} the authors work in $\cL_{abcq}$, an expansion by definition (with new sorts) of $\cL_{abc}$. Namely, it is defined by adding new sorts for $\bA/n\bA$ for all $n\in\mathbb{N}$ (identifying $\bA/0\bA$ with $\bA$), as well as, for every natural number $n$, unary functions $\pi_n\colon \bA\rightarrow \bA/n\bA$ (denoting the canonical projection) and  $\rho_n\colon \bB\rightarrow \bA/n\bA$ which, on $\nu^{-1}(n\bC)$, is the composition of the group homomorphisms
\[
\nu^{-1}(n\bC)=n\bB+\iota(\bA)\rightarrow (n\bB+\iota(\bA))/n\bB\cong\iota(\bA)/ (n\bB\cap\iota(\bA))\cong \bA/n\bA,
\] 
and zero outside $\nu^{-1}(n\bC)$. We additionally equip in $\cL_{abcq}$ each sort $\bA/n\bA$ with the language of abelian groups.\footnote{Note that in \cite{AsChGeZi20} the abelian group on the sorts $\bA/n\bA$ is not part of the language $\cL_{abcq}$.} 

Let $\cL^*_{abcq}=\cL^*_{abc}\cup \cL_{abcq}$, and set 
$T_{abcq}\coloneqq \Th_{\cL_{abcq}}(M)$ and $T^*_{abcq}\coloneqq \Th_{\cL^*_{abcq}}(M)$. 

Finally, let $\cL_{aq}$ and $\cL^*_{aq}$ be the expansions of $\cL_a$ and $\cL^*_a$ obtained by adding for all $n$ the sorts $\bA/n\bA$ with the language of abelian groups and the maps $\pi_n$, and set $T_{aq}\coloneqq\Th_{\cL_{aq}}(\bA(M))$ and $T^*_{aq}\coloneqq\Th_{\cL^*_{aq}}(\bA(M))$.

In what follows we denote by $\bbA$ the family of sorts 
$\{\bA/n\bA\}_{n\in\mathbb{N}}$. The proof of \cite[Corollary~4.3]{AsChGeZi20} gives the following variant where we allow variables from all sorts from $\bbA$ and where we use that in our setting we only consider $\{\bA\}$-$\{\bC\}$-enrichments. 

\begin{fact}\label{F:QE-SES}
Every $\cL_{abcq}^*$-formula $\phi(x_a,x_b,x_c)$, with $x_a, x_b, x_c$ variables from sorts $\bbA$, $\bB$ and $\bC$, respectively, is equivalent in  $T^*_{abcq}$ to a boolean combination of formulas of the following forms: 
\begin{enumerate}[(i)]
    \item $\psi_a(x_a,\rho_{n_1}(t_1(x_b)),\ldots,\rho_{n_l}(t_l(x_b)))$ where $\psi_a(x_a,y_1,\ldots,y_l)$ is an $\cL^*_{aq}$-formula, each  $t_i(x_b)$ is an $\cL_b$-term and $n_1,\ldots,n_l\in\mathbb{N}$,
   \item  $\psi_c(\nu(t_1(x_b)),\ldots,\nu(t_l(x_b)),x_c)$ where $\psi_c(y_1,\ldots,y_l,x_c)$ is an $\cL^*_{c}$-formula and where each $t_i(x_b)$ is an $\cL_b$-term.\qed
\end{enumerate}
\end{fact}

\subsection{Beautiful pairs}\label{subsec:bp-exact}

Throughout this subsection we let $\cL$ be $\cL_{abcq}^*$. Given $M$ and $N$ two superstructures of $K$, we sometimes write $K\langle M,N\rangle$ instead of $\langle M, N\rangle$ to emphasize that $K$ is a common substructure of $M$ and $N$.

\begin{lemma}\label{L:12-11for-SES}
Let $M_0\preccurlyeq M\preccurlyeq \cU$ be such that $\cU$ is a sufficiently saturated and homogeneous model of $T_{abcq}^*$, and let $M_0\subseteq N\subseteq \cU$ be an $\cL$-substructure of $\cU$. Assume the following holds:
\begin{enumerate}[label=(\arabic*)]
    \item $\bA/n\bA(N)\cap \bA/n\bA(M)=\bA/n\bA(M_0)$ for any $n\in\mathbb{N}$, and 
    \item $\tp_{\cL_c}(\bC(N)/\bC(M))$ does not fork over  $\bC(M_0)$.
\end{enumerate}
Let $\sigma$ and $\tau$ be two automorphisms of $\cU$ over $M_0\bbA(M)\bC(M)$. Then, there is an automorphism $\eta$ of $\cU$ mapping 
$M_0\langle M,N\rangle$ to $M_0\langle \sigma(M),\tau(N)\rangle$ such that  $\eta_{|M}=\sigma_{|M}$ and $\eta_{|N}=\tau_{|N}$. Furthermore $\bbA(\langle M, N\rangle)=\langle 
\bbA(M),\bbA(N)\rangle$ and $\bC(\langle M, N\rangle)=\langle 
\bC(M), \bC(N)\rangle$. 
\end{lemma}

\begin{proof}
Without loss of generality, we may assume that $\tau$ is the identity. Suppose first that $N=\langle\bB(N)\rangle$.

\begin{claim}\label{cla:intersection} $\bB(N)\cap \bB(M)=\bB(M_0)$, equivalently $\bB(N)+\bB(M)$ is naturally isomorphic to $\bB(N)\oplus_{\bB(M_0)}\bB(M)$. 
\end{claim}

 Consider $\beta\in \bB(N)\cap \bB(M)$. As $\nu(\beta)\in \bC(N)\cap \bC(M)=\bC(M_0)$ (by assumption (2)), there is $\beta_0\in \bB(M_0)$ such that 
$\nu(\beta)=\nu(\beta_0)$, i.e., $\beta-\beta_0\in\bA(N)\cap \bA(M)=\bA(M_0)$, 
so finally $\beta\in \bB(M_0)$. This completes the proof of the claim.  

\medskip

\begin{claim}\label{cla:generation} For any $n\in\mathbb{N}$ we have $\bA/n\bA(\langle M,N\rangle)=\bA/n\bA(M)+\bA/n\bA(N)$, where the  latter is isomorphic to $\bA/n\bA(M)\oplus_{\bA/n\bA(M_0)}\bA/n\bA(N)$.  
\end{claim}

It is easy to see that 
$\bA/n\bA(\langle M, N\rangle)$ is generated, as an abelian group, by the subset $\bA/n\bA(M)\cup \bA/n\bA(N)\cup\rho_n(\bB(M)+\bB(N))$, and so it suffices to show that for any $b\in \bB(M)$ and $d\in \bB(N)$ we have $\rho_n(b+d)\in \bA/n\bA(M))+\bA/n\bA(N)$. We may assume that $\nu(b+d)\in n\bC(\cU)$. By assumption (2) and stability in $\cL_c$, $\tp_{\cL_{c}}(\nu(b)/\bC(N))$ is finitely satisfiable in $\bC(M_0)$. Then, there is $d_0\in \bB(M_0)$ such that $\nu(b+d_0)=\nu(b)+\nu(d_0)\in n\bC(\cU)$. Thus $\rho_n(b+d)=\rho_n((b+d_0)+(d-d_0))=\rho_n(b+d_0)+\rho_n(d-d_0)\in \bA/n\bA(M)+\bA/n\bA(N)$. This completes the proof of the claim.  
\medskip

Note that $M$ is also $\bB$-generated, since it is a model. It follows that $\langle M,N\rangle$ is also $\bB$-generated. Moreover, since $\bB(\langle M,N\rangle)= \bB(M)+ \bB(N)$, the map $\eta\colon\bB(M)+\bB(N)\to \cU$ given by $\eta(b_1+b_2)=\sigma(b_1)+b_2$ is well-defined by Claim \ref{cla:intersection} and the fact that $\tau$ is the identity on $M_0$. This map extends to a homomorphism on $\langle \bB(M), \bB(N)\rangle=\langle M,N\rangle$. The proof of Claim \ref{cla:generation} shows moreover that $\eta$ respects $\rho_n$ for every $n\geqslant 0$. By Fact \ref{F:QE-SES}, this shows that $\eta$ is an elementary map and hence, by homogeneity, it can be lifted to the desired automorphism.

\medskip

Let us now give the argument for a general $N$. Assumptions (1) and (2) for $N$ imply the same assumptions for $\widetilde{N}=\langle \bB(N)\rangle$. Therefore there is $\rho\in \Aut(\cU)$ mapping $M_0\langle M,\widetilde{N}\rangle$ to $M_0\langle \sigma(M),\widetilde{N}\rangle$ such that  $\rho_{|M}=\sigma_{|M}$, $\rho_{|\widetilde{N}}=\id_{|\widetilde{N}}$,  $\bbA(\langle M, \widetilde{N}\rangle)=\langle 
\bbA(M),\bbA(\widetilde{N})\rangle$ and $\bC(\langle M, \widetilde{N}\rangle)=\langle 
\bC(M), \bC(\widetilde{{N}})\rangle$. Every element in $N\setminus \widetilde{N}$ is either in $\bbA$ or $\bC$. Therefore $\rho_{|M_0\langle M,\widetilde{N}\rangle}$ together with the identity on $N$ induces an $\cL$-isomorphism mapping $M_0\langle M,N\rangle$ to $M_0\langle \sigma(M),N\rangle$. Hence, by homogeneity, there is $\eta\in\Aut(\cU)$ extending such a map. The last statement of the proposition follows from the assumptions, together with the corresponding statement for $\widetilde{N}$.  
\end{proof}

The following is a domination result for short exact sequences. Part (1) slightly generalizes \cite[Proposition 2.12]{touchard}.

\begin{lemma}\label{C:SES-beauty-extension}
Let $M_0\preccurlyeq \cU\models T^*_{abcq}$ and let $M_0\subseteq N\subseteq \cU$, where $N$ is an $\cL^*_{abcq}$-substructure of $\cU$. Then we have:
\begin{enumerate}[label=(\arabic*)]
    \item $\tp_{\cL^*_{abcq}}(N/M_0)$ is a definable type  if and only if the types  $\tp_{\cL^*_{aq}}(\mathcal{\bbA}(N)/\bbA(M_0))$ and $\tp_{\cL^*_c}(\bC(N)/\bC(M_0))$ are both definable;
    \item assuming that $\tp_{\cL_{abcq}^*}(N/M_0)$ is definable and $M_0\preccurlyeq M$, we have 
\[
    \left(\begin{array}{l} \tp_{\cL_{abcq}}(N/M_0)\cup\tp_{\cL_{aq}^*}(\bbA(N)/ \bbA(M_0))\mid \bbA(M) \\
\cup\tp_{\cL_{c}^*}(\bC(N)/\bC(M_0))\mid \bC(M)  \end{array}\right)
\vdash \tp_{\cL_{abcq}^*}(N/M_0)\mid M.
\]
\end{enumerate}
\end{lemma}

\begin{proof}
For (1), from left-to-right, the result follows from the fact that $\bbA(M_0)$ is stably embedded in $M_0$ with induced structure given by $\cL^*_{aq}$ and the fact that $\bC(M_0)$ is stably embedded in $M_0$ with induced structure given by $\cL^*_{c}$. For the converse, first note that $T_{abcq}$ is stable\footnote{The theory of abelian groups is stable, and $T_{abcq}$ is intrepretable in the product $\bA\times \bC$ with the projections named. Moreover, a product of stable structures is stable.}, so $\tp_{\cL_{abcq}}(N/M_0)$ is a definable type. Then, by Fact~\ref{F:QE-SES}, $\tp_{\cL^*_{abcq}}(N/M_0)$ is implied by $\tp_{\cL_{abcq}}(N/M_0)\cup \tp_{\cL^*_{aq}}(\bbA(N)/\bbA(M_0))\cup \tp_{\cL^*_c}(\bC(N)/\bC(M_0))$, which shows the result.  

\medskip

For part (2), for any two realisations $N_1,N_2$ of the left-hand side of the condition in (2), let $\tau$ be an automorphism of $\cU$ over $M_0\bbA(M)\bC(M)$ sending $N_1$ to $N_2$, and let $\sigma$ be the identity of $\cU$. Then Lemma \ref{L:12-11for-SES} gives $\eta$, an automorphism of $\cU$ fixing $M$ sending $N_1$ to $N_2$. 
\end{proof}

Suppose that $\cK_{\mathscr{A}}$ is a natural class of \sepa-pairs of $T^*_{aq}$, satisfying the extension property, such that $\cK_{\mathscr{A}}$-beautiful pairs exist and $\cK_{\mathscr{A}}$ has beauty transfer. Similarly, suppose that $\cK_C$ is a natural class of \sepa-pairs of $T^*_{c}$, satisfying extension, such that $\cK_C$-beautiful pairs exist and $\cK_C$ has beauty transfer. Let $\cK$ be the subclass of $\cK_\Def$ induced by $\cK_{\mathscr{A}}$ and $\cK_C$, i.e., the corresponding pairs lie in $\cK_{\mathscr{A}}$ and $\cK_C$. It is not difficult to see $\cK$ is a natural class (for closure under base extension, one uses the furthermore part of Lemma \ref{L:12-11for-SES}). 

\begin{theorem}\label{thm:exact}
Let $\cK_\mathscr{A}$, $\cK_C$ and $\cK$ be as defined above. Then, $\cK$ satisfies the extension property and $\cK$-beautiful pairs exist. Moreover, $\cK$ has beauty transfer and $T_\BP(\cK)$ is axiomatized by the following conditions on an $\cL_P$-structure $\cM=(M,P(\cM))$ 
\begin{itemize}
    \item $P(\cM)\preccurlyeq M\models T_{abcq}^*$
    \item $(\bbA(M),\bbA(P(\cM)))\models T_\BP(\cK_{\mathscr{A}})$ and $(\bC(M),\bC(P(\cM)))\models T_\BP(\cK_C)$.
\end{itemize}
\end{theorem}

\begin{proof}
Let $T$ be the list of axioms given in the statement. Note that $T$ is consistent, simply by taking the product of a model of $T_\BP(\cK_{\mathscr{A}})$ and a model of $T_\BP(\cK_C)$. So let $\cM$ be an $|\cL|^+$-saturated model of $T$. By Lemma \ref{lem:beauty-transfer-cardinal}, it suffices to show $\cM$ is an  $|\cL|^+$-$\cK$-beautiful pair. So let $\cA\to \cM$ and $\cA\to \cB$ be $\BP$-embeddings with $|\cB|\leq|\cL|$. 

\medskip

\emph{Step 0.} Without loss of generality, we may suppose $\cA\subseteq \cM$, and by Lemma \ref{lem:b-pair-finite} and Lemma-definition~\ref{lem:base-change}, we may further assume $P(\cA)=P(\cB)$. 

\medskip

\emph{Step 1.} Since $\cB\in \cK$, we have $\bbA(\cB)=(\bbA(B), \bbA(P(\cB)))\in \cK_\mathscr{A}$ and similarly $\bC(\cB)=(\bC(B),\bC(P(\cB)))\in \cK_C$. By beauty transfer of $\cK_{\mathscr{A}}$ and beauty transfer of $\cK_C$, $\bbA(\cM)=(\bbA(M),\bbA(P(\cM)))$ is $
|\cL|^+$-$\cK_\mathscr{A}$-beautiful and $\bC(\cM)=(\bC(M),\bC(P(\cM)))$ is $|\cL|^+$-$\cK_C$-beautiful. In particular, there is an $(\cL_{aq}^*)_{\BP}$-embedding $g\colon \bbA(\cB)\to \bbA(\cM)$ over $\bbA(\cA)$, and similarly, there is an $(\cL_c^*)_{\BP}$-embedding $h\colon \bC(\cB)\to \bC(\cM)$ over $\bC(\cA)$.

\medskip

\emph{Step 2.} By quantifier elimination (Fact \ref{F:QE-SES}) in $\cL$ and $|\cL|^+$-saturation of $\cM$, there is an $\cL$-embedding $f\colon B\to M$ over $A$ inducing the maps $h$ and $g$ from Step 1. Note that since $P(\cB)$ is a model, the restriction of $f$ to $P(\cB)=P(\cA)$ induces a 
$\BP$-embedding. Moreover, we have that
\begin{align*}
& f(B)\models \tp(B/P(\cA)) \\ 
& g(\bbA(B))\models  \tp(\bbA(B)/\bbA(P(\cA)))\mid\bbA(P(\cM)) \\ 
& h(\bC(B))\models  \tp(\bC(B)/\bC(P(\cA)))\mid \bC(P(\cM)).  
\end{align*}
Part (2) of Lemma \ref{C:SES-beauty-extension} yields that $f(B)\models \tp(B/P(\cA))\mid P(\cM)$, which shows that $f$ is a $\BP$-embedding (the predicate $P$ is preserved by Lemma \ref{L:12-11for-SES}, more precisely using Claim \ref{cla:intersection}).  
\end{proof}

\begin{remark}[Application to ordered abelian groups]
By Theorem~\ref{thm:exact} we obtain characterisations of beautiful pairs of products of regular ordered abelian groups with the lexicographic order. In particular, we obtain beauty transfer for classes in $\Th(\Z^n)$ and $\Th( \Z^n\times \Q)$ corresponding to combinations of the possibilities in Theorem~\ref{thm:DOAG-complete} and Theorem~\ref{thm:pres-complete}. The results of the current section apply, as one may use induction on $n$ and the fact that in both $\Z^n$ and $\Z^n\times \Q$, all convex subgroups are definable.
\end{remark}

\section{Domination in valued fields}\label{sec:domination}

After setting up the notation and terminology on valued fields that will be used for the rest of the paper, we will prove some key domination results building on and generalizing~\cite[Proposition 12.11]{HHM} and~\cite[Theorem 2.5]{has-ealy-mar} (see subsection~\ref{sec:dominationII}). Some of these results have been recently obtained independently by M. Vicaria in \cite{vicaria}. Slight variants also appear in \cite{hils-mennuni} by the second author and R. Mennuni. For the reader's convenience we decided to keep all of our proofs.

\subsection{Languages and theories of valued fields}\label{sec:languages}

For a valued field $(K,v)$, we let $\VG(K)$ denote its value group $v(K^\times)$, $\mathcal{O}(K)$ its valuation ring, $\cM(K)$ its maximal ideal, $\RF(K)$ its residue field and $\res\colon \mathcal{O}(K)\to \RF(K)$ the residue map. By default, unless otherwise stated, by a valued field extension $K\subseteq L$ we implicitly mean $(K, v)\subseteq (L,v)$ and always use $v$ for the underlying valuation. When working with multi-sorted structures, we will also use the notation $\VF(K)$ to denote the field $K$. 

\medskip

We will work in the following languages and theories of valued fields. 
\begin{itemize}[leftmargin=*]
    \item The 3-sorted language of valued fields $\cLgk$ has sorts $\VF, \RF$, and $\VGi$, where $\VF$ and $\RF$ are equipped with $\cLr$, $\VGi$ is equipped with $\cLog$ together with a new constant symbol for $\infty$, and we have additional function symbols for the valuation $v\colon \VF\rightarrow\VGi$ and the following residue map $\Res\colon \VF^2\rightarrow \RF$ interpreted in a valued field $(K,v)$ by     
\[
\Res(x,y)\coloneqq
\begin{cases}
\res(\frac{x}{y}) & \text{ if } \infty\neq v(y)\leqslant v(x);\\
0 & \text{otherwise.}
\end{cases}
\]
We let $\VG$ be $\VGi\setminus \{\infty\}$. For the remaining of the paper, unless otherwise stated, $\ACVF$ will denote the complete $\cLgk$-theory of an algebraically closed non-trivially valued field.     
\item The 3-sorted language of ordered valued fields $\cLovf$ corresponds to $\cLgk$ together with additional binary predicates for the order in both $\VF$ and $\RF$. As usual, $\RCVF$ is the $\cLovf$-theory of a real closed non-trivially valued field with convex valuation ring.

\item The 3-sorted language $\cLpcf(e,f)$ of $p$-adically closed fields with $p$-ramification index $e$ and residue degree $f$ corresponds to $\cLgk$ where on $\VF$ we replace $\cLr$ by Macintyre's language $\cLmac$ (i.e., we add predicates $P_n$ for the $n^\mathrm{th}$-powers) together with $d=ef$ new constant symbols, and on $\VGi$ we put Presburger's language $\cLpres$ together with a new constant symbol for $\infty$. If $F$ is a finite extension of $\Q_p$ with $p$-ramification index $e$ and residue degree $f$, the new constants are interpreted by elements in $\mathcal{O}$ whose residues in $\mathcal{O}/p\mathcal{O}$ form an $\mathbb{F}_p$-basis. We let $\PCF$ denote the complete $\cLpcf(e,f)$-theory of a finite extension of $\Q_p$ of $p$-ramification index $e$ and residue degree $f$.  

\item We will now consider $(K,v)$ in the leading term language $\mathcal{L}_{\RVV}$ of Flenner \cite{fle-2011}, consisting of a sort $\VF$ for the valued field (endowed with the ring language), a sort $\RV$ for the union of $\{0\}$ and the quotient group $K^{\times}/1+\cM(K)$ and the quotient map $\rv\colon \VF\to\RV$ (extended by $0\mapsto 0$). We endow $\RV$ with the group language and a ternary relation $\oplus$ interpreted as $\oplus(a_1,a_2,a_3)$ for $a_i\in \RV$ if and only if there exist $x_i\in \VF$ such that $\rv(x_i)=a_i$ for $i=1,2,3$ and $x_1+x_2=x_3$. We let $\RV^\times$ denote $\RV\setminus\{0\}$. By an $\RV$-enrichment of $\cL_\RVV$, we mean a language extending $\cL_\RVV$ which may add new sorts but only adds functions, relations and constant symbols to $\RV$ and the new sorts. See \cite[Appendix A]{Rid17} for a formal definition. Abusing notation, given an $\RV$-enrichment $\cL$ of $\cL_\RVV$ and an $\cL$-structure $M$, we continue to write $\RV(M)$ for the union of the $M$-points of $\RV$ together with the $M$-points of all new sorts from $\cL$. 

\item We will denote by $\mathcal{L}_{\RVV,\Gamma}$ the language $\mathcal{L}_{\RVV}$ augmented by a sort $\VGi$ (endowed with $\cLog$ together with a constant symbol for $\infty$) for the value group and a function symbol for the quotient map $v_{\rv}\colon\RV\rightarrow\VGi$ so that $v=v_\rv\circ\rv$. 

\item Recall that an \emph{angular component} map (also called \emph{ac} map) is a multiplicative group homomorphism $\ac\colon \VF^{\times}\rightarrow \RF^{\times}$ extended to $\VF$ by setting $0\mapsto 0$ and such that $\ac(a)=\res(a)$ for any $a$ with $v(a)=0$. The language $\mathcal{L}_{\Pas}$, introduced in~\cite{pas-89}, is the 3-sorted language with sorts 
$\VF$ (endowed with $\mathcal{L}_{\Ring}$), $\VGi$ (endowed with $\cLog$ together with a constant symbol for $\infty$) and $\RF$ (endowed with $\mathcal{L}_{\Ring}$) and the maps $v\colon\VF\rightarrow\VGi$ and $\ac\colon\VF\rightarrow\RF$ between the sorts. Any $\aleph_1$-saturated valued field may be endowed with an $\ac$ map (this follows as for short exact sequences), and the $\mathcal{L}_{\Pas}$-structure associated to a valued field $(K,v)$ is quantifier-free bi-interpretable with the $\mathcal{L}_{\RVV}$-structure associated to $(K,v)$ enriched by a splitting of the short exact sequence 
\[
1\rightarrow \RF^{\times}\rightarrow \RV^\times\rightarrow\VG\rightarrow 0,
\]
i.e., by a group homomorphism $s\colon \RV^\times\rightarrow \RF^{\times}$ such that $s\circ\iota$ is the identity on $\RF^{\times}$. By a $\VG$-$\RF$-enrichment of $\cL_\Pas$ we mean a language extending $\cL_\Pas$ which may add functions, relations and constant symbols to $\VG_\infty$ and $\RF$, separately. Note that, contrarily to $\RV$-enrichments, here we do not allow new sorts.  
\item The language $\cL^\mathcal{G}$ extends the 3-sorted language $\cLgk$ language with the following ``geometric sorts'' introduced in~\cite{HHM1}.
\begin{enumerate}[label=(\arabic*)]
  \item For each natural number $n$, we add a sort $\mathbf{S}_{n}$ which consists of the codes of the $\mathcal{O}$-sublattices of $\VF^{n}$ (i.e. the free $\mathcal{O}$-submodules of $\VF^{n}$ on $n$ generators).
  \item For any $s \in \mathbf{S}_{n}$, we define $\mathrm{red}(s)=s/\mathcal{M}s$ (the reduction of $s$ modulo $\mathcal{M}$). And we let $\mathbf{Lin}_{n}=  \bigcup \{ \mathrm{red}(s) \ | \ s \in \mathbf{S}_{n}\}$. 
  \end{enumerate}
We write ${\mathcal{G}}$ to denote the collection of sorts $\displaystyle{\VF\cup  \bigcup_{n \in \mathbb{N}}\mathbf{S}_{n} \cup \bigcup_{n \in \mathbb{N}} \mathbf{Lin}_{n}}$.
\end{itemize}

\begin{definition}\label{def:benign}
We will call a valued field $(K,v)$ \emph{benign} if its theory $T$ in $\mathcal{L}_{\RVV}$ admits elimination of $\VF$-quantifiers and if any model $M\models T$ admits a unique maximal immediate extension $M'$ (up to isomorphism) such that $M\preccurlyeq M'$.
\end{definition}

\begin{fact}\label{F:ex-benign}
The following valued fields are benign:
\begin{enumerate}[label=(\arabic*)]
    \item Henselian valued fields of equicharacteristic 0,
    \item algebraically maximal Kaplansky valued fields of positive characteristic,
\item algebraically closed valued fields 
\end{enumerate}
\end{fact}

\begin{proof}
For elimination of $\VF$-quantifiers in case (1) see \cite{fle-2011}. Cases (2) and (3) are shown in \cite{HaHa19}. The uniqueness of the maximal immediate extension is well known in these cases (see \cite{kap42}). The extension is elementary as, in these cases, the maximal immediate extension is also a model of $T$.    
\end{proof}

\begin{remark}
In \cite{touchard}, Touchard defines benign valued fields as the valued fields belonging to the list given in Fact~\ref{F:ex-benign}, although he works axiomatically, and the results he obtains work with the more general definition we are adopting. For this reason, we decided to stick to the same terminology. 
\end{remark}

\medskip

By syntactic considerations (see \cite[Proposition~A.9]{Rid17}), elimination of $\VF$-quantifiers in benign valued fields yields the following.

\begin{fact}\label{F:Benign-facts}
Let $(K,v)$ be a benign valued field. Then the following holds.
\begin{enumerate}[label=(\arabic*)]
    \item Any $\RV$-enrichment of $(K,v)$ admits elimination of $\VF$-quantifiers, and the enrichment of $\RV$ is purely stably embedded.
    \item Any $\VG$-$\RF$-enrichment of the $\mathcal{L}_{\RVV,\Gamma}$-structure associated to $(K,v)$ admits elimination of $\VF$-quantifiers, and $\RF$ and $\VGi$ are purely stably embedded and orthogonal.  
    \item Any  $\VG$-$\RF$-enrichment of the $\cL_\Pas$-structure associated to $(K,v)$ admits elimination of  $\VF$-quantifiers, and $\RF$ and $\VGi$ are purely stably embedded and orthogonal. \qed
\end{enumerate} 
\end{fact}

\begin{remark}
    Note that in the previous fact, when we say \emph{purely stably embedded} we mean stably embedded with the natural $\emptyset$-induced structure. When there is no enrichment, in the case of the residue field, this means the ring language, in the case of the value group this means the language of ordered abelian groups (with a constant for $v(p)$ in mixed characteristic $(0,p)$) and in the case of $\RV$ the structure given by the short exact $\RV$-sequence (with a constant for $\rv(p)$ in mixed characteristic $(0,p)$) with the ring language in the residue field and the language of ordered abelian groups in the value group. In the presence of an enrichment, this enrichment will give the $\emptyset$-induced structure in each case.
\end{remark}

The following fact goes back to results of \cite{Rob77}, \cite{cherlin_dickmann} and \cite{pre-ro-84}. 

\begin{fact}\label{fact:QEVF} The following theories have quantifier elimination:
\begin{enumerate}[label=(\arabic*)]
    \item the theory of algebraically closed valued fields $\ACVF$ in $\cLgk$;
    \item the theory of real closed valued fields $\RCVF$ in $\cLovf$;
    \item the theory of a finite extension of $\Q_p$ with $p$-ramification index $e$ and residue degree $f$ $\PCF$ in $\cLpcf(e,f)$.  
    \end{enumerate}
In addition, in each theory, $\VGi$ and $\RF$ are orthogonal and purely stably embedded. \qed
\end{fact}

\subsection{Some lemmas about valuation independence}\label{sec:val-indipendence}

For $K\subseteq K',L$ subfields of a common extension, we write $L\ind_K^{\ld} K'$ if $L$ is linearly disjoint from $K'$ over $K$.

\begin{fact}\label{fact:def-field} Let $T$ be complete theory, $M\preccurlyeq\cU\models T$ and substructures $M\subseteq L,N\subseteq\cU$. Assume that $\tp(L/N)$ is $M$-definable. Let $k$ be an $M$-interpretable field, and assume that $k(L)$ and $k(N)$ are fields. Then $k(L)\ind_{k(M)}^{\ld} k(N)$. \qed
\end{fact}

\begin{proof}The above fact is presumably well-known, but we include a short proof for the sake of completeness. Assume to the contrary, then there are $c_1,\ldots,c_n \in k(L)$ that are linearly independent over $k(M)$ but not over $k(N)$. Then $\{(a_1,...,a_n)\in \cU^n: a_1c_1+\cdots+ a_nc_n=0\}$ is a non-trivial $M$-definable subspace of $\cU^n$. Consequently, it has a non-zero $M$-point, which is a contradiction. 
\end{proof}

Given a valued field extension $K\subseteq L$ and a subset $A\coloneqq \{a_1,\ldots,a_n\}\subseteq L$, we say that $A$ is \emph{$K$-valuation independent} if $v(\sum_{i=1}^n c_ia_i) = \min_{i}(v(c_ia_i))$ for every $K$-linear combination $\sum_{i=1}^n c_ia_i$. The extension $L/K$ is called \emph{$vs$-defectless}\label{VS-def}\footnote{This is the same as ``separated'' in W. Baur's and F. Delon's terminology.} if every finitely generated $K$-vector subspace $V$ of $L$ admits a \emph{$K$-valuation basis}, that is, a $K$-valuation independent set which spans $V$ over $K$. The following is a valuative analogue notion of linear disjointness. 

\begin{definition}[Valuative disjointness]\label{def:valuative-disjoint}
Let $K\subseteq K', L$ be valued subfields of a common valued field. Suppose the extension $L/K$ is $vs$-defectless. We say \emph{$L$ is valuatively disjoint from $K'$ over $K$}, denoted by $L\ind_K^{\vd} K'$, if every $K$-valuation basis of a finite dimensional $K$-vector subspace in $L$ is also $K'$-valuation independent. 
\end{definition}

Observe that $L\ind_K^{\vd} K'$ implies in particular that $L\ind_K^{\ld} K'$. 

\begin{definition}
Let $K\subseteq L$ be a valued field extension and  $B$ be a subset of $L$. We say $B$ is a \emph{normalized $K$-valuation independent set} if it satisfies  
\begin{enumerate}
\item[(N1)] for every $b,b'\in B$, if $v(b)$ and $v(b')$ lie in the same coset modulo $\VG(K)$, then $v(b)=v(b')$;
\item[(N2)] for every $b\in B$, the set $\{\res(b'/b)\ | \ b'\in B \textrm{ and } v(b')=v(b)\}$ is $\RF(K)$-linearly independent;
\item[(N3)] if $b\in B$ and $v(b)\in \VG(K)$, then $v(b)=0$;
\item[(N4)] if $b\in B$ and $\res(b)\in \RF(K)$, then $\res(b)=1$.
 \end{enumerate}
\end{definition}

Note that, as the name indicates, a normalized $K$-valuation independent set is $K$-valuation independent. 

\begin{lemma}[{\cite[Lemma 2.24]{bla-cubi-kuh}}]\label{lem:normalized} Let $K\subseteq L$ be an extension of valued fields, and let $B = \{b_1,\ldots, b_n\}$ be a $K$-valuation independent set. Then there is a set $\{c_i \in K^\times : 1\leqslant i\leqslant n\}$  such that $B'\coloneqq \{c_ib_i : 1\leqslant i\leqslant n\}$ is a normalized K-valuation independent set. \qed
\end{lemma}

The following fact follows directly from the definition of valuation independence. 

\begin{fact}\label{fact:rv-sum} Let $K\subseteq L$ be valued fields. Suppose $\{c_1,\ldots, c_n\}\subseteq L$ is $K$-valuation independent. Then for $a_i\in K$
\begin{flalign*}
&& \rv\left(\sum_{i=1}^n c_ia_i\right)= \bigoplus_{i=1}^n \rv(c_i)\rv(a_i).
&& \qed
\end{flalign*}
\end{fact}

\begin{lemma}\label{lem:vd-equiv}Let $K\subseteq K', L$ be valued subfields of a common valued field. Suppose that $L/K$ is $vs$-defectless, $\VG(L)\cap \VG(K')=\VG(K)$ and $\RF(L)\ind_{\RF(K)}^{\ld} \RF(K')$. Then, $L\ind_{K}^{\vd} K'$. Moreover, $\VG(LK')= \VG (L)+ \VG(K')$, $\RF(LK')$ is the field compositum of $\RF(L)$ and $\RF(K')$ and $LK'/K'$ is $vs$-defectless. In particular, $\RV(LK')$ is generated by $\RV(L)$ and $\RV(K')$ (i.e., $\RV(LK')$ is the smallest subset containing both $\RV(L)$ and $\RV(K')$ which is closed under multiplication, multiplicative inverse, and under $\oplus$ seen as a partial function).  
\end{lemma}

\begin{proof} Let $B=\{b_1,\ldots, b_n\}$ be a $K$-valuation basis of a $K$-vector subspace of $L$. By Lemma~\ref{lem:normalized}, we may suppose $B$ is a normalized $K$-independent set. Suppose for a contradiction, $B$ is not $K'$-valuation independent. Hence, there are $a_1,\ldots, a_n\in K'$ such that 
\[
v(\sum_{i=1}^n b_ia_i) > \gamma\coloneqq \min\{v(b_ia_i) : 1\leqslant i\leqslant n\}. 
\]
Let $J\coloneqq \{i : v(b_ia_i)=\gamma\}$. In particular, $J$ is non-empty and $v(\sum_{i\in J}b_ia_i)>\gamma$. For $i,j\in J$ $v(b_i/b_j)=v(a_j/a_i)$ and since $\VG(L)\cap\VG(K')=\VG(K)$, all $v(b_i)$ lie in the same coset modulo $\VG(K)$. Therefore, by (N1), $v(b_i)=v(b_j)$ for all $i,j\in J$. Fixing some $i_0\in J$, by (N2), the elements in $E\coloneqq\{\res(b_i/b_{i_0}) : i\in J\}$ are $\RF(K)$-linearly independent. Since $\RF(L)$ is linearly disjoint from $\RF(K')$ over $\RF(K)$, they are also $\RF(K')$-linearly independent. But by \cite[Lemma~3.2.2]{prestel2005}, this implies that $\{b_i: i\in J\}$ is $K'$-valuation independent, a contradiction. 

The moreover part follows for $\VG$ and $\RF$ as in the proof of \cite[Proposition 12.11]{HHM}. This implies the result for $\RV$ since $\RV$ sits in the exact sequence $1\to \RF^\times\to \RV^\times\to \VG\to 0$ (note this also follows from Fact \ref{fact:rv-sum}).  
\end{proof}

For the definition of valued vector spaces and the following fact see \cite[Definition 3.2.23 and Corollary 3.2.26]{transseries}.

\begin{fact}\label{lem:max-sep} Let $(K,v)$ be a maximally complete valued field. Then, every finite dimensional valued $K$-vector space $(V,w)$ admits a $K$-valuation basis. \qed
\end{fact}

\begin{lemma}\label{lem:def-to-sep} Let $K$ be valued field having an elementary extension which is maximal. Let $K\subseteq L\subseteq M$ be valued field extensions with $K\preccurlyeq M$. If $\qftp_{\cLgk}(L/K)$ is definable, then the extension $L/K$ is $vs$-defectless.   
\end{lemma}

\begin{proof} Let $V=\mathrm{Span}_K(b_1,\ldots, b_n)$ where $b_1,\ldots, b_n\in L$ are $K$-linearly independent. The pair $(V,v)$ is a valued $K$-vector space. 

Let $\precsim$ be the preorder on $K^n$ given by 
\[a\precsim a' \Leftrightarrow v\left(\sum_{i=1}^n a_ib_i\right)\leq v\left(\sum_{i=1}^n a'_ib_i\right)\]
and let $\sim$ be the associated equivalence relation on $K^n$. Since $\qftp_{\cLgk}(L/K)$ is definable, $\precsim$ and $\sim$ are $K$-definable, and thus the quotient $\Gamma_V\coloneqq K^n/\!\sim$ together with the total ordering $\leq$ on $\Gamma_V$ induced by $\precsim$ is $K$-interpretable.

Letting $w\colon K^n\to \Gamma_V$ be the function sending $a$ to $a/\!\sim$, it is clear that the valued $K$-vector space $(V,v)$ is isomorphic to the $K$-interpretable valued $K$-vector space $(K^n,w)$. Let $K'$ be a maximally complete elementary extension of $K$ and $((K')^{n},w')$ be the corresponding definable vector space. By Fact~\ref{lem:max-sep}, $((K')^{n},w')$ admits a $K'$-valuation basis. But this is a definable property, so $(K^n,w)$ has also a $K$-valuation basis, and hence so does $(V,v)$. 
\end{proof}

\subsection{Domination}\label{sec:dominationII}

Let $\cL$ be either $\cLgk$, $\cLovf$ or $\cLpcf(e,f)$ and $T$ be either a completion of $\ACVF$, $\RCVF$ or $\PCF$, each in their corresponding language. Let $K\subseteq K'$ and $K\subseteq L$ be $\cL$-structures all being $\cL$-substructures of some sufficiently large model $\cU$ of $T$. We let $\VF(K)[\VF(L),\VF(K')]$ denote the ring generated by $\VF(L)$ and $\VF(K')$ over $\VF(K)$ and $\VF(K)(\VF(L),\VF(K'))$ denote the field compositum of $\VF(L)$ and $\VF(K')$ over $\VF(K)$. We let $K\langle L,K'\rangle$ denote the $\cL$-structure generated by $L$ and $K'$ over $K$.

\begin{proposition}\label{prop:our12.11} Let $K\subseteq F$ and $K\subseteq L$ be $\cLgk$-structures, such that both $F$ and $L$ are $\cLgk$-substructures of some sufficiently saturated and homogeneous algebraically closed valued field $\cU$. Suppose that 
\begin{itemize}
    \item $K$ is $\VF$-generated and $\VF(K)$, $\VF(F)$, $\VF(L)$ are fields, 
    \item $\VF(L)/\VF(K)$ is $vs$-defectless, 
    \item $v(\VF(L))\cap v(\VF(F))=v(\VF(K))$ and 
    \item $\res(\OO(L))\ind_{\res(\OO(K))}^{\ld} \res(\OO(F))$.
\end{itemize}
Let $\sigma$ and $\tau$ be two automorphims of $\cU$ over $K\VG(L)\RF(L)$ with $\sigma(L)=L'$ and $\tau(F)=F'$. Then there is an automorphism $\rho$ of $\cU$ mapping $N=K\langle L, F\rangle$ to $N'=K\langle L',F'\rangle$ and such that $\rho_{|L}=\sigma_{|L}$ and $\rho_{|F}=\tau_{|F}$. Furthermore, $\VF(N)$ is the ring generated by $\VF(L)$ and $\VF(F)$, $\VG(N)= \VG(L)+\VG(F)$, $\RF(N)$ is contained in the field compositum $\RF(F)\RF(L)$, and we have $\VF(L)\ind_{\VF(K)}^{\vd} \VF(F)$.   
\end{proposition}

\begin{proof} Without loss of generality we may assume $\tau$ is the identity. Furthermore, we may suppose that $F$ and $L$ are also $\VF$-generated. Indeed, let $\widetilde{F}\subseteq F$ and $\widetilde{L}\subseteq L$ be their respective $\VF$-generated parts, $\widetilde{L}'=\sigma(\widetilde{L})$, $\widetilde{F}'=\tau(\widetilde{F})$ and $\widetilde{\rho}$ be the automorphism sending $\widetilde{N}=K\langle \widetilde{L}, \widetilde{F}\rangle$ to $\widetilde{N}'=K\langle \widetilde{L}', \widetilde{F}'\rangle$. In addition, $\VF(\widetilde{N})$ is the field compositum of $\VF(\widetilde{L})$ and $\VF(\widetilde{F})$ over $\VF(K)$. What remains in $N\setminus \widetilde{N}$ only appears in the sorts $\VG$ and $\RF$. Then, the union of the maps $\widetilde{\rho}_{|\widetilde{N}}$ and $\sigma_{|N\setminus \widetilde{N}}$ defines an $\cLgk$-isomorphism between $N$ and $N'$. By homogeneity it extends to an autormorphism $\rho$ on $\cU$ which, by construction, satisfies $\rho_{|L}=\sigma_{|L}$ and $\rho_{|F}=\tau_{|F}$. The last statement of the proposition follows from the assumptions, together with the corresponding statement for $\widetilde{N}$. 

Assuming that all structures are $\VF$-generated, we may apply Lemma \ref{lem:vd-equiv}. The rest of the proof follows word for word the proof of \cite[Proposition 12.11]{HHM}. The independence $\VF(L)\ind_{\VF(K)}^{\vd} \VF(F)$ ensures that the map
$\rho$ defined on $\VF(K)[\VF(L),\VF(F)]$ by 
\[
\rho(\sum_{i=1}^n l_if_i) \coloneqq \sum_{i=1}^n \sigma(l_i)f_i
\] 
is well-defined. Note that since $N=\langle \VF(K)(\VF(L),\VF(F))\rangle$, it suffices to show that $\rho\colon \VF(N)\to \VF(N')$ is a valued field isomorphism. This is also proved in \cite[Proposition 12.11]{HHM}, and follows from the fact that $\VF(L)\ind_{\VF(K)}^{\vd} \VF(F)$. 
\end{proof}

The following result is new in the case of $\PCF$.  

\begin{proposition}\label{prop:our12.11new} Let $\cL$ be either $\cLovf$ or $\cLpcf(e,f)$ and $T$ be either $\RCVF$ or $\PCF$, each in their corresponding language.  

Let $K\subseteq F$ and $K\subseteq L$ be $\cL$-structures all being $\cL$-substructures of some sufficiently saturated and homogeneous model $\cU$ of $T$. Suppose that 
\begin{itemize}
    \item $K$ is $\VF$-generated and $\VF(K)$,$\VF(F)$, $\VF(L)$ are fields, 
    \item $\VF(L)/\VF(K)$ is $vs$-defectless, 
    \item $v(\VF(L))\cap v(\VF(F))=v(\VF(K))$ and 
    \item $\res(\OO(L))\ind_{\res(\OO(K))}^{\ld} \res(\OO(F))$.
\end{itemize}
Let $\sigma$ and $\tau$ be two automorphisms of $\cU$ over $K\VG(L)\RF(L)$ with $\sigma(L)=L'$ and $\tau(F)=F'$. Then there is an automorphism $\rho$ of $\cU$ mapping $N=K\langle L, F\rangle$ to $N'=K\langle L',F'\rangle$ and such that $\rho_{|L}=\sigma_{|L}$ and $\rho_{|F}=\tau_{|F}$. Furthermore, $\VF(N)$ is the ring generated by $\VF(L)$ and $\VF(F)$, $\VG(N)= \VG(L)+\VG(F)$, $\RF(N)$ is contained in the field compositum $\RF(F)\RF(L)$, and we have $\VF(L)\ind_{\VF(K)}^{\vd} \VF(F)$. 
\end{proposition}

\begin{proof} As in the proof of Proposition \ref{prop:our12.11}, we may suppose that: $\tau$ is the identity, $F$ and $L$ are $\VF$-generated. Moreover, it holds that $\VF(L)\ind_{\VF(K)}^{\vd} \VF(F)$ and that the map $\rho$ on $\VF(K)(\VF(L),\VF(F))$ defined by $\rho(\sum_{i=1}^n l_if_i) = \sum_{i=1}^n \sigma(l_i)f_i$, is an $\cLgk$-isomorphism sending $N$ to $N'$. It remains to show that $\rho$ is also an $\cLovf$-isomorphism when $T$ is $\RCVF$ and an $\cLpcf(e,f)$-isomorphism when $T$ is $\PCF$. For $\RCVF$, this is contained in the proof of \cite[Theorem 2.5]{has-ealy-mar}. Let us show it for $\PCF$. 

It suffices to show that $\rho$ preserves the predicates $P_n$. Suppose $x=\sum_{i=1}^m l_if_i$ with $l_i\in \VF(L)$ and $f_i\in \VF(F)$. Without loss of generality, we may assume that $\{l_i\}_{1\leq i\leq n}$ are $\VF(K)$-valuation independent and so $\VF(F)$-valuation independent by Lemma~\ref{lem:vd-equiv}. Let us first show the result for $m=1$, so $x=l_1f_1$. Let $d=ef$ and $u_1,\ldots, u_d\in \OO(\cU)$ be the interpretations of the $d=ef$ constant symbols from $\cLpcf(e,f)$ (which, by definition, represent an $\mathbb{F}_p$-basis of $\OO(\cU)/p\OO(\cU)$). Since the completion of $\Q(u_1,\ldots,u_d)$ is a model of $T$ and $\Q(u_1,\ldots, u_d)\subseteq K$, $K$ contains representatives of
$\cU^\times/P_n(\cU)$ as $P_n(\cU)$ is an open subgroup. In particular, by possibly dividing by an element of $K$, we may suppose that $l_1$ is already an $n^\mathrm{th}$-power. Therefore, $l_1f_1$ is an $n^\mathrm{th}$-power if and only if $f_1$ is an $n^\mathrm{th}$-power, and therefore $\rho(l_1f_1)=\sigma(l_1)f_1$ is an $n^\mathrm{th}$-power if and only if $f_1$ is an $n^\mathrm{th}$-power. For $m>1$, since the residue field is finite and $\{l_i\}_{1\leq i\leq m}$ is $\VF(F)$-valuation independent, we have that $v(l_1),\ldots, v(l_n)$ all lie in different $\VG(F)$-cosets, and we may thus assume that $v(l_1f_1)<v(l_if_i)$ for all $i>1$. In particular, $x=l_1f_1(1+\epsilon)$ for $\epsilon$ such that $v(\epsilon)>\Z$. Note that $P_n(1+\epsilon)$ holds for all $n\geqslant 1$, since $P_n(\cU)$ contains a ball centered at 1 with radius in $\mathbb{Z}$. Then, $x$ is an $n^\mathrm{th}$-power if and only if $l_1f_1$ is an $n^\mathrm{th}$-power, and we are reduced to the case $n=1$.  \end{proof}

Part (1) of the following proposition generalizes results in \cite{CD} and \cite{cubi-ye}. 

\begin{proposition}\label{lem:domination} Let $\cL$ be $\cLgk$, $\cLovf$ or $\cLpcf(e,f)$ and $T$ be a completion of $\ACVF$, $\RCVF$ or $\PCF$, each in their corresponding language. Let $K\subseteq L\subseteq \cU$ be $\cL$-structures with $K\preccurlyeq \cU\models T$ and $\VF(L)$ a field (but $L$ not necessarily $\VF$-generated). Then 
\begin{enumerate}[label=(\arabic*)]
    \item $\tp(L/K)$ is definable if and only if $\VF(L)/\VF(K)$ is $vs$-defectless and $\tp(\RF(L)/\RF(K))$ and $\tp(\VG(L)/\VG(K))$ are both definable;
    \item if $tp(L/K)$ is definable and $K\subseteq K'\subseteq \cU$, then 
\begin{equation}\label{eq:domination}\tag{$\ast$}
\tp(L/K)\cup\tp(\RF(L)/\RF(K))\mid \RF(K')\cup\tp(\VG(L)/\VG(K))\mid \VG(K') \vdash \tp(L/K)\mid K'.
\end{equation}
\end{enumerate}
\end{proposition}

\begin{remark}
By pure stable embeddedness of both the residue field and the value group, the definability of the types $\tp(\RF(L)/\RF(K))$ and $\tp(\VG(L)/\VG(K))$ can be taken either in the languages of rings and ordered groups, respectively, or in the full ambient valued field language.  
\end{remark}

\begin{proof}
One direction of part (1) follows by Lemma \ref{lem:def-to-sep}, and stable embeddedness of the value group and residue field. The converse direction follows word for word as in the proof of \cite[Theorem 5.9]{cubi-ye}, where the assumption that $L$ is a model is not used. 

For part (2), let $f\colon L\to \cU$ be an $\cL$-embedding such that $f(L)$ realizes the left-hand side of \eqref{eq:domination}. Possibly passing to the fields of fractions, we may suppose $\RF(L)$ is a field. Let $h\colon \VG(L)\to \VG(\cU)$ and $g\colon \RF(L)\to \RF(\cU)$ be the corresponding induced embeddings. By Fact \ref{fact:def-field} and  and Fact~\ref{fact:def-intersection}, $\RF(g(L))\ind_{\RF(K)}^{\ld} \RF(K')$ and $\VG(h(L))\cap \VG(K')=\VG(K)$. By Part (1), the extension $\VF(L)/\VF(K)$ is vs-defectless. We are thus in the situation of Proposition \ref{prop:our12.11} or Proposition \ref{prop:our12.11new}, respectively, which yields the result. 
\end{proof}

We present further domination results---more precisely $\RV$-domination results---in benign valued fields.

\begin{proposition}\label{prop:def-upto-RV}Let $(K,v)$ be benign. Work in the language $\cLrvg$, possibly $\RV$-enriched, and let $K\subseteq L\subseteq K'$, with $K\preccurlyeq K'$ and $\VF(L)$ a field, where $L$ is a (not necessarily $\VF$-generated) substructure of $K'$. Then the following holds:
\begin{enumerate}
    \item $\tp(L/K)$ is definable if and only if $\tp(\RV(L)\cup\VG(L))/\RV(K))$ is definable (in the theory of $\RV(K)$) and $\VF(L)/\VF(K)$ is $vs$-defectless;
    \item assuming that $\tp(L/K)$ is definable and letting $K\subseteq N\subseteq\mathcal{U}$, with $K\preccurlyeq\mathcal{U}$ and $N$ a substructure, it holds that  
    \[
    \tp(L/K)\cup\tp(\RV(L)\cup\VG(L)/\RV(K))\mid \RV(N)\cup\VG(N)\vdash \tp(L/K)\mid N.
    \]
\end{enumerate}
\end{proposition}

\begin{proof} To prove (1), we follow the proof of \cite[Theorem 2.17]{touchard}, where this is shown in the special case where $L$ is an elementary extension of $K$. We give the argument since, in contrast to Proposition \ref{lem:domination}, it has not been fully written elsewhere.  

Suppose $\tp(L/K)$ is definable. By Lemma \ref{lem:def-to-sep}, the extension $\VF(L)/\VF(K)$ is $vs$-defectless. That the type $\tp(\RV(L)\cup\VG(L)/\RV(K))$ is definable follows from the assumption and the fact that $\RV(K)$ is purely stably embedded (Fact~\ref{F:Benign-facts}). 

For the converse, let $\varphi(x,y)$ be an $\cL$-formula, with $y$ a tuple from sort $\VF$ and $a\in L^{|x|}$, with $a_0$ the subtuple given by elements from sort $\VF$. 

By $\VF$-quantifier elimination (in the theory of $K$), we may suppose the formula $\varphi(a,y)$ is of the form
\[
\psi(\rv(P_1(y)),\ldots, \rv(P_m(y)), b) 
\]
where $P_i(y)$ are polynomials with coefficients in the field $K(a_0)$ and $b$ is a tuple of elements in $\RV(L)\cup\VG(L)$, and where $\psi$ is a formula in the (enrichment of) the language $\mathcal{L}_{\RVV,\Gamma}$. Let $c=(c_1,\ldots, c_n)$ be a valuation basis of the $K$-vector space spanned by the coefficients of the polynomials $P_1(y),\ldots,P_m(y)$. Then, each polynomial $P_i$ can be expressed as a sum of the form $\sum_{j=1}^n c_jQ_{ij}(y)$ where $Q_{ij}\in K[y]$. 
By Fact~\ref{fact:rv-sum}, this sum factors through $\rv$, namely 
\[
\rv(P_i(y))=\bigoplus_{j=1}^n \rv(c_j)\rv(Q_{ij}(y)) 
\]
for any $y$ from $K$,
and therefore the subset of $K$ defined by the formula $\varphi(a,y)$ may also be  defined by a formula of the form 
\[
\theta(\rv(Q_{ij}(y)),\rv(c_1), \ldots \rv(c_n), b).  
\]
Introduce new $\RV$-variables $z=(z_{ij})_{ij}$. Since $\tp(\RV(L)\cup\VG(L)/\RV(K))$ is definable, we may find an $\RV(K)$-formula $\theta'(z)$ such that  $\theta'(\RV(K))=\theta(\RV(K), \rv(c_1),\ldots, \rv(c_n), b)$. It follows that 
$\theta'((\rv(Q_{ij}(y)))_{ij})$ defines the same set in $K$ as $\varphi(a,y)$.

To prove (2), we proceed as in Part (2) of Proposition \ref{lem:domination}. Indeed we may suppose that $\RF(L)$ is a field. Moreover,  $\tp(\RV(L)\cup\VG(L)/\RV(K))\mid\RV(N)\cup\VG(N)$ implies $\RF(L)\ind^{\ld}_{\RF(K)}\RF(N)$ as well as $\VG(L)\cap\VG(N)=\VG(K)$ by Fact~\ref{fact:def-field} and Fact~\ref{fact:def-intersection}, respectively. This shows, by Proposition \ref{prop:our12.11}, that $\mathrm{qftp}_{\cLgk}(L/K)|N$ and hence also 
$\mathrm{qftp}_{\cLrvg}(L/K)|N$ is determined by $\tp(L/K)\cup \tp(\RV(L)\cup\VG(L)/\RV(K))\mid\RV(N)\cup\VG(N)$. This completes the result by $\VF$-quantifier elimination (Fact~\ref{F:Benign-facts}).
\end{proof}

The following is an immediate consequence of Proposition~\ref{prop:def-upto-RV}, taking into account part (3) of Fact \ref{F:Benign-facts} and that $\RV^\times$ is canonically isomorphic to $\RF^\times\times\VG$, if an angular component map is added to the language.

\begin{corollary}\label{Cor:Pas-domination}
Let $(K,v)$ be benign. Work in the language $\mathcal{L}_{\Pas}$, possibly $\VG$-$\RF$-enriched, and let $K\subseteq L\subseteq K'$, with $K\preccurlyeq K'$, where $L$ is a (not necessarily VF-generated) substructure of $K'$. Then the following holds:
\begin{enumerate}[label=(\arabic*)]
    \item $\tp(L/K)$ is definable if and only if both $\tp(\RF(L)/\RF(K)$ and $\tp(\VG(L)/\VG(K))$ are definable and $\VF(L)/\VF(K)$ is $vs$-defectless;
    \item assuming that $\tp(L/K)$ is definable and letting $K\subseteq N\subseteq\mathcal{U}$, with $K\preccurlyeq\mathcal{U}$ and $N$ a substructure, it holds that 
    \[
    \tp(L/K)\cup\tp(\RF(L)/\RF(K))\mid \RF(N)\cup\tp(\VG(L)/\VG(K))\mid \VG(N) \vdash \tp(L/K)\mid N. \qed
    \]
\end{enumerate}
\end{corollary}

\begin{remark}\label{rem:dom-with-quotients}
Let $\cL_\mathrm{D}$ be $\cL_\Pas$ without the angular component map together with new sorts for the quotients $\RF^\times/(\RF^\times)^n$ for every $n$ and the corresponding projection maps between them (see \cite[Section 5.4]{AsChGeZi20} for a precise definition). The proof of  \cite[Theorem 5.15]{AsChGeZi20} gives a relative quantifier elimination result for benign valued fields in this language. In this setting, one may also get a domination statement analogous to Corollary \ref{Cor:Pas-domination} for $\VG$-$\RF$-enrichments of $\cL_\mathrm{D}$ combining Proposition \ref{prop:def-upto-RV} and Lemma \ref{C:SES-beauty-extension}. We leave the details to the reader.     
\end{remark}

\section{Beautiful pairs of valued fields}\label{sec:BP-val-fields}

In this section we characterize all completions of stably embedded pairs of $\ACVF$, $\RCVF$, and $\PCF$ in terms of the completions of stably embedded pairs of the corresponding value group and residue field. An Ax-Kochen-Ershov type result will be given in Section \ref{sec:Ax-Kochen-RV} for beautiful pairs of benign henselian valued fields (see Theorem \ref{thm:red-to-k-gamma}).  The results in this section could also be inferred from this more general Ax-Kochen-Ershov type result given in Section \ref{sec:Ax-Kochen-RV}. However, at the cost of being redundant, we chose to include them separately so that those readers interested in these special cases can follow the arguments without going through the machinery of Sections \ref{sec:SES} and \ref{sec:Ax-Kochen-RV}.

\subsection{Beautiful pairs of $\ACVF$ in the three sorted language}\label{sec:3sorted}
Throughout Section~\ref{sec:3sorted}, we let $\cL$ denote the three sorted language of valued fields $\cLgk$. Recall that $\ACVF$ denotes a completion of the $\cLgk$-theory of algebraically closed non-trivially valued fields. Note that $\ACVF$ has $\mathrm{UDDT}$ by \cite[Theorem 6.3]{cubi-ye}. (This will actually follow from the results of the current section.) 
Recall that there are only two completions of the theory of stably embedded pairs of $\ACF$ and both correspond to theories of beautiful pairs with beauty transfer (see Proposition \ref{prop:ACF}). Similarly, by Theorem \ref{thm:DOAG-complete}, the 4 completions of the theory of stably embedded pairs of $\DOAG$ also correspond to theories of beautiful pairs with beauty transfer.

\begin{definition}\label{def:ACVF-pairs}
Let $T_k$ and $T_\Gamma$ be completions of the theory of stably embedded pairs of $\ACF$ and $\DOAG$, respectively. Let $\cK_k$ (resp. $\cK_\Gamma$) be the natural classes such that $T_k$ (resp. $T_\Gamma$) is the theory of $\cK_k$-beautiful pairs (resp. $\cK_\Gamma$-beautiful pairs). 
\begin{itemize}[leftmargin=*]
\item  We let $\ACVF(T_k,T_\Gamma)$ be the $\cL_P$-theory of stably embedded pairs of algebraically closed valued fields, for which the corresponding pair of residue fields is a model of $T_k$, and the corresponding pair of value groups is a model of $T_\Gamma$. 
\item We define $\cK$ to be the class of \sepa-pairs $\cA\in \cK_\Def$ such that $(\RF(A), \RF(P(\cA)))\in \cK_k$ and $(\VG(A), \VG(P(\cA)))\in \cK_\Gamma$. The class $\cK$ is said to be \emph{induced} by $\cK_k$ and $\cK_\Gamma$.
\end{itemize}
\end{definition}

\begin{remark}\label{rem:base} The class $\cK$ is closed under base extension. Indeed, assume that  $\cA\in\cK$ and that  $P(\cA)\preccurlyeq B\preccurlyeq \cU$. Setting $\cA'=\cA_B$, it holds that $(\RF(A'), \RF(P(\cA'))\in \cK_k$ and $(\VG(A'), \VG(P(\cA'))\in \cK_\Gamma$. This follows by Lemma \ref{lem:vd-equiv}, as $A/P(\cA)$ is $vs$-defectless by Lemma \ref{lem:def-to-sep}. This shows $\cA'\in\cK$.
\end{remark}

\begin{lemma}\label{lem:natu} Given $\cK_k$, $\cK_\Gamma$ and $\cK$ as above, $\cK$ is a natural class.
\end{lemma}

\begin{proof}
Properties $(i)$-$(iv)$ and $(vi)$ in Definition \ref{def:natural-class-K} are straightforward. Property $(v)$ (closure under base extension) is proved in Remark \ref{rem:base}. 
\end{proof}

\begin{remark}\label{rem:ACVF2-axioms}
By Part (1) of Proposition \ref{lem:domination}, the theory $\ACVF(T_k,T_\Gamma)$ is axiomatized by axioms stating for a pair $\cM=(M,P(\cM))$
\begin{itemize}
     \item $\VF(M)/\VF(P(\cM))$ is $vs$-defectless;
     \item $P(\cM)\preccurlyeq M\models\ACVF$;
    \item $(\RF(M), \RF(P(\cM)))\models T_k$ and $(\VG(M), \VG(P(\cM)))\models T_\Gamma$.
\end{itemize}
It follows easily that $\ACVF(T_k,T_\Gamma)$ is consistent.
\end{remark}

\begin{theorem}\label{thm:ACVF-Beauty} Let $\cK_k, \cK_\Gamma, \cK, T_k$ and $T_\Gamma$ be as in Definition \ref{def:ACVF-pairs}. Then, $\cK$-beautiful pairs exist. Moreover, $\cK$ has beauty transfer and $\ACVF_\BP(\cK)$ is axiomatized by $\ACVF(T_k,T_\Gamma)$.   
\end{theorem}

\begin{proof}
Let $\cM$ be an $\aleph_1$-saturated model of $\ACVF(T_k,T_\Gamma)$. By Lemma \ref{lem:beauty-transfer-cardinal}, it suffices to show that $\cM$ is an  $\aleph_1$-$\cK$-beautiful pair. So let $\cA\to \cM$ and $\cA\to \cB$ be $\BP$-embeddings with $\cB$ countable. 

\

\emph{Step 0.} Without loss of generality, we may suppose $\cA\subseteq \cM$, and by Lemma~\ref{lem:b-pair-finite} and Lemma-definition~\ref{lem:base-change} we may also assume $P(\cA)=P(\cB)$. Moreover, we may suppose that the valued field sorts of $A$ and $B$ are fields. 

\medskip

\emph{Step 1.} Since $\cB\in \cK$, we have that $\RF(\cB)=(\RF(B), \RF(P(\cB))\in \cK_k$ and similarly $\VG(\cB)=(\VG(B),\VG(P(\cB))\in \cK_\Gamma$. By beauty transfer of $\cK_k$ and beauty transfer of $\cK_\Gamma$, $\RF(\cM)=(\RF(M),\RF(P(\cM)))$ is $\aleph_1$-$\cK_k$-beautiful and $\VG(\cM)=(\VG(M),\VG(P(\cM)))$ is $\aleph_1$-$\cK_\Gamma$-beautiful. In particular, there is an $(\cL_\Ring)_{\BP}$-embedding $g\colon \RF(\cB)\to \RF(\cM)$ over $\RF(\cA)$, and similarly, there is an $(\cL_\og)_{\BP}$-embedding  $h\colon \VG(\cB)\to \VG(\cM)$ over $\VG(\cA)$. Fact \ref{fact:def-field} and Fact~\ref{fact:def-intersection} yield that $\RF(g(B))\ind_{\RF(P(\cA))}^{\ld} \RF(P(\cM))$ and $\VG(h(B))\cap \VG(P(\cM))=\VG(P(\cA))$.

\medskip

\emph{Step 2.} By quantifier elimination in $\cL$ and $\aleph_1$-saturation of $\cM$, there is an $\cL$-embedding $f\colon B\to M$ over $A$ inducing the maps $h$ and $g$ from Step 1. 
Note that since $P(\cA)=P(\cB)$, the restriction of $f$ to $P(\cB)$ is trivially a $\BP$-embedding.
Moreover, we have that
\begin{align*}
& f(B)\models \tp(B/P(\cA)) \\ 
& g(\RF(B))\models  \tp(\RF(B)/\RF(P(\cA)))\mid\RF(P(\cM)) \\ 
& h(\VG(B))\models  \tp(\VG(B)/\VG(P(\cA)))\mid\VG(P(\cM)).  
\end{align*}
Part (2) of Proposition \ref{lem:domination} yields that $f(B)\models \tp(B/P(\cA))\mid P(\cM)$, which shows that $f$ induces a $\BP$-embedding. (Note that it is clear that $f$ respects the predicate $P$ and is thus an $\cL_P$-embedding.) \end{proof}

Combining Proposition~\ref{prop:ACF} with Theorem~\ref{thm:DOAG-complete} and Theorem~\ref{thm:ACVF-Beauty} we obtain the following result.

\begin{corollary}\label{cor:all-completions-acvf} There are exactly 8 possible completions of $\ACVF_{\SE}$ (see Section~\ref{sec:strictpro} for a description), each corresponding to a theory of beautiful pairs of the form $\ACVF(T_k,T_\Gamma)$, and the corresponding natural classes have beauty transfer. \qed
\end{corollary}

\subsection{Beautiful pairs of $\RCVF$ and $\PCF$ in the three sorted language}\label{sec:RCVF-PCF-beauty}

Throughout this subsection we let $\cL$ denote either the 3-sorted language of ordered valued fields $\cL_{\mathrm{ovf}}$ or the 3-sorted language $\cLpcf(e,f)$. Our results for $\RCVF$ and $\PCF$ use the same pattern as those in the case of $\ACVF$ presented in the previous section. 

We first consider the case of $\RCVF$. Recall that there are only two completions of the theory of stably embedded pairs of $\RCF$ and both correspond to theories of beautiful pairs (see Remark \ref{rem:omin_dries}). Let $T_k$ and $T_\Gamma$ be completions of the theory of stably embedded pairs of $\RCF$ and $\DOAG$, respectively. Let $\cK_k$ (resp. $\cK_\Gamma$) be the natural classes such that $T_k$ (resp. $T_\Gamma$) is the theory of $\cK_k$-beautiful pairs (resp. $\cK_\Gamma$-beautiful pairs). Note that $\cK_k$ and $\cK_\Gamma$ have beauty transfer (by Theorem~\ref{thm:omin-beauty} and Theorem~\ref{thm:DOAG-complete}, respectively). 

As in the case of $\ACVF$ (see Definition \ref{def:ACVF-pairs}), we let $\cK$ be the class of \sepa-pairs  induced by $\cK_k$ and $\cK_\Gamma$, i.e., the \sepa-pairs $\cA\in \cK_\Def$ such that $(\RF(A), \RF(P(\cA))\in \cK_k$ and $(\VG(A), \VG(P(\cA))\in \cK_\Gamma$.

The argument given in Lemma \ref{lem:natu} shows that $\cK$ is a natural class. We let $\RCVF(T_k,T_\Gamma)$ be the $\cL_P$-theory of stably embedded pairs of real closed valued fields, for which the corresponding pair of residue fields is a model of $T_k$, and the corresponding pair of value groups is a model of $T_\Gamma$. Moreover, by Part (1) of Proposition \ref{lem:domination}, the theory $\RCVF(T_k,T_\Gamma)$ is axiomatizable and it is easy to see it is consistent.

\begin{theorem}\label{thm:RCVF-Beauty} Let $\cK_k, \cK_\Gamma, \cK, T_k$ and $T_\Gamma$ be as above. Then, $\cK$-beautiful pairs exist. Moreover, $\cK$ has beauty transfer and $\RCVF_\BP(\cK)$ is axiomatized by $\RCVF(T_k,T_\Gamma)$.   
\end{theorem}

\begin{proof} The proof follows exactly the same strategy as that of Theorem \ref{thm:ACVF-Beauty}, using the $\RCVF$ part of Proposition \ref{lem:domination}. 
\end{proof}

\begin{corollary}\label{rem:all-completions-rcvf} There are exactly 8 possible completions of $\RCVF_{\SE}$, each corresponding to a theory of beautiful pairs of the form $\RCVF(T_k,T_\Gamma)$, and the corresponding natural classes satisfy beauty transfer.\qed 
\end{corollary}

In the case of $p$-adically closed fields the situation is much simpler. Combining Proposition~\ref{lem:domination} and Theorem~\ref{thm:pres-complete}, and following the same strategy of the proof of Theorem~\ref{thm:ACVF-Beauty}, we have the following. Recall that  $\mathrm{PRES}$ denotes the theory of Presburger arithmethic.

\begin{theorem}\label{thm:pcf-Beauty}
The theory $\PCF_\BP$ is consistent and $\cK_\Def$ has beauty transfer. It is axiomatized by the theory of $vs$-defectless elementary pairs of models of $\PCF$ for which the corresponding pair of value groups is a model of $(\mathrm{PRES})_\BP$. The theories $\PCF_\BP$ and $\PCF_\BP(\cK_\triv)$ are the only 2 completions of $\PCF_{\SE}$.\qed
\end{theorem}

By Corollary~\ref{cor:NIP-transfer}, we have the following.

\begin{corollary}\label{cor:ACVF-NIP-transfer} 
Any completion of the theory of stably embedded pairs of $\ACVF$, $\RCVF$ and $\PCF$ is NIP.\qed 
\end{corollary}

\subsection{Beautiful pairs of $\ACVF$ in the geometric language}\label{sec:geomsorted}
In this subsection, we consider the theory $\ACVF$ in the geometric sorts $\cL^{\mathcal{G}}$. By the main result of \cite{HHM1}, $\ACVF$ eliminates imaginaries in $\cL^{\mathcal{G}}$.

\begin{theorem}\label{thm:Beauty-geom-ACVF} Let $T_\Gamma$ be a completion of the theory of stably embedded pairs of $\DOAG$ and let $T_k=\ACF_\BP$. Let $\cK_\Gamma$ and $\cK_k$ be their corresponding natural classes. Let $\cK$ be the class of \sepa-pairs (as $\cL^{\mathcal{G}}_P$-structures) $\cN\in \cK_\Def$ induced by $\cK_\Gamma$ and $\cK_k$ (as in Definition \ref{def:ACVF-pairs}). Then, $\cK$-beautiful pairs exist and are elementary pairs. Moreover,  $\cK$ has beauty transfer and $\ACVF_\BP(\cK)$ is axiomatized by $\ACVF(T_k,T_\Gamma)$. \end{theorem}

\begin{proof}
We will reduce the statement to Theorem~\ref{thm:ACVF-Beauty}. Let $\cN$ be an $\aleph_1$-saturated model of $\ACVF(\ACF_\BP,T_\Gamma)$ in the language $\cL^{\mathcal{G}}_{P}$. By Lemma \ref{lem:beauty-transfer-cardinal}, it suffices to show $\cN$ is an $\aleph_1$-$\cK$-beautiful pair. So let $\cA\to \cN$ and $\cA\to \cB$ be $\BP$-embeddings (in $\cK$, i.e., in the language $\cL^{\mathcal{G}}_{P}$) with $\cB\in\cK$ countable. 

As in the proof of Theorem~\ref{thm:ACVF-Beauty}, by Lemma-definition~\ref{lem:base-change} and Lemma \ref{lem:b-pair-finite} we may suppose $\cA\subseteq \cN$, $\cA\subseteq \cB$, $P(\cA)=P(\cB)$, and the valued field sorts of $A$ and $B$ are fields. 

Recall that $\mathbf{S}_n$ denotes the sort of $n$-dimensional lattices. We denote by $\mathbf{S}$ the union of all $\mathbf{S}_n$ for all $n>0$. For any $s\in \mathbf{S}_n(B)$, we let $p_s$ be the pushforward (under any $Q\in \mathop{GL}_n(\VF(\mathcal{U}))$ with $Q(\mathcal{O}^n)=s$) of $\eta^{\otimes n}$, where $\eta$ denotes the generic type of $\mathcal{O}$. Observe that $p_s$ is a generically stable global $s$-definable type. For any lattice $s\in \mathbf{S}(B)$, we choose $b_s\models p_s\mid B$, such that the tuple $(b_s)_{s\in\mathbf{S}(B)}$ is (forking) independent over $B$. Then 
$\cB_1\coloneqq (\langle B\cup\{b_s\mid s \in \mathbf{S}(B)\}\rangle,P(\cA)) \in \cK$, as $\VG(B_1)=\VG(B)$. Next, setting $\cA_1\coloneqq(\langle A\cup\{b_s\mid s \in \mathbf{S}(A)\}\rangle,P(\cA))$, we get $\cA\subseteq\cA_1,\cB\subseteq\cB_1$.

By $\aleph_1$-saturation of $\cN$ and since $\RF(N)\supsetneq P(\RF(N))$, we find an $A P(\cN)$-independent family $(c_s)_{s\in\mathbf{S}(A)}$ in $N$ (in the sense of forking) with $c_s\models p_s\mid AP(\cN)$ for all $s\in\mathbf{S}(A)$. It follows that $b_s\mapsto c_s$ defines a $\BP$-embedding of $\cA_1$ into $\cN$ over $\cA$. 

We now iterate the preceding procedure $\omega$ times and obtain increasing unions $(\cA_n)_{n\geq1}$ and $(\cB_n)_{n\geq 1}$. Set $\cA'\coloneqq\bigcup_{n\geq1}\cA_n$ and $\cB'\coloneqq\bigcup_{n\geq1}\cB_n$.

By construction, every lattice $s\in \mathbf{S}_n(A')$ has a basis in $\VF(A')$, similarly for $B'$. Thus, $s/\cM s$ is $\VF(A')$-definably isomorphic to $\RF^n$ for all $s\in\mathbf{S}_n(A')$, and it follows that  $\dcl^{\mathcal{G}}(A')=\dcl^{\mathcal{G}}(\VF(A')\cup \RF(A'))$. Similarly, $\dcl^{\mathcal{G}}(B')=\dcl^{\mathcal{G}}(\VF(B')\cup \RF(B'))$. 

This shows in particular that $\cA'$ and $\cB'$ are generated by their restriction to the sorts in $\cLgk$. The $(\cLgk)_{\BP}$-embedding given by Theorem~\ref{thm:ACVF-Beauty} between (the $(\cLgk)_{\BP}$-reducts of) $\cB'$ and $\cN'$ over $\cA'$ induces thus a $\BP$-embedding (in $\cK$).     
\end{proof}

\begin{remark}\label{rem:RCVF-imagin} It seems plausible that an analogue of Theorem \ref{thm:Beauty-geom-ACVF} holds for $\RCVF$ in the geometric language using prime resolutions as introduced in \cite[Section 4]{has-ealy-mar}. Alternatively (as suggested by one of the anonymous referees), one can also adapt the proof, using a definable type of bases of any lattice which is orthogonal to $\VG$ and choosing
one of the two generics of the 1-dimensional $\RCF$ vector space at each stage.  
\end{remark}

\subsection{Strict pro-definability of spaces of definable types}\label{sec:strictpro} 

In this section, we apply our results to show that various spaces of definable types in $\ACVF$, $\RCVF$ and $\PCF$ are strict pro-definable. Using the dual description of natural classes of \sepa-pairs in terms of natural classes of definable types given in Section~\ref{sec:def-type}, each theory of beautiful pairs $T_\BP(\cK)$ (where $T$ is one of the above theories) induces a corresponding natural class of definable types $\cF$ for which $\cK=\cK_\cF$. By Theorem \ref{thm:strictness}, the class $\cF$ is strict pro-definable. We now describe the classes of types they correspond to. It is worth pointing out that it was already shown in \cite{cubi-ye} that $\D{}$ is pro-definable for $\ACVF$, $\RCVF$ and $\PCF$.

\begin{remark}\label{rmk:geom-interpret}
In the following list of spaces of types in $\ACVF$ (resp. later in $\RCVF$), when we say that a certain class of types $X$ is the model-theoretic analogue of a certain geometric space $Y$, we mean that they are related via the restriction functor in~\cite[Chapter~14]{HL}. In particular, when $K$ is a maximally complete model with value group $\mathbb{R}$ (and residue field $\mathbb{R}$ in the case of $\RCVF$), we have that the restriction induces a canonical bijection between $X(K)$ and $Y(K)$. 
\end{remark}

\subsubsection{Spaces of definable types in $\ACVF$:}\label{sec:list-spaces-ACVF}
For $T_k$ and $T_\Gamma$ as in Definition \ref{def:ACVF-pairs} let $\cF$ be the corresponding natural class of definable types such that $\ACVF(T_k,T_\Gamma)$ axiomatizes $\ACVF_\BP(\cK_\cF)$. We now describe the 8 corresponding possible classes $\cF$. Suppose first that $T_k$ corresponds to the theory of proper pairs of algebraically closed fields $\ACF_\BP$ and let $V$ be an algebraic variety over a small model $K$ of $\ACVF$. Recall that $\cF_V$ denotes the space of types in $\cF$ that concentrate on $V$.

\begin{enumerate}[leftmargin=*, label=(\arabic*)]
\item Note that in $\ACVF$, a definable type $p$ is stably dominated type  if and only if $p$ is orthogonal to $\Gamma$~\cite[Proposition 2.9.1]{HL}. Thus, when $T_\Gamma$ is the theory of trivial pairs of $\DOAG$, $\cF$ corresponds to the class of stably dominated types. The set $\cF_V(K)$ of $K$-definable types in $\cF$ concentrating on $V$ is denoted $\widehat{V}(K)$, and can be viewed as a model-theoretic analogue of the Berkovich analytification $V^{\mathrm{an}}$ of $V$. It was already shown to be strict pro-definable in \cite{HL}. 
	\item When $T_\Gamma$ corresponds to bounded pairs of $\DOAG$, $\cF$ corresponds to the class of bounded definable types in $\ACVF$, meaning the types $p\in \D{}(\cU)$ that have a realisation in a model whose value group is bounded by $\Gamma(\cU)$. The set $\cF_V(K)$ will be denoted by $\widetilde{V}(K)$ and studied in a subsequent work. It can be seen as the model theoretic analogue of the adic space, in the sense of Huber~\cite{Adic}, associated to $V$. Indeed, when $K$ is a maximally complete with $\Gamma(K)=\mathbb{R}$, following the same machinery~\cite[Chapter 14.1]{HL}, one can show that $\widetilde{V}(K)$ is in natural bijection with the adic space associated to $V$.
	\item When $T_\Gamma$ is the theory of proper pairs of $\DOAG$ such that $P(\cM)$ is convex in $M$, $\cF$ corresponds to those definable types that are orthogonal to $p_{0^+}$ (in $\ACVF$). Indeed, since $\Gamma\models \DOAG$, $P(\cM)$ being convex in $M$ is the same as no realisation of $p_{0^+}$ is added. We are unaware of a geometric interpretation of such type spaces.
	\item When $T_\Gamma$ is the theory that corresponds to all definable types in DOAG, $\cF$ corresponds to the class of all definable types. The set $\cF_V(K)$ can be seen as a model theoretic analogue of the Zariski-Riemann space associated to $V$ following the same approach as identifying $\widehat{V}$ as the Berkovich analytification $V^{an}$ in~\cite[Chapter 14.1]{HL}. (See also (2) above.)
\end{enumerate}
Suppose now that $T_k$ corresponds to the theory of trivial elementary pairs of $\ACF$. 
\begin{enumerate}[resume, leftmargin=*, label=(\arabic*)]
	\item When $T_\Gamma$ is the theory of trivial pairs of DOAG, $\cF$ corresponds to the class of realized types.
	\item When $T_\Gamma$ corresponds to the class of bounded types, the class $\cF$ corresponds to those definable types which are orthogonal to the residue field sort and $p_{+\infty}$. We are unaware of a geometric interpretation of such type spaces.
	\item When $T_\Gamma$ is the theory of proper pairs of DOAG such that $P(\cM)$ is convex in $M$, the class $\cF$ corresponds to those definable which are orthogonal to the residue field sort and to $p_{0^+}$. Such types correspond to the residually algebraic types in~\cite{mu-stab-ACF} or standard types in~\cite{mu-stab-ACVF}. The set $\cF_V(K)$ can be used to give a notion of a branch of $V$ at infinity along a certain direction. Indeed, in the case when $V$ is an affine curve, such types correspond to infinitesimal neighborhoods of points at infinity for some projective embedding of $V$. For more details, see~\cite[Section 3]{mu-stab-ACF}. 
	\item When $T_\Gamma$ is the theory that corresponds to the class of all definable types in DOAG, then $\cF$ corresponds to the class of definable types orthogonal to the residue field sort. Again, we are unaware of a geometric interpretation of such type spaces.
\end{enumerate}

\begin{theorem}\label{cor:acvf-strictpro}
	Let $X$ be an interpretable set in $\ACVF$. Then, for $\cF$ as in classes (1)-(5) above, $\cF_X$ is strict pro-definable. When $X$ is a definable set in the 3-sorted language of $\ACVF$, for $\cF$ as in (6)-(8) above, $\cF_X$ is strict pro-definable.
\end{theorem}

\begin{proof} When $X$ is definable, the result follows for (1)-(8) by Theorems \ref{thm:ACVF-Beauty} and \ref{thm:strictness}. When $X$ is only assumed to be interpretable, the result follows for (1)-(5) by Theorems \ref{thm:Beauty-geom-ACVF} and \ref{thm:strictness}. 
\end{proof}

\begin{remark}
Note that if $f\colon Z\to W$ is a surjective pro-definable map between pro-definable sets with $Z$ strict pro-definable, then $W$ is also strict pro-definable. In particular, since $\ACVF$ has surjectivity transfer in $\cL^{\mathcal{G}}$ by \cite[Lemma 4.2.6]{HL}, one could alternatively prove strict pro-definability of $\D{X}(\cU)$ when $X$ is interpretable using strict pro-definability in case $X$ is a definable set in $\ACVF$ in the valued field sort. 

In the remaining cases, using the same method, one would have to show surjectivity transfer for the corresponding natural classes of types. Since these conditions have not been checked, invoking Theorem~\ref{thm:Beauty-geom-ACVF} (instead of just invoking Theorem~\ref{thm:ACVF-Beauty}) is necessary when $X$ is interpretable.
\end{remark}

\subsubsection{Spaces of definable types in $\RCVF$:}\label{sec:RCVF-description}
For $T_k$ and $T_\Gamma$ as in Definition \ref{def:ACVF-pairs} let $\cF$ be the corresponding natural class of definable types such that $\RCVF(T_k,T_\Gamma)$ axiomatizes $\RCVF_\BP(\cK_\cF)$. We now describe the 8 corresponding possible classes $\cF$. Suppose first that $T_k$ corresponds to the theory of proper pairs of real closed fields $\RCF_\BP$ and let $V$ be a semi-algebraic set in $\VF$.

\begin{enumerate}[leftmargin=*, label=(\arabic*)]
	\item When $T_\Gamma$ is the theory of trivial pairs of $\DOAG$, $\cF$ corresponds to the class of definable types which are orthogonal to $\Gamma$. The set $\cF_V(K)$, also denoted by $\widehat{V}(K)$, can be viewed as the model theoretic analogue of the Berkovich analytification of semi-algebraic sets as defined in \cite{jell_etal}.
	\item When $T_\Gamma$ corresponds to bounded pairs of $\DOAG$, $\cF$ corresponds to the class of bounded types in $\RCVF$ (see point (2) in Section \ref{sec:list-spaces-ACVF}). The set $\cF_V(K)$ could be thought of as (a model theoretic analogue of) the adic space associated to $V$. 	
	\item When $T_\Gamma$ is the theory of proper pairs of $\DOAG$ such that $P(M)$ is convex in $M$, as for $\ACVF$, we are unaware of  a geometric interpretation of such spaces. 
	\item When $T_\Gamma$ is the theory that corresponds to all definable types in $\DOAG$, $\RCVF(T_k,T_\Gamma)$ corresponds to the class of all definable types. The set $\cF_V(K)$ is a model theoretic analogue of the Zariski-Riemann space above a given ordering on the real spectrum of $V$.
    \item[(5-8)] These correspond to $T_k$ being the theory of trivial pairs of $\RCF$ and $T_\Gamma$ each of the 4 possibilities as in the case of $\ACVF$ (Section \ref{sec:list-spaces-ACVF}).
\end{enumerate}

\begin{theorem}\label{cor:rcvf-strictpro}
	Let $X$ be an $\cL_{\mathrm{ovf}}$-definable set in $\RCVF$. Then, for $\cF$ as in (1)-(8) above, $\cF_X$ is strict pro-definable.  \qed
\end{theorem}

\begin{remark}\label{rmk:rcvf-geom} 
A careful look at the proof of~\cite[Lemma 4.2.6 (1)]{HL} shows that $\RCVF$ has density of definable types and moreover, density of definable bounded types (the former result is also shown in \cite{hils-rideau-2024}). By Theorem \ref{thm:RCVF-Beauty} and Lemma \ref{lem:density} (and Remark \ref{rem:density}), this shows that whenever $X$ is interpretable, both $\D{X}$ and the set of bounded definable types concentrating on $X$ are strict pro-definable. 
\end{remark}

Similar methods give us the following:

\begin{theorem}\label{thm:pcf-strictpro}
Let $X$ be definable set in $\PCF$ in $\cLpcf$, then $S^\Def_X$ is strict pro-definable.
\end{theorem}

In this case, $\D{V}$ is a model-theoretic analogue of the $p$-adic spectrum of a variety $V$ in the sense of~\cite{ERob-padic} (see also \cite{belair}).

\section{An Ax-Kochen-Ershov principle for beauty}\label{sec:Ax-Kochen-RV}

\subsection{Pairs of henselian valued fields in $\RV$-enrichments}\label{sec:enrichments}

Let $\cL$ be an $\RV$-enrichment of $\cL_{\RVV}$ in which we allow a set of auxiliary sorts from $\RV^\eq$ which always contains $\VG$. Recall that $\RV$ denotes the $\RV$-sort together with the auxiliary sorts. Let $\cL^-$ be the restriction of $\cL$ to the sorts in $\RV$. For an $\cL$-structure $A$, let $A|\cL^-$ be its reduct to $\RV$. 

Let $T$ be the $\cL$-theory of a benign valued field $(K,v)$. Let $\cK_{\RVV}$ be a natural class of \sepa-pairs in the $\cL^-$-theory of $K|\cL^-$ (in sorts $\RV$). Assume that $\cK_{\RVV}$-beautiful pairs exist and that $\cK_{\RVV}$ has beauty transfer. We will assume in addition that $\cK_{\RVV}$ has the extension property (recall that this implies that $\cK_{\RVV}$-beautiful pairs are elementary pairs by Theorem~\ref{thm:charct-beaut-pairs}).

Let $\cK$ be the class of $\cL_P$-structures $\cM\in \cK_\Def$ induced by $\cK_{\RVV}$ analogously as in Definition \ref{def:ACVF-pairs}, i.e., $\cK$ is the class of \sepa-pairs $\cA\in \cK_\Def$ such that $(\RV(A'), \RV(P(\cA')))\in \cK_\RV$. 

\begin{lemma}\label{lem:natu-RV}
$\cK$ is a natural class.
\end{lemma}

\begin{proof}
We argue as in the proof of Lemma \ref{lem:natu}. Properties (i)-(iv)  and (vi) are clear. As for Property (v) (closure under base extension), let us first assume that $\RV$ is unenriched. In this case, the result follows exactly as in Remark \ref{rem:base}. If $\RV$ is enriched, let $\cM\in\cK$ and $P(\cM)\preccurlyeq B$. For $\cM'\coloneqq\cM_B$, since we work with an $\RV$-enrichment, by Lemma~\ref{lem:vd-equiv}, it still holds that $\RV(\cM')=\RV(\cM)_{\RV(B)}$. Since $\cK_\RVV$ is natural, $\cM'\in \cK$.
\end{proof}

\begin{theorem}\label{thm:RV-pairs} 
Let $\cK_{\RVV}$ be a natural class with the extension property. Assume $\cK_{\RVV}$-beautiful pairs exist and that $\cK_{\RVV}$ has beauty transfer. Let $\cK$ be the class of $\cL_P$-structures $\cM\in \cK_\Def$ induced by $\cK_{\RVV}$. Then, $\cK$ has the extension property and $\cK$-beautiful pairs exist. Moreover, $\cK$ has beauty transfer and $T_\BP(\cK)$ is axiomatized by the following conditions on $\cM=(M,P(\cM))$

\begin{itemize}
    \item $\VF(M)/\VF(P(\cM))$ is $vs$-defectless;
    \item $(M,P(\cM))$ is an elementary pair of models of $T$;
    \item $(M|\cL^-, P(\cM)|\cL^-)\models T_\BP(\cK_{\RVV})$.
\end{itemize}
\end{theorem}

\begin{proof}
The proof follows exactly the same strategy as that of Theorem \ref{thm:ACVF-Beauty}, using Proposition \ref{prop:def-upto-RV} instead of Proposition \ref{lem:domination}. 
\end{proof}

\subsection{Pairs of henselian valued fields with angular component}

Let $T$ be the theory of a benign valued field $(K,v)$, where we work in a $\VG$-$\RF$-enrichment of $\cL_\Pas$. Let $\cK_k$ and $\cK_\Gamma$ be natural classes of \sepa-pairs in the theory of the residue field $k$ and of the value group $\Gamma$ of $(K,v)$, respectively. As before, for ease of presentation, we will assume both classes $\cK_{k}$ and $\cK_\Gamma$ have the extension property and the amalgamation property, so that  $\cK_k$-beautiful pairs and $\cK_\Gamma$-beautiful pairs exist. Assume moreover that $\cK_k$ and $\cK_\Gamma$ have beauty transfer. Let $\cK$ be the class of \sepa-pairs $\cM\in \cK_\Def$ induced by $\cK_k$ and $\cK_\Gamma$. Recall that adding an angular component map is bi-interpretable to adding a splitting of the short exact sequence $1\rightarrow \RF^{\times}\rightarrow \RV^\times\rightarrow\VG\rightarrow 0$, and thus we obtain a uniquely determined associated $\RV$-enrichment, in which moreover  $\VGi$ and $\RF$ are orthogonal (as this holds in any  $\VG$-$\RF$-enrichment of the $\cL_\Pas$-theory of $(K,v)$) and in which we have  a definable  product decomposition $\RV^\times\cong\RF^\times \times \VG$. In particular, $\cK_k$ and $\cK_\Gamma$ induce a natural class $\cK_{\RVV}$, which in turn induces (going back to the original enrichment of $\cL_\Pas$) the same class $\cK$. As a corollary of Theorem \ref{thm:RV-pairs} we thus obtain:

\begin{corollary}
Let $\cK$ be the class of $\cL_P$-structures $\cM\in \cK_\Def$ induced by $\cK_k$ and $\cK_\Gamma$. Then, $\cK$ has the extension property and $\cK$-beautiful pairs exist. Moreover, $\cK$ has beauty transfer and $T_\BP(\cK)$ is axiomatized by the following conditions on $\cM=(M,P(\cM))$:
\begin{itemize}
    \item $\VF(M)/\VF(P(\cM))$ is $vs$-defectless;
    \item $(M,P(\cM))$ is an elementary pair of models of $T$;
    \item $(\RF(M), \RF(P(\cM)))\models T_\BP(\cK_{k})$ and $(\VG(M), \VG(P(\cM)))\models T_\BP(\cK_{\Gamma})$.\qed
\end{itemize}
\end{corollary}

\subsection{Ax-Kochen-Ershov principle}\label{sec:apl_hen_fields}
Let $T$ be the theory of the benign valued field $(K,v)$. Consider the (definitional) $\cL_\RVV$-expansion where we add sorts $\bbA$ for $k^*/(k^*)^n$ for all $n\geqslant 0$ and $\VG$ (with the natural quotient maps). We are in the setting of Section \ref{sec:enrichments}. 

Let $\cK_{\mathscr{A}}$ and $\cK_\Gamma$ be natural classes of \sepa-pairs in the theory of the residue field $k$ (in the sorts $\bbA$) and of the value group $\Gamma$ of $(K,v)$, respectively. Assume both $\cK_{\mathscr{A}}$ and $\cK_\Gamma$ have extension and amalgamation. By Theorem~\ref{thm:charct-beaut-pairs}, amalgamation is equivalent to the fact that $\cK_{\mathscr{A}}$-beautiful pairs and $\cK_\Gamma$-beautiful pairs exist. Assume in addition that $\cK_{\mathscr{A}}$ and $\cK_\Gamma$ have beauty transfer.

\begin{theorem}\label{thm:red-to-k-gamma}
Let $\cK$ be the class of $\cL_P$-structures $\cM\in \cK_\Def$ induced by $\cK_{\mathscr{A}}$ and $\cK_\Gamma$. Then, $\cK$-beautiful pairs exist. Moreover, $\cK$ has beauty transfer and $T_\BP(\cK)$ is axiomatized by the following conditions on an $\cL_P$-structure $\cM=(M,P(\cM))$:
\begin{itemize}
    \item $\VF(M)/\VF(P(\cM))$ is vs-defectless;
    \item $P(\cM)\preccurlyeq M\models T$;
    \item $(\bbA(M),\bbA(P(\cM)))\models T_\BP(\cK_{\mathscr{A}})$ and $(\VG(M),\VG(P(\cM)))\models T_\BP(\cK_\Gamma)$.
\end{itemize}
\end{theorem}

\begin{proof}
This follows directly from Theorems \ref{thm:RV-pairs} and \ref{thm:exact}. 
\end{proof}

\begin{remark}
    Note that Theorem \ref{thm:red-to-k-gamma} is actually an equivalence. Indeed, if we do not assume beauty transfer for $\cK_{\mathscr{A}}$ and $\cK_\Gamma$ from the start, then beauty transfer for $\cK$ (trivially) implies beauty transfer for $\cK_{\mathscr{A}}$ and $\cK_\Gamma$.
\end{remark}

\begin{examples}\
\begin{enumerate}[leftmargin=*,label=(\arabic*)]
    \item If $K=k(\!(t)\!)$ where $k\in \{\mathbb{C},\R,\Q_p\}$ then all sorts $k^*/(k^*)^n$ for $n>0$ are finite and consist of $0$-definable elements. Therefore, in all three cases, $\cK_{\mathscr{A}}$ may be identified with $\cK_{k}$, and so, by inspection, one sees that all completions of stably embedded elementary pairs of models of $\Th(K)$ correspond to theories of beautiful pairs for appropriate choices of natural classes in the residue field and value group. Moreover, since $\Q_p(\!(t)\!)$ with the $t$-adic valuation is bi-interpretable with $\Q_p(\!(t)\!)$ with the total valuation, we get the mixed-characteristic result for $\Q_p(\!(t)\!)$ using the above machinery.
    \item Note that both 
    Theorem~\ref{thm:ACVF-Beauty} and Theorem~\ref{thm:RCVF-Beauty} can be recovered from Theorem~\ref{thm:red-to-k-gamma}. In the latter case, the only non-trivial quotients $k^*/(k^*)^n$ in $\bbA$ appear for even $n$ when they are of size 2, allowing to define the ordering on the valued field without $\VF$-quantifiers.
    \item The results can also be applied to any algebraically maximal Kaplansky field in positive characteristic. A particular example of interest is $K=\mathbb{F}^{\alg}_p(\!(\Gamma)\!)$, with $\Gamma=1/p^\infty\Z$ the subgroup of $\mathbb{Q}$ consisting of the elements that can be expressed as $m/p^n$. The beautiful pairs of models of $\Th(K)$ are axiomatized by the corresponding beautiful pairs of the residue fields and value groups. The residue field pairs are characterized by Proposition \ref{prop:ACF}. For the value group, due to the fact that stably embedded elementary pairs are not elementary (see Proposition~\ref{prop:reg-abel}), the situation is slightly more complicated since we need to restrict to convex elementary pairs in the sense of Theorem \ref{thm:ROAG-main}, where the class still has beauty transfer. Together with the beautiful pairs of $\mathrm{ACF}_p$ these give certain theories of beautiful pairs of $\Th(K)$ and beauty transfer by Theorem~\ref{thm:red-to-k-gamma}.
\end{enumerate}
\end{examples}

 \section{Concluding remarks and open questions}

\subsection{Variants in mixed characteristic}

We sketch in this section an analogous Ax-Kochen-Ershov type result in mixed characteristic. This applies in particular to the field of Witt vectors $W(\mathbb F_p^\mathrm{alg})$. The strategy being completely analogous to the one given for benign valued fields, we will only point to the main changes without giving any proof.    

We work in a slight variant of the language $\cL_\RVV$, with sorts $\RV_n$ for $n\geqslant 0$ as defined in \cite{fle-2011}. Abusing notation, we let $\RV$ denote the union of all sorts $\RV_n$ (including all new sorts when working in an $\RV$-enrichment) and $\cL_\RVV$ the corresponding language. See \cite[Section 7]{hils-mennuni} for more details about this context.

Let $\cL$ be an $\RV$-enrichment of $\cL_\RVV$ and $T$ be the $\cL$-theory of a mixed characteristic finitely ramified henselian valued field $K$ with perfect residue field. Recall $T$ eliminates $\VF$-quantifiers by \cite{fle-2011}. Suppose $\cL$ contains a sort for $\Gamma$. Let $\cK_{\RVV}$ be a natural class of \sepa-pairs in the $\cL^-$-theory of $K|\cL^-$ ($\cL^-$ as in Section \ref{sec:Ax-Kochen-RV}). Assume that $\cK_{\RVV}$-beautiful pairs exist, $\cK_{\RVV}$ has the extension property and $\cK_{\RVV}$ has beauty transfer. Let $\cK$ be the class of $\cL_P$-structures $\cM\in \cK_\Def$ induced by $\cK_{\RVV}$. As in Lemma \ref{lem:natu-RV}, $\cK$ is a natural class. 

\begin{theorem}\label{thm:RV-pairs-mixed} The analogue statements of Proposition \ref{prop:def-upto-RV}  and Theorem \ref{thm:RV-pairs} hold for $T$. 
\end{theorem}

\begin{proof}
This follows from an easy adaptation of the corresponding proofs.
\end{proof}

The reduction from $\cL_\RVV$ to $\cLgk$ is a little more subtle as it needs a variant of the results in Section \ref{sec:SES} for short exact sequences. Indeed, one needs to consider more generally inverse systems as explained in \cite[Remark 7.6]{hils-mennuni}. For $T=\Th(W(\mathbb{F}_p^\alg))$ the situation is somewhat easier, as the sequence of kernels $W_n[\mathbb{F}_p^\alg]^\times$ for $n\geqslant 0$ of the maps $\RV_n\to\Gamma$ is bi-interpretable with the residue field, and hence no quotients are needed since the residue field eliminates imaginaries. In this special case, one obtains the following result.  

\begin{corollary}
Let $T_k=T_\BP(\cK_k)$ and $T_\Gamma=T_\BP(\cK_\Gamma)$ be completions of $(\ACF_p)_{\SE}$ and of $\mathrm{PRES}_{\SE}$, respectively, where $\cK_k$ and $\cK_\Gamma$ are the corresponding natural classes of \sepa-pairs. Let $\cK$ be the class of $\cL_P$-structures $\cM\in \cK_\Def$ induced by $\cK_k$ and $\cK_\Gamma$. Then, $\cK$ is a natural class satisfying the extension property and $\cK$-beautiful pairs exist. Moreover, $\cK$ has beauty transfer and $T_\BP(\cK)$ is axiomatized by the following conditions on $\cM=(M,P(\cM))$:
\begin{itemize}
    \item $\VF(M)/\VF(P(\cM))$ is $vs$-defectless;
    \item $(M,P(\cM))$ is an elementary pair of models of $T$;
    \item $(\RF(M), \RF(P(\cM)))\models T_k$ and $(\VG(M), \VG(P(\cM)))\models T_\Gamma$.
\end{itemize}

In particular, there are exactly 4 completions of the theory of stably embedded elementary pairs of models of $T=\Th(W(\mathbb F_p^\mathrm{alg}))$, all corresponding to theories of beautiful pairs whose classes have beauty transfer. \qed
\end{corollary}

\subsection{Valued differential fields}

We now briefly discuss a similar context of valued differential fields in which an Ax-Kochen-Ershov principle for beauty applies. Let $\mathrm{VDF}$ be the theory 
of existentially closed valued differential fields $(K,v,\partial)$ of 
residue characteristic 0 satisfying $v(\partial(x))\geq v(x)$ for all $x$.\footnote{This theory is sometimes denoted by $\mathrm{VDF}_\mathcal{EC}$.}  
This theory had first been studied by Scanlon \cite{Sca00}, and it was further investigated by Rideau-Kikuchi \cite{RidVDF}. The theory $\mathrm{VDF}$ is NIP. In $\mathrm{VDF}$ the value group is stably embedded and a pure model of $\DOAG$, and the residue field is stably embedded and a pure model of DCF$_0$, with the derivation induced by $\partial$. As DCF$_0$ is stable with nfcp (see for instance \cite[Section 3]{hru-pillay2000}), (DCF$_0)_\BP$ exists and its class has beauty transfer by Theorem~\ref{thm:poizat} and Proposition~\ref{prop:poizat_equiv}.

\begin{remark}\label{rem:VDF} 
Using the methods from \cite[Section~8]{hils-mennuni} and following the proof strategy for Theorem~\ref{thm:ACVF-Beauty} 
one may show an analogue of Theorem \ref{thm:ACVF-Beauty} for $\mathrm{VDF}$. Indeed, let $\cK_k$ and $\cK_\Gamma$ be natural classes of \sepa-pairs in DCF$_0$ and $\DOAG$, respectively. For ease of presentation, we will assume both classes $\cK_{k}$ and $\cK_\Gamma$ have the extension property and the amalgamation property. Assume moreover that $\cK_k$ and $\cK_\Gamma$ have beauty transfer. Let $\cK$ be the class of \sepa-pairs $\cM\in \cK_\Def$ induced by $\cK_k$ and $\cK_\Gamma$. Then, in the 3-sorted language with sorts $\VF$, $\RV$ and $\VG$, beautiful pairs exist. Moreover, $\cK$ has beauty transfer and $\mathrm{VDF}_\BP(\cK)$ is axiomatized by the following conditions on $\cM=(M,P(\cM))$:
\begin{itemize}
    \item $\VF(M)/\VF(P(\cM))$ is $vs$-defectless;
    \item $(M,P(\cM))$ is an elementary pair of models of $\mathrm{VDF}$;
    \item $(\RF(M), \RF(P(\cM)))\models (\mathrm{DCF}_0)_\BP(\cK_k)$ and $(\VG(M), \VG(P(\cM)))\models \DOAG_\BP(\cK_\Gamma)$.
\end{itemize}   
\end{remark}

\begin{remark}\label{rem:VDF2}
Let $X$ be a definable set in a model of $\mathrm{VDF}$, in sorts $\VF$, $\RV$ and $\VG$. Then the class of definable types concentrating on $X$ in the classes appearing in Remark~\ref{rem:VDF} are strict pro-definable (this uses Theorem \ref{thm:strictness} together with Remark \ref{rem:VDF}). In particular, this holds for the space $\widehat{X}$ of all stably dominated types concentrating on $X$, and for the space $\D{X}$ of all definable types. 

Note that, for $\widehat{X}$, strict pro-definability has already been shown in work by the second author, Kamensky and Rideau-Kikuchi \cite{HKR2016}.
\end{remark}

\subsection{Some open questions}\label{sec:questions} We gather in this section some open questions. 

\medskip

\begin{question}\label{Q1}
Is there a theory~$T$ (with $\mathrm{UDDT}$) such that $\cK_\Def$ does not have the amalgamation property?  
\end{question}

\begin{remark}
Since the first version of this manuscript has been made public, R. Mennuni and the second author have found a theory $T$ in which definable types do not amalgamate (\cite{hils-mennuni2}). Although, in this example the definable types are not uniformly definable, so this only provides a partial answer to Question~\ref{Q1}.
\end{remark}

When $T$ is stable, by Theorem \ref{thm:poizat}, both strict pro-definability and beauty transfer are equivalent to nfcp. This suggests the following question.  

\begin{question}\label{question:nfcp} Is there a theory $T$ such that beautiful pairs exist, $\D{}$ (resp. a natural subclass $\cF$ of definable types) is strict pro-definable, but $\cK_\Def$ (resp. $\cK_\cF$) does not have beauty transfer?
\end{question}

Corollary \ref{cor:NIP-transfer} shows that, assuming $\cK$ has beauty transfer and extension property, NIP is preserved from $T$ to $T_\BP(\cK)$. It seems plausible that, using the methods from \cite{ChHi14}, the same holds for $\mathrm{NTP}_2$. We leave it as a question.  

\begin{question}\label{Q:NTP2-transfer}
Suppose $\cK$ has beauty transfer. Does it hold that if $T$ has $\mathrm{NTP}_2$ so does $T_\BP(\cK)$? If not, does this hold if $\cK$ has the extension property? 
\end{question}

Recall that Theorem \ref{thm:DOAG-complete} shows that $\DOAG_{\SE}$ has only four completions. Moreover, $\RCF_{\SE}$ has only two completions by Remark \ref{rem:omin_dries}. This suggests the following.  

\begin{question}\label{question:amalg-omin} Suppose $T$ is an o-minimal expansion of $\DOAG$. Is every completion of $T_{\SE}$ one of the following: $T_\BP(\cK_\triv)$, $T_\BP(\cK_\bdd)$, $T_\BP(\cK_{\conv})$ or $T_\BP$? 
\end{question}

Let $p$ be a prime and $e\in \mathbb{N}$. Let SCVF$_{p,e}$ be the theory of separably closed non-trivially
valued fields of characteristic $p$ and imperfection degree $e$. In \cite[Corollary~7.7]{HKR2016} it is shown that SCVF$_{p,e}$ is metastable and that for any interpretable set $X$, the space $\widehat{X}$ of stably dominated types concentrating on $X$ is strict pro-definable  (see also \cite{HKR2016}). Here, being stably dominated is equivalent to being orthogonal to $\Gamma$, hence we have a natural class of definable types. We expect that the following question does have a positive answer, which would yield a generalization of this result. 

\begin{question}
Does $(\mathrm{SCVF}_{p,e})_\BP$ exist? If so, does $\cK_\Def$ have beauty transfer? Similarly, denoting by $\cK$ the natural class of \sepa-pairs corresponding to the class of stably dominated types, does $(\mathrm{SCVF}_{p,e})_\BP(\cK)$ exist and does $\cK$ have beauty transfer?
\end{question}

The proof of Theorem~\ref{thm:Beauty-geom-ACVF} relies on generic resolutions, and thus the pair of residue fields needs to be non-trivial. This leads to the following question. 

\begin{question}
Does the analogous statement of Theorem~\ref{thm:Beauty-geom-ACVF} also hold in the cases when the pair of residue fields is assumed to be trivial? 
\end{question}

\bibliographystyle{siam}
\bibliography{biblio}

\end{document}